\documentclass[12pt]{amsart}
\usepackage[utf8]{inputenc}
\usepackage{BAstyle}

\usepackage[margin=1in]{geometry}
\usepackage{comment}

\newcommand{\tw}{\operatorname{tw}}

\title{Counting Number Fields using Multiple Dirichlet Series}
\author{Brandon Alberts and Alina Bucur}

\begin{document}
\allowdisplaybreaks

\begin{abstract}
    We provide a method for counting number fields of fixed Galois group ordered by arbitrary inertial invariants using analytic techniques from the study of multiple Dirichlet series. We prove unconditional results for infinitely many new (concentrated and semiconcentrated) groups that were not approachable by previous methods. Conditional on subconvexity bounds bounds for certain Dirichlet series (e.g. the generalized Lindel\"of hypothesis), we use these techniques to prove the existence of an asymptotic growth rate for $G$-extensions for infinitely many new groups $G$ for which the minimum index elements of $G$ are contained in a union of proper abelian normal subgroups. In particular, our conditional results include all groups with nilpotency class $2$. Additionally, when $G$ is nilpotent our results give a power saving error term.
\end{abstract}

\maketitle

\tableofcontents

\pagebreak

\section{Introduction}\label{sec:introduction}

There has been significant interest in determining the asymptotic behavior of the number of extensions of a fixed number field with given Galois group ordered by various invariants. One could say that the problem goes back to Gauss in the case of counting quadratic fields of bounded discriminant. To explain exactly what we mean, we need to introduce some notation.

Fix a number field $k$ and $G$ a transitive permutation group of degree $n$. We want to study the set $\mathcal{F}_{k}(G)$ of $G$-extensions of $k$, namely the set of field extension $L/k$ of degree $[L:k] = n$ contained in a fixed algebraic closure $\overline{k}$ whose Galois closure $\widetilde{L}$ has Galois group $\Gal(\widetilde{L}/k)\cong G$ which acts on the $n$ embeddings $L\hookrightarrow \widetilde{L}$ according to the transitive permutation action of $G$. The inverse Galois problem asks for which groups $G$ is this set nonempty. However, we are interested in a different issue. When the inverse Galois problem has infinitely many solutions, we want to study how the number of elements in the set $\mathcal{F}_{k}(G)$ grows when we order them by some invariant. 

Perhaps the most natural invariant is the discriminant, so we will study the set
\[
    \mathcal{F}_{\disc,k}(G;X) = \left\{L/k\ G\text{-extension } : |\disc(F/k)| \leq X \right\},
\]
where we use $|\cdot|$ to denote the norm down to $\Q$. One expects a result of the following form.

\begin{conjecture}[Number Field Counting Conjecture]\label{conj:number_field_counting}
    Let $k$ be a number field and $G$ a transitive permutation group of degree $n$. Then there exist positive constants $a,b,c > 0$ depending on $k$ and $G$ such that
    \[
    \#\mathcal{F}_{\disc,k}(G;X)
        \sim c X^{1/a} (\log X)^{b-1}
    \]
    as $X\to \infty$.
\end{conjecture}

This is a conjecture first formulated by Malle in \cite{malle2002, malle2004} based on work by Baily \cite{baily1980}, Davenport--Heilbronn \cite{davenport-heilbronn1971}, Wright \cite{wright1989}, and many others. In particular, Malle conjectured that the constant $a$ in the exponent of $X$ does not depend on the base field $k$, it can be predicted simply in terms of $G$ and its embedding into $S_n.$ Namely, he predicted that $a = a(G)$ to be the minimum index $\min_{g\ne 1} \ind_n(g)$ of the nonidentity elements of $G$, where $\ind_n:S_n\to \Z$ is given by $\ind_n(g) = n - \#\{\text{orbits of }g\}$.

In this paper we develop a new method for proving cases of Conjecture \ref{conj:number_field_counting} for permutation groups that we refer to as semiconcentrated. See \Cref{subsec:semiconcentrated} below for the definition. We define a multiple Dirichlet series associated to the counting problem for these groups and take advantage of Hartogs' phenomenon in several complex variables to get a meromorphic continuation of said multiple Dirichlet series to a region larger than previously known. This allows us to prove \Cref{conj:number_field_counting} in new cases (albeit ineffectively in $b$ and $c$), it allows us to get better power savings in the error terms in many already known cases, and it packages the counting by all invariants into one single multiple Dirichlet series whose analytic behavior encodes the information about the asymptotics for each invariant.

\subsection{Notation}
Since we are going to consider the same abstract group in different representations, we need some notation that distinguishes between these representations, as well as between the different invariants. Thus we will introduce the set

\[
\mathcal F_{\inv, k} (G;X) = \left\{L/k\ G\text{-extension of } : |\inv(L/k)| \le X \right \},
\]

In general for Galois groups $G$ we will use the transitive group label convention from LMFDB; namely each label will have the form $nTd$, where $n$ is the degree of the representation and $d$ is the $T$-number. (See \url{https://www.lmfdb.org/GaloisGroup/Labels} for more details.) Thus the degree of the $G$-extension $L/k$ will be given by the number $n$ in the label $n$T$d$. We will use the abstract group notation (e.g. $C_2$) when a transitive representation has not been specified, or when there is no danger of confusion.

\subsection{Semiconcentrated groups}\label{subsec:semiconcentrated}

\begin{definition}\label{def:semiconcentrated}
    A permutation group $G = nTd$ is said to be {\em semiconcentrated} (for the discriminant ordering) in a set of proper normal subgroups $T_1,T_2,\dots,T_r\normal G$ if the minimum index elements of $G$ are contained in the union of $T_1,T_2,\dots,T_r$, i.e.
    \[
    \{ g\in G : \ind_n(g) = a(G) \} \subseteq \bigcup_{j=1}^r T_j,
    \]
    and we say that $G$ is semiconcentrated if this holds for the collection of all proper normal subgroups of $G$.
    Moreover, we say that $G$ is {\em properly semiconcentrated} if the minimum index elements are not all contained in one single proper normal subgroup. 
\end{definition}
A number of important groups are semiconcentrated. In particular, every noncyclic nilpotent group is a union of its proper normal subgroups, and therefore semiconcentrated.

This is a weakening of the notion of a \emph{concentrated} group defined in the recent work of Alberts, Lemke Oliver, Wang, and Wood \cite{ALOWW}: a permutation group $G$ is called concentrated if its minimum index elements generate a proper normal subgroup. More specifically, $G$ is concentrated in the normal subgroup $T\normal G$ if
\[
    \{ g\in G : \ind_n(g) = a(G) \} \subseteq T.
\]
Every concentrated group is semiconcentrated, but the converse need not be true. A properly semiconcentrated group is precisely a group which is semiconcentrated but not concentrated.

Nearly all the known cases of \Cref{conj:number_field_counting} are concentrated groups, with a number of past results in this direction now falling under the ``inductive counting method'' described in \cite{ALOWW} for concentrated groups. More specifically, all the concentrated groups for which Conjecture \ref{conj:number_field_counting} is known are specifically concentrated in a proper normal subgroup which is either abelian, or one of $S_3^m$, $S_4$, or $S_5$.  Meanwhile, very little is known about \Cref{conj:number_field_counting} for groups which are not concentrated. The only nonconcentrated groups for which \Cref{conj:number_field_counting} has been proven are elementary abelian groups $C_p^r$ \cite{wright1989}, $S_3$ in degree $3$ \cite{datskovsky-wright1988}, $S_3$ in degree $6$ \cite{bhargava-wood2007}, $S_4$ in degree $4$ \cite{bhargava-shankar-wang2015}, $S_5$ in degree $5$ \cite{bhargava-shankar-wang2015}, $D_4$ in degree $8$ \cite{shankar-varma2025octic}, and most recently $D_6$ in degree $12$ \cite{koymans-lemke-oliver-sofos-thorne2025}. In fact, $D_4$ in degree $8$ and $D_6$ in degree $12$ are the only nonabelian properly semiconcentrated groups for which Conjecture \ref{conj:number_field_counting} was previously known. Our results represent a significant breakthrough: we prove \Cref{conj:number_field_counting} for infinitely many new groups and illustrate the power of the method by the detailed computation (all the way to the power saving error term) for one properly semiconcentrated group ($16T11$). Conditional on certain analytic assumptions, we are able to do this for infinitely many properly semiconcentrated groups.

\subsection{Multiple Dirichlet series: idea of the method}\label{subsec:Idea_of_method}

Our method is of a similar inductive flavor to that of \cite{ALOWW}. When $G$ is semiconcentrated in a set of proper normal subgroups $T_1,T_2,\dots,T_r$ we combine ``uniform'' information about relative $T_j$-extensions of number fields with fairly weak information on the number of $G/T_j$-extensions for each $j=1,2,...,r$. When $G$ is concentrated in a single normal subgroup, $T$, \cite[Theorem 2.1]{ALOWW} describes how to convert these counting hypotheses to an asymptotic using a dominated convergence argument. In our setting, however, we need to do quite a bit more in order to stitch together the information from multiple different normal subgroups.

This is where multiple Dirichlet series come into the picture. Rather than working directly with counting functions as in \cite{ALOWW}, we instead work with generating Dirichlet series of several complex variables. We will define a generating multiple Dirichlet series $D_k(G;\underline{s})$ for $G$-extensions as a function of a certain tuple of complex variables $\underline{s}$. We ask for three ingredients.
\begin{enumerate}
    \item Firstly, we need a meromorphic continuation for the generating multiple Dirichlet series of relative $T_j$-extensions (which implies the counting results used in \cite{ALOWW}),
    \item Secondly, we need upper bounds for this series in vertical strips (which plays the role of ``uniformity'' for the Dirichlet series),
    \item Finally, we need a region of absolute convergence for the generating multiple Dirichlet series for $G/T_j$-extensions (which plays the role of upper bounds for the number of $G/T_j$-extensions used in \cite{ALOWW}).
\end{enumerate}
This results in a meromorphic continuation for $D_k(G;\underline{s})$ in the direction of \emph{only some} of the variables in the tuple $\underline{s}$, exactly which variables depends on the normal subgroup $T_j$. Call this region of meromorphicity $\Omega_j$. The meromorphic continuation to $\Omega_j$ is essentially a ``complex analytic version'' of the method described in \cite{ALOWW}.

Once we show that $D_k(G,\underline{s})$ is meromorphic on $\bigcup \Omega_j$, we can take advantage of the multivariable setting: Hartogs phenomenon for multivariable holomorphic functions states that not every domain in multidimensional complex space gets to be a domain of holomorphy. In our context, Hartogs phenomenon in the form of Bochner's tube theorem will give us an extra meromorphic continuation to the convex hull of $\bigcup \Omega_j$ practically for free. Taking the convex hull acts as a way of stitching together the results for each normal subgroup $T_1,T_2,\dots,T_r$. Specializing $D_k(G;\underline{s})$ to a specific complex line gives the usual single variable generating Dirichlet series which inherits a meromorphic continuations from the multivariable series. The specific choice of line is dictated by the invariant one uses to order the set of $G$-extensions (discriminant, conductor, product of ramified primes, etc).  If the point expected to give the right-most pole of this one-variable specialization happens to be inside the convex hull (that is, inside the region of meromorphic continuation), then the desired asymptotic results follow from a standard Tauberian argument.

We execute this strategy when $G$ is semiconcentrated in a set of proper abelian normal subgroups, leaning on the twisted number field counting program of study to control the relative $T_j$-extensions similar to \cite{ALOWW}. We remark that the structure of such transitive groups is fairly well understood.
\begin{lemma}\label{lem:semiconcentrated_in_abelian}
    Suppose $G$ is a transitive group that is semiconcentrated in a set of proper abelian normal subgroups. Then one of the following is true:
    \begin{enumerate}[$($a$)$]
        \item $G$ is nilpotent, or
        \item $G$ is concentrated in a proper nilpotent normal subgroup; equivalently $G$ is concentrated in its Fitting subgroup.
    \end{enumerate}
\end{lemma}
This follows from purely group theoretic facts: Fitting's theorem states that the internal product of normal nilpotent subgroups is again nilpotent, and the Fitting subgroup is then defined to be the maximal nilpotent normal subgroup. We note that the converse of course is not true: there exist groups which are concentrated in their Fitting subgroup but not semiconcentrated in abelian normal subgroups, such as the dihedral group $D_8$ of order $16$ in degree $8$ (transitive group label $8T6$). Informed by \Cref{lem:semiconcentrated_in_abelian}, we will separately highlight results for each of these families.

The analysis that we make use of is rather delicate. There are three primary sticking points:
\begin{itemize}
    \item The bounds in vertical strips. When $T_j$ is abelian, we will show that these can be given by (sub)convexity bounds of certain Hecke $L$-functions.
    \item Computing convex hulls in more than two dimensions is a very involved process, and determining when the convex hull is ``sufficiently large'' can be difficult.
    \item Strictly speaking, Hartogs' phenomenon applies to holomorphic functions and not meromorphic functions.
\end{itemize}
The first two points can be checked for an individual group $G$ if you have enough patience (or a good computer program), but is significantly harder in infinite families. We include several examples of small groups for which we work out these details unconditionally to prove \Cref{conj:number_field_counting}. However, for most properly semiconcentrated groups the best known subconvexity bound for Hecke $L$-functions is not enough to show that the convex hull is ``sufficiently large''. These two issues can be overcome by assuming the Lindel\"of hypothesis in the conductor aspect (or even something slightly weaker), from which we are able to give conditional proofs of \Cref{conj:number_field_counting} for infinitely many new groups, including infinitely many properly semiconcentrated groups.

The third point can be accounted for by multiplying $D_k(G;\underline{s})$ by a polynomial that cancels out the (only finitely many) poles in the region of interest. However, by doing so we lose some control over the poles within the convex hull: we know their location, but not their order. As a consequence, all the cases we prove for \Cref{conj:number_field_counting} are ineffective in the constants $b\in \Z_{>0}$ and $c\in \R_{>0}$. It may be possible to determine $b$ and $c$ by a closer study of the residues of $D_k(G;\underline{s})$ before taking the convex hull, as similar things have been done in the study of moments of $L$-functions.

\subsection{Unconditional results for nilpotent groups}\label{subsec:intro_undconditional_nilpotent}

When $G$ is small, existing hybrid subconvexity bounds are sufficient to prove a meromorphic continuation to a ``sufficiently large'' convex hull. Moreover, we always prove meromorphic continuations to open orthants.  If we have vertical bounds in the $t$-aspect, these can automatically be converted to power saving error bounds using an appropriate Tauberian theorem.

Our main unconditional result for nilpotent groups is \Cref{thm:general_unconditional_nilpotent}, which gives an explicit cone of meromorphicity for $D_k(G;\underline{s})$. We need to build up a fair amount of notation before we can state this result, so for the introduction state several corollaries for specific groups to give the reader a good sense of our results.

First, we explicitly compute the largest cone of meromorphicity for $D_\Q(G;\underline{s})$ for a couple of small groups, along with explicit bounds for the growth in vertical strips for two small groups, to prove the following results:

\renewcommand{\labelenumi}{{{$($\alph{enumi}$)$}}}
\begin{theorem}\label{thm:D4}
Consider the dihedral group $D_4$ of order $8$.
\begin{enumerate}
    \item For the two transitive representations of $D_4$ we have the asymptotic expansions
    \begin{align*}
        \#\mathcal{F}_{\disc,\Q}(4T3;X) &= c_4X + O_{\epsilon}(X^{15/22+\epsilon})\\
        \#\mathcal{F}_{\disc,\Q}(8T4;X) &= X^{1/4}P_8(\log X) + O_{\epsilon}(X^{61/274+\epsilon}),
    \end{align*}
    where $c_4$ is the leading coefficient from \cite{cohen-diaz-y-diaz-olivier2002}, and $P_8$ is a degree $2$ polynomial with leading coefficient given by \cite{shankar-varma2025octic} but whose other coefficients are undetermined.
\item Ordering $D_4$-extensions by conductor we have 
    \begin{align*}
        \#\mathcal{F}_{{\rm cond},\Q}(D_4;X) &= XP_{\rm cond}(\log X) + O_{\epsilon}(X^{39/44+\epsilon})
    \end{align*}
    where  $P_{\rm cond}$ is a degree $1$ polynomial with leading coefficient given by \cite{altug-shankar-varma-wilson2021} and constant given by \cite{friedrichsen2021secondary}.
\end{enumerate}
\end{theorem}

\begin{theorem}\label{thm:Q8sxC2}
Consider the semidirect product $Q_8\rtimes C_2$, where $Q_8$ is the quaternion group and the action sends $i\mapsto -i$, $j\mapsto -j$, and $k\mapsto k$. For the two transitive representations of this group we have the asymptotic expansions
    \begin{align*}
        \#\mathcal{F}_{\disc,\Q}(8T11;X) &= c_8 X^{1/2} + O_{\epsilon}(X^{19/55+\epsilon})\\
        \#\mathcal{F}_{\disc,\Q}(16T11;X) &= X^{1/8}P_{16}(\log X) + O_{\epsilon}(X^{97/800+\epsilon}),
    \end{align*}
    where $c_8$ is the leading coefficient from \cite{ALOWW}, and $P_{16}$ is a polynomial of degree $1\le \deg P_{16} \le 3 = b(\Q,16T11)-1$, but is otherwise undetermined.
\end{theorem}

The group $16T11$ is properly semiconcentrated, so this is a new case of Conjecture \ref{conj:number_field_counting}. Namely, $16T11$ is concentrated in three different proper abelian normal subgroups, and is generated by its minimum index elements.

As remarked in \Cref{subsec:Idea_of_method} above, our method only proves the existence of a main term while not producing the precise power of $\log X$ because we do not produce the order or residues of the corresponding pole. This is why the polynomial in the main term  for $16T11$ is unspecified. The asymptotics in the other four cases of the above results were previously known by work in \cite{cohen-diaz-y-diaz-olivier2002,altug-shankar-varma-wilson2021,friedrichsen2021secondary,shankar-varma2025octic,ALOWW}, allowing us to cite those results for the shape of the main term.

With the exception of $4T3$, these are now the best known error bounds for each of these groups. The power savings we can extract from the $t$-aspect of the vertical bounds tends to get much worse and more tedious to calculate as $G$ becomes large, as they are a complicated mixture of $t$-aspect subconvexity bounds and convex hull computations. We have taken care to determine the best possible power savings in \Cref{thm:D4} and \Cref{thm:Q8sxC2} in order to demonstrate the computations involved, but for the remaining results of the paper we opt to only state when a power saving error bound exists and not compute an explicit such bound.

The previously known counting results for $4T3$ and $8T11$ have been extended to arbitrary base fields \cite{bucur-florea-serrano-lopez-varma2022,ALOWW}. We do the same for $8T4$ and $16T11$ below, along with several new nilpotent groups. 

\begin{theorem}\label{thm:16T11_over_k}
    Let $k$ be a number field and consider its degree 16 extensions with Galois group $Q_8\rtimes C_2$. Here  $Q_8$ denotes again the quaternion group and the action sends $i\mapsto -i$, $j\mapsto -j$, and $k\mapsto k$. There exists a constant $\delta_{k,16T11}>0$ for which
    \begin{align*}
        \#\mathcal{F}_{\disc,k}(16T11;X) &= X^{1/8}P_{k,16T11}(\log X) + O(X^{1/8-\delta_{k,16T11}}),
    \end{align*}
    where $P_{k,16T11}$ is a polynomial of degree $1\le \deg P_{k,16T11} \le 3 = b(k,16T11)-1$, but is otherwise undetermined.
\end{theorem}

\begin{theorem}\label{thm:intro_nilpotent_groups_with_shortcut}    
    Let $k$ be a number field and $G$ be any of the nilpotent transitive groups 
    $8T4$, 
    $8T28$, 
    $8T29$, 
    $16T209$, 
    $16T223$, 
    $16T227$, 
    $16T243$, 
    $16T247$, 
    $16T271$, 
    $16T280$, 
    $16T283$, 
    $16T345$, 
    $16T388$, 
    $16T401$, 
    $16T479$, 
    $16T480$, 
    $16T481$, 
    $16T492$, 
    $16T505$, 
    $16T517$, 
    $16T519$, 
    $16T527$, 
    $16T625$, 
    $16T630$, 
    $16T1555$, 
    $16T1566$, 
    $16T1609$, 
    $16T1700$, 
    $16T1702$, and
    $16T1728$.
    
    Then $G$ is concentrated in a proper nilpotent normal subgroup (except for $8T4$, which is properly semiconcentrated), but not concentrated in an abelian normal subgroup, and there exists a constant $\delta_{k,G} > 0$ such that
    \[
        \#\mathcal{F}_{\disc,k}(G;X) = X^{1/a(G)}P_{k,G}(\log X) + O\left(X^{1/a(G) - \delta_{k,G}}\right),
    \]
    where $P_{k,G}$ is a polynomial of degree $1\le \deg P_{k,G} \le 2 = b(k,G) - 1$.
\end{theorem}

We remark that optimal regions of meromorphicity become very cumbersome to compute for even moderately sized groups. At the risk of some weaker power savings, we will employ several shortcuts to determine when our results are strong enough to produce an asymptotic. One such shortcut is to reduce the problem to computing a two-dimensional convex hull, even though that only gives a subregion of the full region of meromorphicity (\Cref{subsec:2dim}). Using some short \verb|Magma| code \cite{MAGMA,codeurl}, we can check these simplified conditions for several nilpotent groups of moderate degree.

By contrast, for the group $16T11$, which is semiconcentrated in {\em three} proper normal abelian subgroups, we compute the three-dimensional convex hull in \Cref{sec:16T11}. There are other groups that are concentrated in three different abelian normal subgroups (e.g. $16T146$, $16T147$), and it should be possible to examine these cases using the same method.

\subsection{Unconditional results for concentrated groups}

Our methods apply equally well to the twisted number field counting problem as they do to \Cref{conj:number_field_counting}. By combining our methods with the inductive counting recipe in \cite{ALOWW}, we can prove new results for groups concentrated in proper {\em nonabelian} normal subgroups. The necessary hypotheses on class groups torsion are a bit more complicated to state, but largely of a similar flavor to those in \cite{ALOWW}.

Our main unconditional result for concentrated groups is \Cref{thm:unconditional_concentrated}, which we leave to be stated in the main body of the paper. We prove two families of unconditional examples of a similar flavor to those in the introduction to \cite{ALOWW} to demonstrate this type of result.

The first result includes several new infinite families of wreath products, generalizing the results on wreath products by cyclic groups in \cite[Corollary 1.3]{ALOWW} and wreath products by $S_3$ in \cite[Corollary 1.6]{ALOWW}.

\begin{theorem}\label{thm:wreath_products}
    Let $k$ be a number field and $G=N\wr B$ be the wreath product of the transitive groups $N$ in degree $n$ and $B$ in degree $m$. Additionally, suppose one of the following holds:
    \begin{enumerate}[(a)]
        \item $N=8T4$ and there exists some $\delta > 0$ such that
        \[
            0 < \#\mathcal{F}_{\disc,k}(B;X) \ll X^{1+\frac{1}{m[k:\Q]} - \delta}
        \]
        \item $N=8T28$, 
        $8T29$, 
        $16T209$, 
        $16T223$, 
        $16T227$, 
        $16T243$, 
        $16T247$, 
        $16T271$, 
        $16T280$, 
        $16T283$, 
        $16T345$, 
        $16T388$, 
        $16T401$, 
        $16T479$, 
        $16T480$, 
        $16T481$, 
        $16T492$, 
        $16T505$, 
        $16T517$, 
        $16T519$, 
        $16T527$, 
        $16T625$, or
        $16T630$ and there exists some $\delta > 0$ such that
        \[
            0 < \#\mathcal{F}_{\disc,k}(B;X) \ll X^{2+\frac{2}{m[k:\Q]} - \delta}
        \]
        \item $N = 16T1555,16T1566,16T1609,16T1700,16T1702,$ or $16T1728$ and there exists some $\delta > 0$ such that
        \[
            0 < \#\mathcal{F}_{\disc,k}(B;X) \ll X^{4+\frac{4}{m[k:\Q]} - \delta}.
        \]
    \end{enumerate}
    Then there exist ineffective constants $c\in \R_{>0}$ and $b\in \Z_{>0}$ with $2\le b \le 3$ such that
    \[
        \#\mathcal{F}_{\disc,k}(G;X) \sim cX^{1/a(G)}(\log X)^{b-1}.
    \]
\end{theorem}

We note that $N\wr B$ is concentrated in $N^m$, but for each of the $N$ above it is not concentrated in an abelian normal subgroup, and thus outside the scope of \cite{ALOWW}. Since the (relatively mild) upper bounds we require in the theorem above are known for infinitely many groups, this result produces infinitely many new asymptotics for each choice of $N$.

We prove a similar result for direct products.

\begin{theorem}\label{thm:direct_products}
    Let $k$ be a number field and $G=N\times B$ be the wreath product of the transitive groups $N$ in degree $n$ and $B$ in degree $m$. Additionally, suppose that $N$ is one of the thirty transitive nilpotent groups in \Cref{thm:intro_nilpotent_groups_with_shortcut} and that there exists some $\delta > 0$ for which
    \[
        0 < \#\mathcal{F}_{\disc,k}(B;X) \ll X^{\frac{n}{ma(N)} - \delta}.
    \]
    Then there exist ineffective constants $c\in \R_{>0}$ and $b\in \Z_{>0}$ with $2\le b \le 3$ such that
    \[
        \#\mathcal{F}_{\disc,k}(G;X) \sim cX^{1/a(G)}(\log X)^{b-1}.
    \]
\end{theorem}

In particular, if $B$ is also nilpotent it is known that $\#\mathcal{F}_{\disc,k}(B;X) \ll X^{1/a(B)+\epsilon}$ (see, for example, \cite[Corollary 1.8]{alberts2020}). Thus, the conditions of \Cref{thm:direct_products} hold for nilpotent $B$ if $a(B)/m > a(N)/n$. This includes several nice families:
\begin{itemize}
    \item \Cref{conj:number_field_counting} holds for $8T4\times B$ for $B$ any nilpotent group of odd order in the regular representation.
    \item \Cref{conj:number_field_counting} holds for $N\times B$ for $N\ne 8T4$ any of the remaining transitive nilpotent groups listed in \Cref{thm:intro_nilpotent_groups_with_shortcut} and for $B$ any nilpotent group in the regular representation.
\end{itemize}
Among groups of small degree, $24T15$, $24T351$, and $24T352$ are given by direct products $8T4\times C_3$, $8T28\times C_3$, and $8T29\times C_3$ respectively, so that \Cref{conj:number_field_counting} follows for these groups.

\subsection{Conditional results for nilpotent groups}

Now that we established that our method can be used to access new groups, it would be interesting to examine its limits. Namely, we want to explore what kind of asymptotic one can get if one assumes the best possible subconvexity bounds for the growth of $L$-functions in the critical strip.  In order to help keep our theorem statements clean and short, we define a shorthand for the desired hypothesis for Hecke $L$-functions, which is a hybrid Lindel\"of hypothesis type bound in both the imaginary part and the conductor aspects. We will only need this hypothesis for a particular family of Hecke $L$-functions depending on $k$ and $G$.

\begin{definition}
    For a number field $k$ and a group finite transitive group $G$, let $\mathcal{L}(k,G)$ denote the set of Hecke $L$-functions $L_K(s,\chi)$ for which the following conditions hold.
    \begin{enumerate}[(a)]
        \item $[K:k]$ divides the order of  $G$.
        \item There exists a Galois extension $L/k$ with $\Gal(L/k)\cong G$ and an integer $d\mid \exp(G)$ for which $K\subseteq L(\zeta_d)$.
        \item The character $\chi$ is unramified away from the primes dividing $|G|\infty$.
    \end{enumerate}
\end{definition}

\begin{hypothesis}
    Fix a number field $k$ and a group finite transitive group $G$. Let $\alpha,\beta, \gamma \ge 0$. We say the property $H(\gamma,\alpha,\beta; k, G)$ holds if, for every $\epsilon > 0$ and for every L-function in $\mathcal{L}(k,G)$, the hybrid bound
    \[
        |L_K(\sigma+it,\chi)| \ll_{\epsilon} q(\chi)^{\alpha+\epsilon}(1+|t|)^{\beta+\epsilon}
    \]
    holds for each $\sigma > \gamma$. Here $q(\chi)$ denotes the conductor of $L_K(s,\chi).$ The Lindel\"of hypothesis for each $L$-function in this family is $H(1/2,0,0;k,G)$.

    Further, we say $H(\gamma,\alpha,*; k, G)$ holds if, for every $\epsilon > 0$, the hybrid bound
    \[
        |L_K(\sigma + it,\chi)| \ll_{t,\epsilon} q(\chi)^{\alpha+\epsilon}
    \]
    holds for each $\sigma > \gamma$ and for every $ L_K(s,\chi) \in \mathcal{L}(k,G)$, where the implied constant depends only on $t$ and $\epsilon$.
\end{hypothesis}

Recall that our main unconditional result for nilpotent groups is \Cref{thm:general_unconditional_nilpotent}. We show the necessary hypotheses follow from $H(\gamma,0,*;k,G)$, so that we can conclude the following conditional result.

\begin{theorem}\label{thm:intro_nilpotent}
    Let $k$ be a number field and $G$ a nilpotent transitive group which is semiconcentrated in some proper abelian normal subgroups.
    Suppose there exists a nonnegative real number $\gamma < 1$ such that $k$ and $G$ satisfy the hypothesis $H(\gamma,0,*; k, G)$. 
    \begin{enumerate}
        \item In this case, there exist ineffective constants $c \in \R_{>0}$ and $b \in \Z_{>0}$ for which
        \[
            \#\mathcal{F}_{\disc,k}(G;X) \sim cX^{1/a(G)}(\log X)^{b-1},
        \]
        in particular satisfying Conjecture \ref{conj:number_field_counting}.
        \item  
            Moreover, if there exists $\beta \ge 0$ such that the property $H(\gamma,0,\beta; k, G)$ holds, then there exists a nonzero ineffective polynomial $P=P_{k,G}$ of degree $b-1$ and an effective constant $\delta=\delta_{k,G}(\beta) > 0$ for which
        \[
            \#\mathcal{F}_{\disc,k}(G;X) = X^{1/a(G)}P(\log X) + O_{\epsilon}\left(X^{1/a(G) - \delta + \epsilon}\right),
        \]
        satisfying \Cref{conj:number_field_counting} with a power saving error bound.
    \end{enumerate}
\end{theorem}

We note that our assumption $H(\gamma,0,*;k,G)$ for some $0\leq \gamma<1$ is much weaker than the Lindel\"of hypothesis in the conductor aspect, which is equivalent to $H(1/2,0,*;k,G)$. 

There are infinitely many properly semiconcentrated groups in this family. In particular, any group of nilpotency class $2$ (i.e. a central extension of an abelian group) is the union of its abelian normal subgroups, and so belongs to this family. It is widely believed that Conjecture \ref{conj:number_field_counting} for nilpotency class $2$ groups should be accessible with existing methods; this intuition is supported by several individual results closely related to counting such number fields, including \cite{fouvry-kluners2007,alberts-klys2020,fouvry-koymans2021nonicheisenberg}. Despite this evidence, \Cref{thm:intro_nilpotent} is the first general statement for asymptotics of this family.

There is nothing special about the discriminant for our methods, as demonstrated by the conductor case of \Cref{thm:D4}. Given a class function ${\rm wt}:G \, \setminus \{1\} \to \R_{>0}$, define the corresponding inertial invariant
\[
    \inv(F/k) := \prod_{c\subset G}\prod_{\pk, I_p(F/k)\sim c} |\pk|^{{\rm wt}(c)},
\]
where the outer product is over nontrivial conjugacy classes $c\subset G$ and we write $I_\pk(F/k)\sim c$ to mean that $I_\pk(F/k)$ is generated by an element of $c$. (See \Cref{prop:specialize} below for a slightly more general notion that replaces conjugacy classes with ramification types). A standard example is the product of ramified primes ordering, which comes from taking ${\rm wt}(g) = 1$ identically.

The minimum weight is then $a_\inv(G) := \min_{g\ne 1} {\rm wt}(g)$. If the minimum weight elements are contained in the union abelian normal subgroups $T_1,T_2,\dots,T_r\normal G$, then our methods for this ordering that are just as strong as \Cref{thm:intro_nilpotent} above. Indeed, our general unconditional result \Cref{thm:general_unconditional_nilpotent} is proven for arbitrary inertial invariants and implies \Cref{thm:intro_nilpotent}.

The case of nilpotency class $2$ groups is particularly interesting in this regard: if $G$ has nilpotency class $2$, then $G$ is covered by its abelian normal subgroups. This leads to a very general result.

\begin{theorem}\label{thm:nilpotent_general_invariant}
    Let $k$ be a number field, $G$ a nilpotent group with nilpotency class $2$, and $\inv$ an inertial invariant with minimal weight $a_{\inv}(G)$. Suppose there exists a nonnegative real number $\gamma < 1$ such that $k$ and $G$ satisfy the hypothesis $H(\gamma,0,*; k, G)$. 
    \begin{enumerate}
        \item In this case, there exist ineffective constants $c \in \R_{>0}$ and $b \in \Z_{>0}$ for which
        \[
            \#\mathcal{F}_{\inv,k}(G;X) \sim cX^{1/a_{\inv}(G)}(\log X)^{b-1},
        \]
        in particular satisfying Conjecture \ref{conj:number_field_counting}.
        \item  
            Moreover, if there exists $\beta \ge 0$ such that the property $H(\gamma,0,\beta; k, G)$ holds, then there exists a nonzero ineffective polynomial $P=P_{\inv,k,G}$ of degree $b-1$ and an effective constant $\delta=\delta_{\inv,k,G}(\beta) > 0$ for which
        \[
            \#\mathcal{F}_{\inv,k}(G;X) = X^{1/a(G)}P(\log X) + O_{\epsilon}\left(X^{1/a_{\inv}(G) - \delta + \epsilon}\right),
        \]
        satisfying \Cref{conj:number_field_counting} with a power saving error bound.
    \end{enumerate}
\end{theorem}

Previously, counting results for all inertial invariants were known only for abelian groups (this is implicit in \cite{wright1989,wood2010} and explicit in \cite{alberts-odorney2021,darda2023torsors}) and $S_3$ with $k=\Q$ \cite{shankar-thorne2024generalS3}, although such a result for $D_4$ with $k=\Q$ can be reasonably inferred from \cite{hansen-zanoli2025multiD4}.

\subsection{Conditional results for concentrated groups}

For the second family of Lemma \ref{lem:semiconcentrated_in_abelian}, we combine our method with that of \cite{ALOWW} for concentrated groups. We will use the same the notation as \cite{ALOWW} for the set of surjections from the absolute Galois group of the field $k$ to our fixed group $G$, namely 
\[
    \Sur(G_k,G;X) = \{\pi\in \Sur(G_k,G) : |\disc(\pi)|\le X\}.
\]

Once again we show that the hypotheses of our main unconditional result \Cref{thm:unconditional_concentrated} follow from $H(\gamma,0,*;k,G)$, allowing us to prove the following conditional result.

\begin{theorem}\label{thm:main_concentrated}
    Let $k$ be a number field and $G$ a transitive group which is semiconcentrated in the proper abelian normal subgroups $T_1,T_2,...,T_r$.

    Suppose there exists at least one $G$-extension of $k$ and that there exists a nilpotent normal subgroup $N\normal G$ which contains each of the subgroups $T_1,T_2,\dots T_r$. Denote by  $q:G\to G/N$ the quotient map and assume there exists $\theta \ge 0$ such that
    \begin{align}\label{eq:sum_h1ur}
        \sum_{\pi\in q_*\Sur(G_k,G;X)} \max_{1\leq j \leq r} h_{ur}^1(k,N,T_j,\pi) \ll_{n,k} X^{\theta},
    \end{align}
    where $h^1_{ur}(k,N,T_j,\pi)$ is given by Definition \ref{def:h1ur}.

    If there exists $\gamma < 1$ such that $H(\gamma,0,*;k,G)$ is satisfied, then the following hold.
    \begin{enumerate}[(i)]
        \item If $\theta < 1/a(N)$ then there exist ineffective constants $c\in \R_{>0}$ and $b\in \Z_{>0}$ such that
        \[
            \#\mathcal{F}_{\disc,k}(G;X) \sim cX^{1/a(N)}(\log X)^{b-1}.
        \]
        \item If $\theta \ge 1/a(N)$ then
        \[
            \#\mathcal{F}_{\disc,k}(G;X) \ll_{|G|,k} X^{\theta + \epsilon}.
        \]
    \end{enumerate}
\end{theorem}

The statement of \Cref{thm:main_concentrated} is extremely similar to that of \cite[Theorem 1.11]{ALOWW}. The primary difference is that the result from \cite{ALOWW} holds only when $N$ is abelian, and in our result the quantity $\max_j h^1_{ur}(k,N,T_j,\pi)$ appears instead of $H^1_{ur}(k,N(\pi))$. The former is a slightly more complicated cohomological object, but is closely related to $H^1_{ur}(k,M)$ for certain Galois modules $M$ and is therefore also controlled by certain torsion in the class groups of some extensions of $k$. See \Cref{sec:bounds} for more details. 

\begin{remark}
    We expect that \Cref{thm:main_concentrated}(i), \Cref{thm:wreath_products}, \Cref{thm:direct_products} can be upgraded to a power saving error bound if we incorporate a $t$-aspect bound (such as $H(\gamma,0,\beta; k, G)$) analogous to the error bounds in \Cref{thm:intro_nilpotent}; however, we appeal to \cite[Theorem 2.1]{ALOWW} in our proof which does not include a power saving error term. We expect that that the methods in \cite{ALOWW} can be improved in this way, which would also lead to power savings in \cite[Theorem 1.11]{ALOWW} and the other results in \cite{ALOWW} for groups concentrated in proper abelian normal subgroups.
\end{remark}

\Cref{thm:main_concentrated} is proven by verifying the hypotheses of \cite[Theorem 2.1]{ALOWW}. The first hypothesis is a ``precise counting of the fibers'', which requires proving the twisted number field counting conjecture \cite[Conjecture 2]{ALOWW} for $N\normal G$. 

\begin{conjecture}[Twisted Number Field Counting Conjecture]\label{conj:twisted_number_field_counting}
    Let $k$ be a number field, $G$ a transitive permutation group of degree $n$, and $N\normal G$ with quotient map $q:G\to G/N$ and pushforward $q_*:\Sur(G_k,G)\to \Sur(G_k,G/N)$. For each $\pi \in q_*\Sur(G_k,G)$, there exist positive constants $a,b,c > 0$ depending on $k$, $G$, $N$, and $\pi$ such that
    \[
    \#\{\psi\in q_*^{-1}(\pi) : |\disc(\psi)|\le X\}
        \sim c X^{1/a} (\log X)^{b-1}
    \]
    as $X\to \infty$.
\end{conjecture}

This conjecture and the corresponding counting problem were first introduced in \cite{alberts2021}. This is referred to as a ``twisted'' counting conjecture because the set being counted is in bijection with crossed homomorphisms valued in the group $N$ carrying a certain Galois action. We prove this conjecture for several new cases.

\begin{theorem}\label{thm:nilpotent_fibers_main}
    Let $k$ be a number field, $G = nTd$ a transitive group in degree $n$, and $N\normal G$ a nilpotent normal subgroup with quotient map $q:G\to G/N$. Suppose that
    \begin{enumerate}
        \item there exists at least one $G$-extension of $k$,
        \item there exist abelian normal subgroups $T_1,T_2,\dots T_r\normal G$ for which
    \[
        \{g\in N : \ind_n(g) = a(N)\} \subseteq \bigcup_{j=1}^r T_j \subseteq N,
    \]
    \item there exists some $\gamma < 1$ such that  $H(\gamma,0,*; k, G)$ holds.
    \end{enumerate}

    Then for each $\pi \in q_*\Sur(G_k,G)$ there exist ineffective constants $c \in \R_{>0}$ and $b \in \Z_{>0}$ for which
    \[
        \#\{\psi\in q_*^{-1}(\pi) : |\disc(\psi)|\le X\} \sim c X^{1/a(N)}(\log X)^{b-1}.
    \]
    In particular, \Cref{conj:twisted_number_field_counting} is satisfied.
\end{theorem}

See \Cref{thm:unconditional_nilpotent_fibers} for an unconditional version of this result.

The second hypothesis of \cite[Theorem 2.1]{ALOWW} is ``uniform upper bounds on the fibers", which we prove for the same normal subgroups $N$.

\begin{theorem}\label{thm:nilpotent_fibers_uniform}
    Let $k$, $G$, $N$, $q$, and $T_1,T_2,\dots,T_r$ satisfy the conditions of \Cref{thm:nilpotent_fibers_main}.

    Then for each $\pi \in q_*\Sur(G_k,G)$ we have
    \[
        \#\{\psi\in q_*^{-1}(\pi) : |\disc(\psi)|\le X\}  = O_{n,[k:\Q],\epsilon}\left(\frac{\max_j h^1_{ur}(k,N,T_j,\pi)}{(q_*\disc_G(\pi))^{1/a(N) - \epsilon}}X^{1/a(N)}(\log X)^{b-1}\right),
    \]
    where $b \in \Z_{>0}$ is the same ineffective constant from \Cref{thm:nilpotent_fibers_main}, $q_*\disc_G$ is the pushforward discriminant, and $h^1_{ur}(k,N,T_j,\pi)$ is the quantity defined in \Cref{def:h1ur}.
\end{theorem}

See \Cref{thm:unconditional_nilpotent_fibers_uniform} for an unconditional version of this result. We remark that \Cref{thm:nilpotent_fibers_uniform} is effectively the same strength as the corresponding result for $N$ abelian proven in \cite[Theorem 6.1]{ALOWW}. The only difference is again the fact that $H^1_{ur}(k,N(\pi))$ has been replaced by $h^1_{ur}(k,N,T_j,\pi)$.

The third and final hypothesis of \cite[Theorem 2.1]{ALOWW} is the ``criterion for convergence'', which follows from the condition \eqref{eq:sum_h1ur}. As in \cite{ALOWW}, this relies only on the (average) size of certain class group torsion and mild bounds for the number of $G/N$-extensions.

\subsection{Comparison to previously known results}

The number field counting conjecture (\Cref{conj:number_field_counting}) is known in many cases, although exhaustive lists have become difficult to write down. This is particularly evident in \cite{ALOWW}, whose inductive results prove Conjecture \ref{conj:number_field_counting} for a large number of new, infinite families, but the shape of their results makes writing a closed form expression for the complete list of said families somewhat challenging.
 
We attempt to organize all the previously known cases of \Cref{conj:number_field_counting} in \Cref{table:history}.  We include only the most general references in order to save space and to give a clear picture of where progress has been made towards \Cref{conj:number_field_counting}, although we remark that a number of earlier works also studied groups in these families. This is most pronounced in the case that $G$ is concentrated in an abelian group, as this case has seen the most incremental progress over the years. The results include abelian groups \cite{wright1989}, quartic $D_4$ \cite{cohen-diaz-y-diaz-olivier2002}, wreath products $C_2\wr H$ subject to certain growth conditions on $H$-extensions \cite{kluners2012}, $C_3\wr C_2$ \cite{kluners2005}, various families of nilpotent groups concentrated in an abelian normal subgroup \cite{klunersHab2005,fouvry-koymans2021nonicheisenberg,kluners-wang2023heisenberg,koymans-pagano2021,ALOWW}, wreath products $C_n\wr H$ subject to certain growth conditions on $H$-extensions \cite[Corollary 1.3 and 1.4]{ALOWW}, and certain semidirect products with abelian kernel \cite[Corollary 1.7]{ALOWW}. All of these results  now fall under the umbrella of \cite[Theorem 1.11]{ALOWW}, after inputting the appropriate bounds on class group torsion; thus,  they are special cases of the first row of \Cref{table:history}. We refer the reader to the introduction of \cite{ALOWW} for a more full discussion of this family.

\begin{table}[!htpb]
	\begin{tabular}{|c|c|l|}
		\hline
        \multicolumn{3}{|c|}{\bf Concentrated groups}\\\hline
		Group(s) & Concentrated in & Most general reference for known cases\\\hline
		any $G$ & an abelian $T\normal G$ & \cite[Theorem 1.11]{ALOWW}\\
        &&\hspace{1cm} subject to average bounds for certain\\
        &&\hspace{1cm} class group torsion in $G/T$-extensions\\

        $S_3 \wr B$ & $S_3^m$ & \cite[Theorem 1.9]{ALOWW}\\
        &&\hspace{1cm} subject to average bounds for certain\\
        &&\hspace{1cm} class group torsion in $B$-extensions\\
		
		$S_n\times B$, $n\le 5$& $S_n$ & \cite{jwang2021,masri-thorne-wang-tsai2020,mishra-ray2024}\\
        &&\hspace{1cm}$B$ nilpotent in degree $|B|$\\
        &&\hspace{1cm}subject to certain conditions on the\\
        &&\hspace{1cm}prime divisors of $|B|$\\

		\hline 
        
        \multicolumn{3}{c}{$\,$}\\
        \hline
        \multicolumn{3}{|c|}{\bf Properly semiconcentrated groups }\\\hline
		Group(s) & semiconcentrated in & Reference \\\hline
        $C_p^r$ for $r > 1$ & order $p$ subgroups & \cite{wright1989}\\
        $D_4$ in degree $8$ & abelian normal subgroups & \cite{shankar-varma2025octic}\\
        $D_6$ in degree $12$ & proper normal subgroups & \cite{koymans-lemke-oliver-sofos-thorne2025}\\
        \hline 
        \multicolumn{3}{c}{$\,$}\\
        \hline
        \multicolumn{3}{|c|}{\bf Groups that are not semiconcentrated}\\\hline
		\multicolumn{2}{|c|}{Group(s)} & Reference \\\hline
		\multicolumn{2}{|c|}{$C_p$ for $p$ prime} & \cite{wright1989}\\
		\multicolumn{2}{|c|}{$S_n$ in degree $n\le 5$} & \cite{datskovsky-wright1988,bhargava-shankar-wang2015}\\
		\multicolumn{2}{|c|}{$S_3$ in degree $6$} & \cite{bhargava-wood2007,belabas-fouvry2010sextic}\\
        
		\hline
	\end{tabular}
	\caption{Families containing previously known cases of Conjecture \ref{conj:number_field_counting}}
    \label{table:history}
\end{table}

Our results introduce whole new families of groups for which we can prove Conjecture \ref{conj:number_field_counting} (possibly under some analytic hypotheses), disjoint from the families considered in previous work outlined in \Cref{table:history}.
\begin{itemize}
    \item In \Cref{thm:Q8sxC2} and \Cref{thm:16T11_over_k} we prove \Cref{conj:number_field_counting} for the properly semiconcentrated group $16T11$. \Cref{thm:nilpotent_general_invariant} gives our most general unconditional result in this direction.
    
    Conditional on certain analytic assumptions, \Cref{thm:intro_nilpotent} proves \Cref{conj:number_field_counting} for infinitely many nilpotent transitive groups which are properly semiconcentrated in their abelian normal subgroups.
    
    \item Previously, \cite{ALOWW} covered groups concentrated in {\em abelian} normal subgroups or groups of the form $S_3\wr B$, subject to certain class group torsion bounds. In this paper, we prove a number of new results for transitive groups which are concentrated in a proper {\em nilpotent normal subgroup} and semiconcentrated in abelian normal subgroups. By proving new cases of \Cref{conj:twisted_number_field_counting}, we are able to apply the framework in \cite{ALOWW} to new groups.
    
    {\em Unconditional results:} \Cref{thm:intro_nilpotent_groups_with_shortcut} does this unconditionally for ten new transitive nilpotent groups, while \Cref{thm:wreath_products} and \Cref{thm:direct_products} does this for infinitely many new wreath products and direct products. \Cref{thm:unconditional_concentrated} gives our most general unconditional result in this direction.
    
    {\em Conditional results:} Conditionally on certain analytic assumptions, \Cref{thm:intro_nilpotent} proves \Cref{conj:number_field_counting} for any nilpotent group in this family. Meanwhile, \Cref{thm:main_concentrated} proves \Cref{conj:number_field_counting} for a number of new nonnilpotent groups in this family subject to average bounds for certain class group torsion. This is completely analogous to the statement of \cite[Theorem 1.11]{ALOWW} for transitive groups concentrated in a single abelian normal subgroup.
\end{itemize}

However, in addition to most of the new cases of \Cref{conj:number_field_counting} we prove being conditional on certain analytic assumptions, it is important to remark that our results prove the existence of the asymptotic \emph{without} producing an explicit value for the exponent $b$ of $\log X$ in the asymptotic. All of the previously known cases of Conjecture \ref{conj:number_field_counting} falling into one of the families in Table \ref{table:history} are proven with an explicit formula for $b$ (although sometimes evaluating that formula is not straight-forward, as outlined in \cite{jwang2025counterexamples}).

Our methods are also capable of producing power saving error terms, most of which will be new. Previously, explicit power savings were known for abelian groups \cite{frei-loughran-newton2018,alberts2024}, $S_3$ in degrees $3$ and $6$ \cite{bhargava-taniguchi-thorne2023}, $S_3\times A$ \cite{jwang2017}, $D_4$ in degree $4$ \cite{cohen-diaz-y-diaz-olivier2002, bucur-florea-serrano-lopez-varma2022,mcgown-tucker2023}, $S_4$ in degree $4$ \cite{belabas2010error}, and $S_5$ in degree $5$ \cite{shankar2014counting}.

We give explicit power savings for $4T3$ and $8T4$ in \Cref{thm:D4} and $8T11$ and $16T11$ in \Cref{thm:Q8sxC2}. All of our results for nilpotent groups prove power saving error bounds, although we opt not to compute them explicitly as they tend to become rather cumbersome as $G$ gets larger. The power savings for small groups like $D_4$ and $Q_8\rtimes C_2$ are fairly strong compared to what was previously known.
\begin{itemize}
    \item The error bound for quartic $D_4$-extensions is not the best known, falling short of $O_{\epsilon}(X^{5/8+\epsilon})$ proven by McGown and Tucker \cite{mcgown-tucker2023}. 
    \item The error bound for $D_4$ in the conductor ordering is better than the error term $O_{\epsilon}(X^{11/12+\epsilon})$ proven in \cite{friedrichsen2021secondary}.
    \item Power saving error bounds for counting octic $D_4$-fields and octic $Q_8\rtimes C_2$-fields were not previously known, as Shankar--Varma \cite{shankar-varma2025octic} and Alberts--Lemke Oliver--Wang--Wood \cite{ALOWW} only give the main terms.
\end{itemize}

As part of our work for proving new concentrated cases of \Cref{conj:number_field_counting}, we also prove new cases of the twisted number field counting conjecture, \Cref{conj:twisted_number_field_counting}. The framework of \cite{ALOWW} gives a recipe for converting new results towards \Cref{conj:twisted_number_field_counting} into new results towards \Cref{conj:number_field_counting} for concentrated groups, and our work shows how this recipe can be extended to apply to semiconcentrated groups. This conjecture is much newer, and far fewer cases are currently known. See \Cref{table:history_twisted} below.

\begin{table}[!htpb]
	\begin{tabular}{|c|c|l|}
		\hline
		Group(s) & Normal subgroup & Most general reference for known cases\\\hline
		any $G$ & abelian & \cite[Corollary 1.2]{alberts-odorney2021}\\
        &&\hspace{1cm} assuming there is at least one $G$-extension\\

        $S_3 \wr B$ & $S_3^m$ & \cite[Theorem 3.1]{ALOWW}\\
		
		$G\subseteq A\wr B$ & $T\supseteq A^m\cap G$ & \cite[Theorem 1.6 and 1.7]{choudhary2025twisted}\\
        &&\hspace{1cm} assuming there is at least one $G$-extension,\\
        &&\hspace{1cm} $G$ concentrated in $A^m \cap G$, and subject \\
        &&\hspace{1cm} to certain group theoretic conditions on\\
        &&\hspace{1cm} the abelian group $A$ and $T/(A^m\cap G$).\\
		\hline
	\end{tabular}
	\caption{Families containing previously known cases of Conjecture \ref{conj:twisted_number_field_counting}}
    \label{table:history_twisted}
\end{table}

The unconditional result \Cref{thm:unconditional_nilpotent_fibers}, and corresponding conditional result \Cref{thm:nilpotent_fibers_main}, prove new cases of \Cref{conj:twisted_number_field_counting} for $G$ semiconcentrated in abelian normal subgroups which are contained in a nilpotent normal subgroup $N\normal G$. In the process of proving \Cref{thm:wreath_products}, we verify the hypotheses of \Cref{thm:unconditional_nilpotent_fibers} to prove \Cref{conj:twisted_number_field_counting} for $G = N\wr B$ and the normal subgroup $N^m\normal G$ as described in \Cref{thm:wreath_products}, as well as $G=N\times B$ with normal subgroup $N\times 1$ as described in \Cref{thm:direct_products}.

\subsection{Layout of the paper}

We begin in Section \ref{sec:strategy} with the definition of the generating multiple Dirichlet series for $G$-extensions of $k$, $D_k(G;\underline{s})$. This essentially follows work of Gundlach \cite{gundlach2022multimalle} and Ellenberg--Venkatesh \cite{ellenberg-venkatesh2005} on ordering $G$-extensions by a system of basic invariants determined by ramification type, with a slight modification to give us some flexibility for handling the wildly ramified places. From here, we describe our method in more detail, including discussions of Hartogs' phenomenon and the relevant Tauberian theorems that show that it suffices to construct a meromorphic continuation of $D_k(G;\underline{s})$ to a large enough domain in order to prove Conjecture \ref{conj:number_field_counting}. Of particular interest is \Cref{subsec:D4}, where we work out the proof of \Cref{thm:D4}. This example can be essentially reduced to two dimensions, allowing us to include several pictures illustrating how the method works.

We then define ramification types in \Cref{sec:ramification_types}, and develop their important properties. This section is primarily setup for the underlying structure of $D_k(G;\underline{s})$. There are several equivalent notions of ``tame ramification type'' in the literature. While we make reference to only two of these in the main body of the paper, we include proofs that each of these notions are equivalent in \Cref{app:tame_ramification_types}.

Our most technical section is \Cref{sec:fiber_series}, where we construct the necessary meromorphic continuations of the subseries of $D_k(G;\underline{s})$ associated to the fibers of an abelian normal subgroup $T\normal G$. This section includes multivariable versions of several results proven in \cite{alberts2024}.

We define the cohomological invariants $h_{ur}^1(k,N,T,\pi)$ in \Cref{sec:bounds} and prove some basic properties.

We state and prove our main unconditional results in \Cref{sec:unconditional}. For each inertial invariant, we specify which point needs to be inside the region of meromorphic continuation in order to obtain an asymptotic for  ordering by that invariant.  The unconditional results from the introduction are proven in \Cref{sec:unconditional_corollaries} (except for \Cref{thm:D4}, which is proven in \Cref{subsec:D4} for illustrative purposes). For each of the groups under consideration, we prove that the point of interest identified in \Cref{sec:unconditional} indeed lies inside the region of meromorphicity. For some of the computations we rely on \verb|Magma| code, which we have provided in a github repository \cite{codeurl}.

The conditional results from the introduction are proven in \Cref{sec:conditional}, once again by proving that, under some version of the Lindel\"of hypothesis, the point of interest lies inside the region of meromorphic continuation of the multiple Dirichlet series $D_k(G;\underline{s})$.

\section*{Acknowledgments}
We would like to thank Justine Dell, Kiran S. Kedlaya, Robert Lemke Oliver, Daniel Loughran, Evan O'Dorney, and Will Sawin for helpful discussions.
This research was supported through the program {\em Oberwolfach Research Fellows} by the Mathematisches Forschungsinstitut Oberwolfach in 2025. Alberts was also supported by an AMS-Simons Travel Grant. Bucur would like to thank the DFG-funded Hausdorff Research Institute for Mathematics and the Lodha Mathematical Sciences Institute for their hospitality. Bucur was also supported by the NSF grant DMS-2002716.

\section{The General Strategy}\label{sec:strategy}

In this section we define the generating multiple Dirichlet series for $G$-extensions over a number field and outline our method for giving a meromorphic continuation of this function. Before stating the method in the abstract via Theorem \ref{thm:main}, we work out this procedure explicitly to prove \Cref{thm:D4} for $G=D_4$ in \Cref{subsec:D4}. This example is particularly nice for introducing the main ideas because it can be ``reduced to two dimensions'', allowing us to draw the relevant pictures demonstrating what happens at each step.

There are some technical definitions and steps in this section that we postpone until later in the paper, so as not to distract from the big ideas.

\subsection{The generating multiple Dirichlet series}

The essential definition we make is the generating multiple Dirichlet series for an abstract finite group $G$. It can be specialized to the single variable generating series for any transitive representation of $G$ to get the required asymptotic.

\begin{definition}\label{def:DkGs}
    Let $k$ be a number field and $G$ a finite group. For any $\pi \in \Hom(G_k,G)$ we define
    \[
        \inv_{k,G}(\pi,\underline{s}) = \prod_{\substack{\tau\in \RT_{k,0}(G)}}\prod_{\substack{\pk\text{ ram.}\\\text{in }\pi\text{ with}\\\text{type }\tau}} |\pk|^{-s_\tau},
    \]
    where $\RT_{k,0}(G)$ is the set of nontrivial $k$-ramification types in $G$ (see Definition \ref{def:RT}), $|\pk|$ is the norm of $\pk$ down to $\Q$, and $\underline{s} = (s_\tau)$ is a tuple of variables indexed by ramification types.
    
    For an abstract group $G$, we then define
    \[
        D_k(G;\underline{s}) = \sum_{\pi\in \Sur(G_k,G)} \inv_{k,G}(\pi,\underline{s}).
    \]
\end{definition}

Under the Galois correspondence, $\Sur(G_k,G)$ is in a many-to-one correspondence with the set of $G$-extensions $F/k$, showing that this is a generating series for $G$-extensions. The summands are a multivariable version of the analytic character $\disc(\pi)^{-s}$ built out of the (finitely many) primes that ramify in $\pi$.

In order to make sense of this definition, we need to define ``ramification types.'' We postpone the formal definitions to \Cref{sec:ramification_types}, as they get somewhat technical. Briefly, we consider two flavors of ramification type:
\begin{itemize}
    \item ``Tame ramification types'' measure the ways in which a prime can tamely ramify in a $G$-extension over $k$. The formulation we give in \Cref{def:tameRT} is equivalent to the various other notions appearing in the literature (often referred to as just ``ramification types'', see \cite{malle2004,ellenberg-venkatesh2005,smith2022selmer1,gundlach2022multimalle,alberts2024} for several examples). In fact, our construction of $D_k(G;\underline{s})$ is inspired by Gundlach's construction of a similar multiple Dirichlet series in \cite{gundlach2022multimalle}.

    For the purposes of this section, it will suffice to know that the tame ramification types are in bijection with $k$-conjugacy classes as defined by Malle in \cite{malle2004}: these are minimal subsets of $G$ closed under conjugation and closed under the cyclotomic action $x.g = g^{\chi(x)^{-1}}$, where $\chi:G_k\to \Gal(k\Q^{\rm ab}/k)\subseteq \widehat{\Z}^{\times}$ is the cyclotomic character.

    \item ``Wild ramification types'' measure the ways in which a prime can wildly ramify in a $G$-extension over $k$. To the best of our knowledge, the notion of wild ramification types does not currently exist in the literature. Most researchers do not require it, as the wild primes are not expected to affect the rates of growth for number field counting functions (at least, in characteristic $0$).

    The same is mostly true in our situation: it is reasonable to expect $D_k(G;\underline{s})$ to be an entire function in the wild variables (that is, complex variables indexed by wild ramification types), and in all meromorphic continuations we construct the singularities will indeed be cut out by equations in only the tame variables. However, the wild variables need to be included to make sure that $D_k(G;\underline{s})$ cleanly specializes to the generating single variable Dirichlet series we are interested in.
    
    For the purposes of this section, and in particular our example in \Cref{subsec:D4}, the reader can safely ignore the wild ramification types.
\end{itemize}

The most important property of $D_k(G;\underline{s})$ is that it specializes to all single variable generating Dirichlet series for $G$-extensions ordered by an inertial invariant.

\begin{proposition}\label{prop:specialize}
    Let ${\rm wt}:\RT_{k,0}(G)\to \R^+$ be a weight function on nontrivial ramification types and define an invariant $\inv:\Hom(G_k,G) \to \R^+$ by
    \[
        \inv(\pi) = \prod_{\tau\in \RT_{k,0}(G)} \prod_{\substack{\pk\text{ ram.}\\\text{ in }\pi\text{ with}\\\text{type }\tau}} |\pk|^{{\rm wt}(\tau)}.
    \]
    Specializing $D_k(G;\underline{s})$ to the complex line $s_\tau = {\rm wt}(\tau)s$ gives the generating Dirichlet series
    \[
        \sum_{\pi\in \Sur(G_k,G)} \inv(\pi)^{-s}.
    \]
\end{proposition}

The proof is immediate, as $\inv(\pi)^{-s}$ is by definition the specialization of $\inv_{k,G}(\pi,\underline{s})$ to this complex line. As a consequence, properties of  $D_k(G;\underline{s})$ can be specialized to the generating series of $G$-extensions ordered by any inertial invariant. In this way, $D_k(G;\underline{s})$ can be interpreted as the generating series for ``all inertial invariants'' of $G$-extensions over $k$.

\begin{remark}
    We remark that the equivalent term \emph{inertial height} is used in some places in the literature instead of inertial invariant, to avoid overloading the word ``invariant'' and draw important comparisons with counting rational points. See, for example, \cite{loughran2024malle,tavernier2025counting,gundlach2026lifts}.
\end{remark}

We can extend the notion of semiconcentrated groups to general inertial invariants. First, we define the equivalent of the minimal index.

\begin{definition}\label{def:min_index_gen_inv}
    For a finite group $G$ and a weight function ${\rm wt}:\RT_{k,0}(G)\to \R^+$ on nontrivial ramification types, we define the minimal weight associated to the inertial invariant $\inv$ induced by $\wt$ to be \[a_\inv(G) = \min\{ \wt(\tau); \tau \in \RT_{k,0}(G) \}.\]
    We say that an element $g\in G$ is of minimal weight with respect to $\inv$ if $\wt(g) = a_\inv(G)$.
\end{definition}

We can now generalize the definitions of (semi)concentrated groups to these more general invariants.
\begin{definition}\label{def:semiconc_gen_inv}
    We say that a finite group $G$ is \emph{concentrated} in a proper normal subgroup $T$ with respect to an inertial invariant $\inv$ if all the elements of $G$ of minimal weight (for the $\wt$ function associated to $\inv$) are contained in $T$.
    
    We say that $G$ is \emph{semiconcentrated} with respect to $\inv$ in a collection $T_1, \dots, T_r$ of proper normal subgroups if all the elements of minimal weight are contained in the union of of the subgroups $T_1, \dots, T_r.$
    
    We say that $G$ is \emph{properly semiconcentrated} with respect to $\inv$ if it is semiconcentrated but not concentrated with respect to $\inv.$
\end{definition}

Our primary goal is to construct a meromorphic continuation of $D_k(G;\underline{s})$, which means we should discuss where this function is convergent.

\begin{proposition}\label{prop:nilpotent_absolute_convergence}
    Let $k$ be a number field and $G$ be a nilpotent group. Then $D_k(G,\underline{s})$ converges absolutely on the orthant
    \[
        \{\underline{s} : {\rm Re}(s_\tau) > 1\text{ for each }\tau\in \RTo{k}(G)\},
    \]
    where $\RTo{k}(G)$ is the set of nontrivial tame ramification types.
\end{proposition}

\begin{proof}
    While we have deferred the precise definition of ramification types until \Cref{sec:ramification_types}, the intricacies of the definition are mostly unnecessary for this proof.

    It follows from \cite[Theorem 1.6]{kluners-wang2022elltorsion} that
    \[
        \#\left\{F/k\ G\text{-extension} : \prod_{\mathfrak{p}\text{ ram. in }F/k}|\mathfrak{p}| = D\right\} \ll_{k,G,\epsilon} D^{\epsilon}.
    \]
    In particular, this implies
    \[
        D_k(G;\underline{s}) \ll_{k,G,\epsilon} \sum_{\underline{n}} \prod_{\tau\in \RT_{k,0}(G)} n_\tau^{\epsilon - {\rm Re}(s_\tau)}
    \]
    where $\underline{n}=(n_\tau)$ varies over tuples of positive integers indexed by nontrivial ramification types, as the only primes that can occur in a summand are precisely those ramified in the homomorphism $\pi$, and therefore ramified in the corresponding field. We remark that only the finitely many primes dividing $|G|$ can be wildly ramified and we will see in \Cref{sec:ramification_types} that there are only finitely many wild ramification types. Thus, the implied constant absorbs the contribution of wild ramification and we can bound
    \[
        D_k(G;\underline{s}) \ll_{k,G,\epsilon} \sum_{\underline{n}} \prod_{\tau\in \RTo{k}(G)} n_\tau^{\epsilon - {\rm Re}(s_\tau)}.
    \]
    This series converges absolutely for ${\rm Re}(s_\tau) > 1 + \epsilon$ for each tame $\tau$, so that letting $\epsilon$ tend to zero concludes the proof.
\end{proof}

\begin{remark}
    When $G$ is not nilpotent, a similar argument can be used to give a nonempty orthant of absolute convergence. The only difference is that for a general group one needs to use results like those of \cite{schmidt1995} instead of \cite[Theorem 1.6]{kluners-wang2022elltorsion} to produce the upper bound, which will be significantly weaker. We only state \Cref{prop:nilpotent_absolute_convergence} for nilpotent groups because this is the only case we use to prove our results.
\end{remark}

\subsection{The Hartogs phenomenon}
Our method relies on taking advantage of Hartogs phenomenon for multivariable complex analytic functions, specifically for tubular regions in the following form due to Bochner:

\begin{theorem}[Bochner's Tube Theorem {\cite[Theorem 2.5.10]{hormander1973severalvariables}}]\label{thm:hartogs}
    Let $\Omega\subset \C^n$ be a connected tubular region, that is $\Omega = \{\underline{s}\in \C^n : {\rm Re}(\underline{s})\in U\}$ for some connected open set $U\subseteq \R^n$.

    If $f$ is a holomorphic function on $\Omega$, then $f$ has an analytic continuation to the convex hull ${\rm Hull}(\Omega)$.
\end{theorem}

Essentially, in multi-dimensional complex space a singularity occurs as a union of codimension $1$ hypersurfaces in $\C^n$, i.e.~ a divisor. This is a restrictive structure when $n \ge 2$, and in particular not every domain $\Omega\subseteq \C^n$ has a boundary given by some (limit of) divisors. This means that any analytic function on $\Omega$ can be continued to a larger region whose boundary is given by a (limit of) divisors. In the case of connected tubular regions, the boundary is given by a (limit of) divisors if and only if the tubular region is convex. We can use this to produce analytic continuations essentially for ``free''.

We will apply the theorem to argue that a function $f(\underline{s})$ that is meromorphic in a region $\Omega$ with polar divisors $D_1, \dots, D_r$  has meromorphic continuation to ${\rm Hull}(\Omega)$ with the same polar divisors and the same orders of the poles along them. To see that, one multiplies the function $f(\underline{s})$ by a simpler function $g(\underline{s})$ that precisely cancels the poles and applies Bochner's tube theorem to this new function $g(\underline{s})f(\underline{s}).$

We do this by using the two (or generally, at least two) fibrations of $D_k(G;\underline{s})$ instead of the functional equations in the usual multiple Dirichlet series playbook that is explained in \cite{DGH}. In both cases, the goal is to meromorphically continue the multiple Dirichlet series of interest to a tubular region whose convex hull contains the pole that gives the asymptotic via Tauberian arguments.

There are two reasons that we opt not to consider functional equations in this paper: first of all, $D_{k}(G,\underline{s})$ does not typically have a functional equation. For the partitions of $D_k(G;\underline{s})$ corresponding to a normal subgroup $T\normal G$, it is in fact often the case that the summands have natural boundary at the hyperplane ${\rm Re}(s_\tau) = 0$ for each $\tau\in \iota_*\RTo{k}(T)$. Second, the poles we are interested in are easier to access and a functional equation is often not necessary. The ``rightmost pole'' of $D_k(G;\underline{s})$ is expected to be on the boundary of the region of absolute convergence, ${\rm Re}(s_\tau) > 1$. This means we only need to produce a slight increase to the region of meromorphicity to reveal this pole. In contrast, when studying moments of $L$-functions one is interested in poles at points like $(1/2,1)$, which is a distance of $1/2$ away from the region of absolute convergence. This indicates that a larger region of meromorphicity is needed, which may require using functional equations to reach.

In order to extend this process to prove power saving error bounds, or upper bounds like in \Cref{thm:nilpotent_fibers_uniform}, we will need a way to extrapolate bounds for $f$ on $\Omega$ to the convex hull. This is accomplished by an application of the Phragm\'en--Lindel\"of principle together with some useful facts about multivariable holomorphic functions.

\begin{proposition}\label{prop:multivariable_phragmen-lindelof}
    Let $\Omega\subset \C^n$ be a connected tubular region, that is $\Omega = \{\underline{s}\in \C^n : {\rm Re}(\underline{s})\in U\}$ for some connected open set $U\subseteq \R^n$. For any $\underline{s}\in \C^n$, write $\underline{s}=\underline{\sigma}+i\underline{t}$ for the real and imaginary parts $\underline{\sigma},\underline{t}\in \R^n$.

    Let $f$ be a holomorphic function on $\Omega$, and suppose there exist constants $A$ and $\mu$ for which
    \[
        |f(\underline{\sigma}+i\underline{t})| \ll_{\underline{\sigma}} A(1+|\underline{t}|)^{\mu}.
    \]
    Then the same bound holds for the analytic continuation to the convex hull ${\rm Hull}(\Omega)$.
\end{proposition}

\begin{proof}
    The idea is to apply the Phragm\'en--Lindel\"of principle to a specialization of $f$ to a complex line. We will use the version given by \cite[Theorem 2]{rademacher1959phragmen}, although a number of other forms would also suffice.
    
    For any $\underline{s}\in {\rm Hull}(\Omega)$, the definition of the convex hull implies that there exist points $\underline{\omega}_1,\underline{\omega}_2\in \Omega$ such that $\underline{s}$ is on the line segment connecting $\underline{\omega}_1$ to $\underline{\omega}_2$. If $f$ has finite order on this complex line, then the result follows from the Phragm\'en--Lindel\"of principle.

    Choose a closed tubular region $\Omega'\subseteq \Omega$ which contains $\underline{\omega}_1,\underline{\omega}_2$, and lies over a bounded domain in $\R^n$. It suffices to show that $f$ has finite order on the convex hull ${\rm Hull}(\Omega')$, which necessarily contains the line segment connecting $\underline{\omega}_1$ to $\underline{\omega}_2$. That is it suffices to show that there exist constants $C$ and $c$ for which
    \[
        |f(\underline{s})| \le C e^{|\underline{t}|^c}
    \]
    for each $\underline{s}\in {\rm Hull}(\Omega')$. This is certainly true on $\Omega'$, as $A(1+|\underline{t}|)^{\mu} \le Ce^{|\underline{t}|}$ for some positive constant $C$ depending on $A$ and $\mu$.

    We can extrapolate information to the convex hull using \cite[Proposition C.5]{cech2024ratios}: if $f$ and $h$ are holomorphic on $\Omega$ and satisfy $|f(\underline{s})|\le |h(\underline{s})|$ for each $\underline{s}\in \Omega$, then the analytic continuations also satisfy $|f(\underline{s})|\le |h(\underline{s})|$ for each $\underline{s}\in {\rm Hull}(\Omega)$.

    Let $\underline{a}\in \R^n$ be a distance of at least $1$ from $\Omega'$ in each coordinate, and take
    \[
        h(\underline{s}) := A\prod_{j=1}^n (4s_j - 4a_j)^{2\mu'},
    \]
    where $\mu' = \max\{0,\lceil \mu/2 \rceil\}$, so that $\mu'$ is a positive integers and satisfies $\mu \le 2\mu'$.

    On the one hand, for each $\underline{s}\in \Omega'$
    \begin{align*}
        |f(\underline{s})| \le A(1+|\underline{t}|)^{\mu}
        \le 4^{2\mu'}A(1+|\underline{t}|^2)^{\mu'}
        \le 4^{2\mu'}A \prod_{j=1}^n (1+|t_j|^2)^{\mu'}
        \le 4^{2\mu'}A\prod_{j=1}^n|s_j-a_j|^{2\mu'}
        = |h(\underline{s})|,
    \end{align*}
    so we can conclude that $|f(\underline{s})| \le |h(\underline{s})|$ for each $\underline{s}\in {\rm Hull}(\Omega')$.

    On the other hand,
    \begin{align*}
        |h(\underline{s})| &\ll_{\mu} \prod_{j=1}^n\left(|\sigma_j-a_j|^2 + |t_j|^2\right)^{\mu'}
        \ll_{n,\mu,\Omega'} \prod_{j=1}^n\left(1 + |t_j|^2\right)^{\mu'}
        \le (1+|\underline{t}|^2)^{n\mu'}
        \ll_{n,\mu,\Omega'} e^{n\mu' |\underline{t}|^2},
    \end{align*}
    which implies $f$ has finite order on ${\rm Hull}(\Omega')$, concluding the proof.
\end{proof}

We will most often use the following consequence of \Cref{prop:multivariable_phragmen-lindelof}, which allows for a function in the exponent depending on variables that are unrelated to the convex hull.

\begin{corollary}\label{cor:independent_variable_multi_phragmen_lindelof}
    Let $\Omega'\subseteq \C^{n-d}$ be a connected tubular region, $d \le n$, and $\underline{c}\in \R^d$. The domain
    \begin{align*}
        \Omega &= \{\underline{z} \in \C^d : {\rm Re}(z_j) > c_j\text{ for }j=1,\dots d\}\times \Omega'
    \end{align*}
    is a connected tubular region in $\C^n$.

    Suppose $f$ is holomorphic on $\Omega$, and that there exists a constant $A$ and a continuous function $\mu:\R^d\to \R_{\ge 0}$ which is non-increasing in each coordinate such that for each $\epsilon > 0$
    \[
        |f(\underline{s})| \ll_{\underline{\sigma},\epsilon} A(1+|\underline{t}|)^{\mu(\sigma_1,\dots,\sigma_d)+\epsilon}
    \]
    for each $\underline{s}\in \Omega$. Then the same bound holds on the convex hull
    \[
        {\rm Hull}(\Omega) = \{\underline{z} \in \C^d : {\rm Re}(z_j) > c_j\text{ for }j=1,\dots d\}\times {\rm Hull}(\Omega').
    \]
\end{corollary}

\begin{proof}
    For each $\underline{x}\in \R^d$ with $x_j > c_j$ for $j=1,\dots, d$, we apply \Cref{prop:multivariable_phragmen-lindelof} to the domain
    \[
        \{\underline{z}\in \C^d : z_j > x_j\text{ for }j=1,\dots,d\} \times \Omega'
    \]
    with the bound
    \[
        |f(\underline{s})|\ll_{\underline{\sigma}} A(1+|\underline{t}|)^{\mu(x_1,\dots,x_d)+\epsilon},
    \]
    where the bound holds by $\mu$ being a non-increasing function. Now that $\mu(x_1,\dots,x_d)+\epsilon$ is constant in terms of $\underline{s}$, \Cref{prop:multivariable_phragmen-lindelof} implies the same bound holds on the convex hull
    \[
        \{\underline{z}\in \C^d : z_j > x_j\text{ for }j=1,\dots,d\} \times {\rm Hull}(\Omega').
    \]
    For any particular $\underline{\sigma}$ with $\sigma_j > c_j$, choose $x_j = \sigma_j-\epsilon$. Continuity of $\mu$ then implies
    \[
        \mu(\sigma_1-\epsilon,\dots,\sigma_d-\epsilon) \le \mu(\sigma_1,\dots \sigma_d) + \epsilon',
    \]
    for $\epsilon' > 0$ tending to $0$ with $\epsilon$.
\end{proof}

\begin{remark}
    In the single variable setting, one often allows different bounds for $f(s)$ at the edges of the strip $a \le {\rm Re}(s) \le b$, and the Phragm\'en--Lindel\"of principle gives a bound on the interior of the strip that is interpolated between the bounds on the edges. A similar generalization certainly holds for multivariable complex analytic functions.

    We opt not to state \Cref{prop:multivariable_phragmen-lindelof} at this level of generality because (1) we do not need it for our application, and (2) the ``edges'' of $\Omega$ can be significantly more complicated, making a complete statement much more unwieldy than the current form of \Cref{prop:multivariable_phragmen-lindelof}.
\end{remark}

\subsection{Tauberian theorems}

We use Bochner's tube theorem to produce meromorphic continuations for multiple Dirichlet series. Our results will then follow by restricting to a single variable Dirichlet series along a complex line as in \Cref{prop:specialize}, then applying a Tauberian theorem. The Tauberian theorems we use can both be found in the expository paper \cite{pierce2025guide}, which includes a useful exposition of these results and self-contained proofs. We provide the statements of these theorems here for ease of reference.

The first Tauberian theorem we make use of is \cite[Theorem A]{pierce2025guide}, also found in earlier works such as \cite[Theorem 7.7]{bateman-diamond2004}.

\begin{theorem}\label{thm:tauberian_main_term}
    Let $a_n\ge 0$ be a sequence of nonnegative numbers, and $A(s) = \sum a_n n^{-s}$ the associated Dirichlet series.

    Suppose we know the following:
    \begin{enumerate}[(a)]
        \item $A(s)$ converges absolutely on the halfplane ${\rm Re}(s) > \sigma_a$, and
        \item $A(s)$ has a meromorphic continuation to an open neighborhood of the halfplane ${\rm Re}(s) \ge \sigma_a$ with a pole at $s=\sigma_a$ of order $b\ge 1$.
    \end{enumerate}
    Then
    \[
        \sum_{n\le X} a_n \sim \left(\frac{1}{\Gamma(b)}\lim_{s\to\sigma_a} (s-\sigma_a)^bA(s)\right)X^{\sigma_a}(\log X)^{b-1}.
    \]
\end{theorem}

The key hypothesis that $a_n \ge 0$ is what allows this result to require essentially no information to the left of $s=\sigma_a$. We will only apply a Tauberian theorem to specializations of $D_k(G;\underline{s})$ to complex lines of the form described in \Cref{prop:specialize}, which by construction has nonnegative integer coefficients determined by the number of $\pi\in \Sur(G_k,G)$ with a particular fixed invariant.

One would like a larger region of meromorphicity for $A(s)$ to correspond to bounds for the error term. This turns out to be true if we additionally assume a $t$-aspect bound for $A(s)$ in vertical strips. The following result is the nonnegative case of \cite[Theorem 6.1]{alberts2024}, see \cite[Theorem B]{pierce2025guide} for a sharper result.

\begin{theorem}\label{thm:tauberian_power_saving}
    Let $a_n\ge 0$ be a sequence of nonnegative numbers, and $A(s) = \sum a_n n^{-s}$ the associated Dirichlet series.

    Suppose we know the following:
    \begin{enumerate}[(a)]
        \item $A(s)$ converges absolutely on the halfplane ${\rm Re}(s) > \sigma_a$,
        \item $A(s)$ has a meromorphic continuation to the halfplane ${\rm Re}(s) > \sigma_a - \delta$ with finitely many poles $s_1,s_2,\dots,s_r$ in this region, and
        \item $|A(\sigma+it)| \ll (1+|t|)^{\xi}$ for each $\sigma > \sigma_a - \delta$.
    \end{enumerate}
    Then
    \[
        \sum_{n\le X} a_n = \sum_{j=1}^r \underset{s=s_j}{\rm Res}\left(A(s)\frac{X^s}{s}\right) + O_{\epsilon}\left(X^{\sigma_a - \frac{\delta}{\xi+1}+\epsilon}\right).
    \]
\end{theorem}

If $s=s_j$ is a pole of order $b_j$, then a direct computation as in \cite[Remark 2.3.1]{pierce2025guide} shows that
\[
    \underset{s=s_j}{\rm Res}\left(A(s)\frac{X^s}{s}\right) = X^{s_j} P_j(\log X)
\]
for an explicit polynomial $P_j$ of degree $b_j-1$. This is how we present the main terms in our results.

\subsection{$D_4$-extensions as an example}\label{subsec:D4}

We present a sketch of a meromorphic continuation for $D_{\Q}(D_4;\underline{s})$. The goal of this section is to demonstrate our method in a small, concrete example, before jumping into the general results. The heart of the argument is using the two fibrations given by the two normal abelian subgroups described below to get meromorphic continuation to a region of $\C^n$ that is {\em not} a domain of holomorphy. We then apply Bochner's tube theorem to get meromorphic continuation to a domain that contains an open neighborhood of the right-most pole, so that we can apply a Tauberian theorem in order to get the counting asymptotic.

Fix a presentation $D_4 = \langle a,b : a^4 = b^2 = 1,\ bab^{-1}=a^{-1}\rangle$. Throughout the paper we will use the LMFDB labels for conjugacy classes. For $D_4$, see \cite[\href{https://www.lmfdb.org/GaloisGroup/4T3}{Transitive Group 4T3}]{lmfdb} (after choosing an embedding that sends $b$ to a transposition in $S_4$) and \cite[\href{https://www.lmfdb.org/GaloisGroup/8T4}{Transitive Group 8T4}]{lmfdb} for information on the quartic and octic orderings. Combined with the weights for the conductor derived from \cite[Proposition 2.4]{altug-shankar-varma-wilson2021}, the conjugacy classes are given by the following table:

\begin{table}[ht]
    \centering
    $$\begin{array}{lllllll}
     \textrm{Label} & \textrm{Representative} & \textrm{Size} & \textrm{Order} & \textrm{Quartic} & \textrm{Conductor} & \textrm{Octic} \\
     \textrm{} & \textrm{} & \textrm{} & \textrm{} & \textrm{Index} & \textrm{Weight} & \textrm{Index}\\
    \hline
    \textrm{1A} &  1 &  1 &  1 &  0 & 0 & 0\\
    \textrm{2A} &  a^{2} &  1 &  2 & 2 & 2 & 4\\
    \textrm{2B} &  ab &  2 &  2 & 2 & 1 & 4\\
    \textrm{2C} &  b &  2 &  2 &  1 & 1 & 4\\
    \textrm{4A} &  a &  2 &  4 &  3 & 2 & 6\\
    \end{array}$$

    \caption{Tame ramification types of $D_4$}
    \label{tab:D4_ram_types}
\end{table}

These conjugacy classes are all $\Q$-conjugacy class as each one is closed under the cyclotomic action, and so are in bijection with the tame ramification types. Thus, the generating multiple Dirichlet series is a function of four tame variables (and several wild variables)
\[
    D_\Q(D_4;s_{2A},s_{2B},s_{2C},s_{4A},\underline{w}).
\]
\Cref{prop:nilpotent_absolute_convergence} states that this series converges absolutely on the orthant cut out by $\sigma_\tau > 1$ for each nontrivial tame ramification type $\tau$.

In order to draw pictures, we will look at the projection of the variables to $(\sigma_{2B},\sigma_{2C})$. This will not be a significant sacrifice for visualizing the argument, as $\{2B,2C\}$ contains the minimum weight elements in each of the three orderings we want to consider (save for the central element in the octic ordering, which is easier to deal with separately). In \Cref{fig:abs_conv} we have shaded the region of absolute convergence. The two lines are the specializations that we are interested in: for the quartic discriminant, the specialization is $(s_{2B},s_{2C}) = (2s,s)$, for the conductor $(s_{2B},s_{2C}) = (s,s)$, and for the octic discriminant $(s_{2B},s_{2C}) = (8s,8s)$ (which is in fact the same line as for the conductor in these two variables).

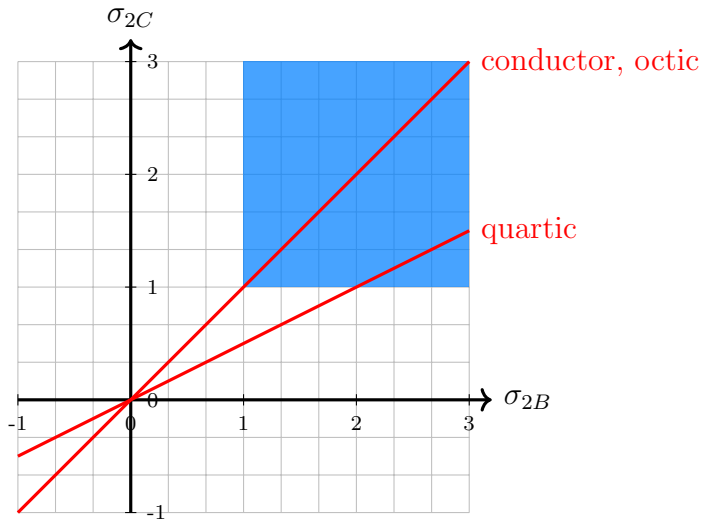
\begin{figure}[!ht]
\begin{center}
\hspace*{3cm} 
    \begin{tikzpicture}[scale=1.5]

    \draw[gray!50, thin, step=1/3] (-1,-1) grid (3,3);
    \draw[very thick,->] (-1,0) -- (3.2,0) node[right] {$\sigma_{2B}$};
    \draw[very thick,->] (0,-1) -- (0,3.2) node[above] {$\sigma_{2C}$};

    \foreach \x in {-1,...,3} \draw (\x,0.05) -- (\x,-0.05) node[below] {\tiny\x};
    \foreach \y in {-1,...,3} \draw (-0.05,\y) -- (0.05,\y) node[right] {\tiny\y};

    \fill[blue!50!cyan,opacity=0.6] (1,3) -- (1,1) -- (3,1) -- (3,3);
    \fill[blue!50!cyan,opacity=0.3] (1,3) -- (1,1) -- (3,1) -- (3,3);
    
    \draw[red,very thick,-] (-1,-1) -- (3,3) node[right] {conductor, octic};
    \draw[red,very thick,-] (-1,-1/2) -- (3,1.5) node[right] {quartic};
    \end{tikzpicture}
    \caption{\label{fig:abs_conv} Region of Absolute Convergence}
\end{center}
\end{figure}

Our goal is to construct a meromorphic continuation to a larger segment of these lines, which we will do by constructing a meromorphic continuation of the entire multiple Dirichlet series.

Each of the three orderings we are considering for $D_4$ is semiconcentrated in the abelian normal subgroups
\begin{align*}
    T_B:= \langle 2A, 2B\rangle = 1A\cup 2A\cup 2B\text{ and }T_C := \langle 2A, 2C\rangle = 1A \cup 2A \cup 2C.
\end{align*}
(In fact, the quartic ordering is concentrated in solely $T_C$). We use the two normal subgroups $T_B$ and $T_C$ to partition the series. This is analogous to the ``fibering method'' used in \cite{ALOWW}. It will be useful to know that both of these subgroups are abstractly isomorphic to $C_2\times C_2$, and their corresponding quotients are both isomorphic to $C_2$.

We first give an analytic continuation in the ``$T_C$-directions". Write $q_C:G\to G/T_C$ for the corresponding quotient map. For each $\pi\in \Sur(G_\Q,C_2)$ we define
\[
    D_\Q(T_C,\pi;\underline{s}) := \sum_{\substack{\psi\in \Sur(G_\Q,D_4)\\(q_C)_*\psi = \pi}} \inv_{\Q,D_4}(\psi,\underline{s})
\]
for each $\pi \in \Sur(G_\Q,C_2)$. This is precisely the subseries of extensions whose fixed field by $T_C$ is the quadratic field defined by $\pi$. These give a partition
\[
    D_{\Q}(D_4,\underline{s}) = \sum_{\pi\in \Sur(G_\Q,C_2)} D_\Q(T_C,\pi;\underline{s}).
\]
The most technical part of our work is determining the analytic structure of $D_\Q(T_C,\pi;\underline{s})$ with enough accuracy to say something about the summation. This is done in general in Theorem \ref{thm:meromorphic_continuation_of_fiber_series}, which works in this case because $T_C$ is abelian. We summarize the result for this specific example here.

\begin{description}
    \item[Fact 1] $D_\Q(T_C,\pi;\underline{s})$ is meromorphic on the region
    \[
        \{\underline{s} : {\rm Re}(s_{2A}) > 1/2,\ {\rm Re}(s_{2C}) > 1/2\},
    \]
    with possible poles at $s_{2A}=1$ and $s_{2C} = 1$ of order at most $1$. Note that there are no restrictions on $s_{2B}$, $s_{4A}$, or any of the wild variables.

    \item[Fact 2] As a part of proving Fact 1 above, $D_\Q(T_C,\pi;\underline{s})$ is shown to be an explicit finite sum of products of Hecke $L$-functions in this region, which is bounded above by
    \[
        |D_\Q(T_C,\pi;\underline{s})| \ll  |\disc(\pi)|^{-\min\{\sigma_{2B},\sigma_{4A}\}+\epsilon}\max_{\chi_0}L(s_{2A}, \chi_0)\max_{\chi}\left\{|L_{\pi}(s_{2C},\chi)|,1\right\},
    \]
    where $\chi$ ranges over Hecke characters over the quadratic field cut out by $\pi$ which satisfies $\chi^2 = 1$ and is unramified away from $\{2,\infty\}$, and $\chi_0$ ranges over quadratic Dirichlet characters over $\Q$ unramified away from $2$. In particular, the conductor of $\chi$ is bounded $\mathfrak{f}(\chi) \ll 1$ which makes the analytic conductor of the Hecke $L$-function $L_{\pi}(s,\chi)$ over a quadratic field bounded by
    \[
        \mathfrak{q}(\chi,s) \ll |\disc(\pi)|(|s|+4)^2
    \]
    for an absolute implied constant. See \cite[Chapter 5]{iwaniec-kowalski2004} for more information on the analytic conductor.
    
    The $L(s_{2A}, \chi_0)$ factor is a Dirichlet $L$-function of conductor dividing $8$, which helps streamline this example.
\end{description}

Theorem \ref{thm:meromorphic_continuation_of_fiber_series} is proven using the twisted counting framework of \cite{alberts2021,alberts-odorney2021,alberts2024}, as the function $D_k(T_C,\pi;\underline{s})$ is closely related to the generating series for crossed homomorphisms $Z^1(\Q,T_C(\pi))$.

Let $p(\underline{s}) = \frac{(s_{2A}-1)(s_{2B}-1)(s_{2C}-1)}{(s_{2A}+1)(s_{2B}+1)(s_{2C}+1)}$. We use the Weyl-strength hybrid subconvexity bounds for Dirichlet $L$-functions \cite[Theorem 1.1]{petrow2023fourth} and the subconvexity bounds for Hecke $L$-functions proven in \cite[Corollary 1.2]{yang2023burgess} following Burgess' work to prove that
\begin{align*}
    |L(s,\chi_0)| &\ll_{\epsilon} (1+|t|)^{\frac{1}{3}\max\{1-\sigma,0\}+\epsilon}\\
    |L_\pi(s,\chi)| &\ll_{\epsilon} \mathfrak{q}(\chi,s)^{\frac{3}{8}\max\{1-\sigma,0\}+\epsilon} \ll_{\epsilon} |\disc(\pi)|^{\frac{3}{8}\max\{1-\sigma,0\}+\epsilon} (1+|t|)^{\frac{3}{4}\max\{1-\sigma,0\}+\epsilon}
\end{align*}
for each $\sigma > 1/2$ (with $s$ bounded away from $1$ in case $\chi$ or $\chi_0$ is the trivial character). For $\underline{t}$ the imaginary part of $\underline{s}$, we conclude that
\begin{align*}
    |p(\underline{s})D_\Q(T_C,\pi;\underline{s})|
    &\ll |\disc(\pi)|^{-\min\{\sigma_{2B},\sigma_{4A}\} + \frac{3}{8}\max\{1-\sigma_{2C},0\}+\epsilon}(1+|\underline{t}|)^{\frac{1}{3}\max\{1-\sigma_{2A},0\} + \frac{3}{4}\max\{1-\sigma_{2C},0\}+\epsilon}
\end{align*}
on the region $\sigma_{2A},\sigma_{2B} > 1/2$. The factor $p(\underline{s})$ is used to cancel out poles of the zeta function and Hecke $L$-function at the trivial character.

On the region $\sigma_{2A},\sigma_{2B} > 1/2$, the above bound implies that $|p(\underline{s})D_\Q(D_4,\underline{s})|$ is bounded above by
\begin{align*}
    &\ll \sum_{\pi\in \Sur(G_\Q,C_2)}|\disc(\pi)|^{-\min\{\sigma_{2B},\sigma_{4A}\} + \frac{3}{8}\max\{1-\sigma_{2C},0\} + \epsilon} (1+|\underline{t}|)^{\frac{1}{3}\max\{1-\sigma_{2A},0\} + \frac{3}{4}\max\{1-\sigma_{2C},0\} + \epsilon}.
\end{align*}
This converges absolutely if $-\min\{\sigma_{2B},\sigma_{4A}\} + \frac{3}{8}\max\{1-\sigma_{2C},0\} < -1$. Putting these together, we have constructed an analytic continuation of $p(\underline{s})D_\Q(D_4,\underline{s})$ to the region cut out by
\begin{align*}
    \sigma_{2A} &> \frac{1}{2}, & \sigma_{2C} &> \frac{1}{2}, & \sigma_{2B} &> 1, & \sigma_{4A} &> 1, &
    \sigma_{2B} + \frac{3}{8}\sigma_{2C} &> \frac{11}{8}, & \sigma_{4A} + \frac{3}{8}\sigma_{2C} &> \frac{11}{8},
\end{align*}
with a $t$-aspect bound
\[
    |p(\underline{s})D_\Q(D_4,\underline{s})| \ll  (1+|\underline{t}|)^{\frac{1}{3}\max\{1-\sigma_{2A},0\} + \frac{3}{4}\max\{1-\sigma_{2C},0\}+\epsilon}
\]
in vertical strips. To simply the remaining argument, we consider the subdomain
\[
    \Omega_C = \left\{\underline{s} : \sigma_{2A} > 1/2,\ \sigma_{2C} > 1/2,\ \sigma_{2B} > 1,\ \sigma_{4A} > 19/16,\ \sigma_{2B} + \frac{3}{8}\sigma_{2C} > \frac{11}{8}\right\}.
\]
This region, projected to the $(\sigma_{2B},\sigma_{2C})$ plane, is depicted in \Cref{fig:TC_direction}.

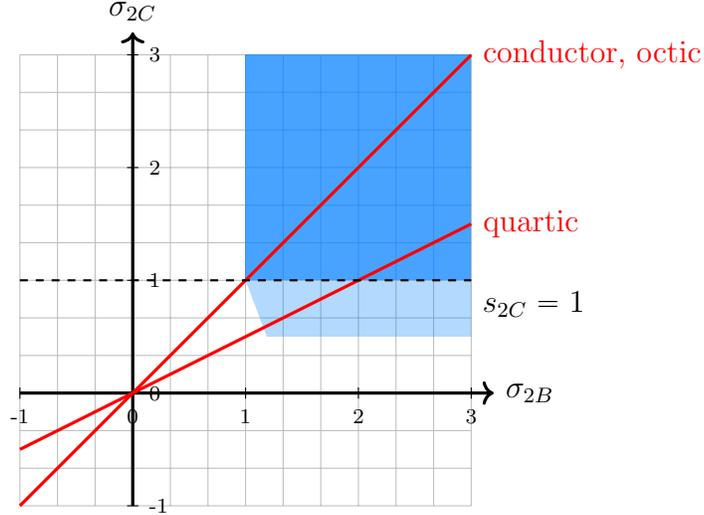
\begin{figure}[!ht]
\begin{center}
    \hspace*{3cm} 
    \begin{tikzpicture}[scale=1.5]

    \draw[gray!50, thin, step=1/3] (-1,-1) grid (3,3);
    \draw[very thick,->] (-1,0) -- (3.2,0) node[right] {$\sigma_{2B}$};
    \draw[very thick,->] (0,-1) -- (0,3.2) node[above] {$\sigma_{2C}$};

    \foreach \x in {-1,...,3} \draw (\x,0.05) -- (\x,-0.05) node[below] {\tiny\x};
    \foreach \y in {-1,...,3} \draw (-0.05,\y) -- (0.05,\y) node[right] {\tiny\y};

    \fill[blue!50!cyan,opacity=0.6] (1,3) -- (1,1) -- (3,1) -- (3,3);
    \fill[blue!50!cyan,opacity=0.3] (3,1/2) -- (19/16,1/2) -- (1,1) -- (1,3) -- (3,3);

    \draw[red,very thick,-] (-1,-1) -- (3,3) node[right] {conductor, octic};
    \draw[red,very thick,-] (-1,-1/2) -- (3,1.5) node[right] {quartic};

    \draw[black,dashed,thick,-] (-1,1) -- (3,1) node[below right] {$s_{2C}=1$};
    \end{tikzpicture}
    \caption{\label{fig:TC_direction} Continuation in $T_C$-direction}
\end{center}
\end{figure}

We can already see that some progress has been made for the quartic discriminant. Indeed, a meromorphic continuation constructed from a single abelian normal subgroup in this way is another way of thinking about the inductive methods of \cite{ALOWW} for concentrated groups. The dashed line represents a possible pole at $s_{2C}=1$ for $D_\Q(D_4;\underline{s})$ in this region, which would be canceled out in the product $p(\underline{s}) D_{\Q}(D_4;\underline{s})$ if a pole really exists here.

The group $T_B$ is completely analogous to $T_C$ (in fact, they are images of each other under an automorphism of $D_4$). By symmetry, the same process gives an analytic continuation of $p(\underline{s})D_\Q(D_4;\underline{s})$ to the region
\[
    \Omega_B = \left\{\underline{s} : \sigma_{2A} > 1/2,\ \sigma_{2B} > 1/2,\ \sigma_{2C} > 1,\ \sigma_{4A} > 19/16,\ \sigma_{2C} + \frac{3}{8}\sigma_{2B} > \frac{11}{8}\right\}.
\]
with a $t$-aspect bound
\[
    |p(\underline{s})D_\Q(D_4,\underline{s})| \ll  (1+|\underline{t}|)^{\frac{1}{3}\max\{1-\sigma_{2A},0\} + \frac{3}{4}\max\{1-\sigma_{2B},0\}+\epsilon}
\]
in vertical strips. We can combine this with $\Omega_C$ to show that $p(\underline{s})D_\Q(D_4;\underline{s})$ has an analytic continuation to the region $\Omega_B \cup \Omega_C$, depicted in \Cref{fig:TCandTB_direction}.  The dashed lines represent possible polar divisors of $D_k(D_4;\underline{s})$ that $p(\underline{s})$ could be canceling out.

\begin{figure}[!ht]
\begin{center}
    \hspace*{3cm} 
    \begin{tikzpicture}[scale=1.5]

    \draw[gray!50, thin, step=1/3] (-1,-1) grid (3,3);
    \draw[very thick,->] (-1,0) -- (3.2,0) node[right] {$\sigma_{2B}$};
    \draw[very thick,->] (0,-1) -- (0,3.2) node[above] {$\sigma_{2C}$};

    \foreach \x in {-1,...,3} \draw (\x,0.05) -- (\x,-0.05) node[below] {\tiny\x};
    \foreach \y in {-1,...,3} \draw (-0.05,\y) -- (0.05,\y) node[right] {\tiny\y};

    \fill[blue!50!cyan,opacity=0.6] (1,3) -- (1,1) -- (3,1) -- (3,3);
    \fill[blue!50!cyan,opacity=0.3] (1/2,3) -- (1/2,19/16) -- (1,1) -- (19/16,1/2) -- (3,1/2) -- (3,3);

    \draw[red,very thick,-] (-1,-1) -- (3,3) node[right] {conductor, octic};
    \draw[red,very thick,-] (-1,-1/2) -- (3,1.5) node[right] {quartic};

    \draw[black,dashed,thick,-] (1,-1) -- (1,3) node[above] {$s_{2B}=1$};
    \draw[black,dashed,thick,-] (-1,1) -- (3,1) node[below right] {$s_{2C}=1$};
    \end{tikzpicture}
    \caption{\label{fig:TCandTB_direction} Continuation in $T_B$- and $T_C$-directions}
\end{center}
\end{figure}
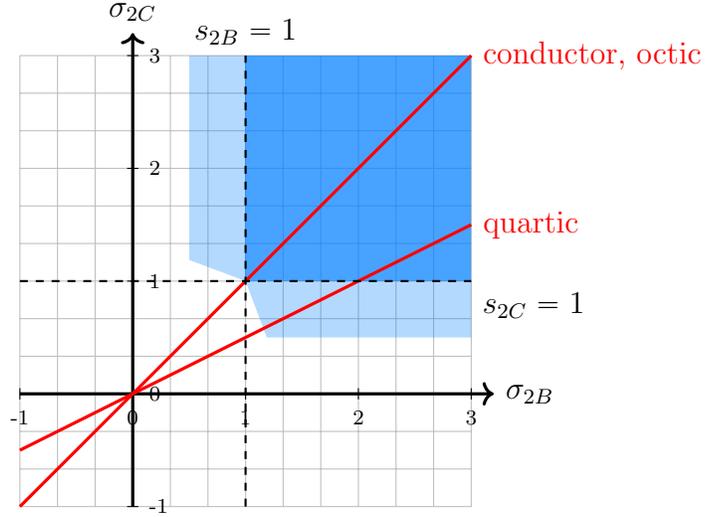

This region is not convex! Bochner's tube theorem now implies that we have an analytic continuation to the convex hull
\[
    \Omega = \{\underline{s} : \sigma_{2A},\sigma_{2B},\sigma_{2C} > 1/2,\ \sigma_{4A} > 19/16,\ \sigma_{2B} + \sigma_{2C} > 27/16\}.
\]
The projection of this region to $(\sigma_{2B},\sigma_{2C})$ is depicted in \Cref{fig:final_continuation}.

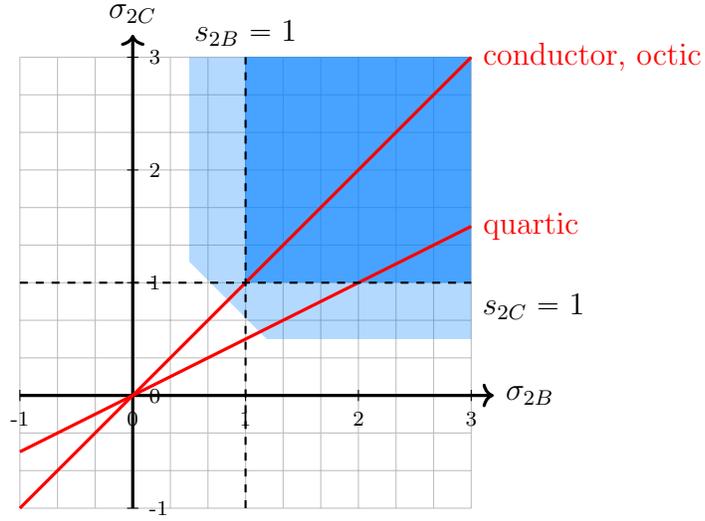
\begin{figure}[!ht]
\begin{center}
    \hspace*{3cm} 
    \begin{tikzpicture}[scale=1.5]

    \draw[gray!50, thin, step=1/3] (-1,-1) grid (3,3);
    \draw[very thick,->] (-1,0) -- (3.2,0) node[right] {$\sigma_{2B}$};
    \draw[very thick,->] (0,-1) -- (0,3.2) node[above] {$\sigma_{2C}$};

    \foreach \x in {-1,...,3} \draw (\x,0.05) -- (\x,-0.05) node[below] {\tiny\x};
    \foreach \y in {-1,...,3} \draw (-0.05,\y) -- (0.05,\y) node[right] {\tiny\y};

    \fill[blue!50!cyan,opacity=0.6] (1,3) -- (1,1) -- (3,1) -- (3,3);
    \fill[blue!50!cyan,opacity=0.3] (1/2,3) -- (1/2,19/16) -- (19/16,1/2) -- (3,1/2) -- (3,3);

    \draw[red,very thick,-] (-1,-1) -- (3,3) node[right] {conductor, octic};
    \draw[red,very thick,-] (-1,-1/2) -- (3,1.5) node[right] {quartic};

    \draw[black,dashed,thick,-] (1,-1) -- (1,3) node[above] {$s_{2B}=1$};
    \draw[black,dashed,thick,-] (-1,1) -- (3,1) node[below right] {$s_{2C}=1$};
    \end{tikzpicture}
    \caption{\label{fig:final_continuation} Continuation to the convex hull}
\end{center}
\end{figure}

Now we can see some meromorphic continuation in the conductor and octic cases. At this stage, we can also explain why our method loses control over the main term: essentially, we have shown that $D_\Q(D_4;\underline{s})$ is analytic on this region except for possible poles at $s_{2A}=1$, $s_{2B}=1$, and $s_{2C}=1$ that could have been canceled out by $p(\underline{s})$. It takes significant extra work to determine the orders of these poles, especially at the intersection points like $(\sigma_{2B},\sigma_{2C}) = (1,1)$.

\Cref{cor:independent_variable_multi_phragmen_lindelof} lets us interpolate the vertical bounds between $\Omega_B$ and $\Omega_C$. By plugging in $\sigma_{2B},\sigma_{2C} > 1/2$ we can give upper bounds depending only on $\sigma_{2A} > 1/2$, so that we conclude
\[
    |p(\underline{s})D_\Q(D_4;\underline{s})| \ll (1+|\underline{t}|)^{\frac{1}{3}\max\{1-\sigma_{2A},0\} + \frac{3}{8} + \epsilon}
\]
on all of $\Omega$, the convex hull of $\Omega_B\cup \Omega_C$. We are now ready to prove Theorem \ref{thm:D4} by specializing to the relevant complex lines.

\subsubsection{Quartic discriminants}

Specializing to the line $\underline{s} = (\ind_4(\tau)s)_{\tau}$ gives a meromorphic continuation for the generating single variable Dirichlet series of quartic $D_4$-extensions to the region
\[
    \{s : 2\sigma,2\sigma,\sigma > 1/2,\ 3\sigma > 19/16,\ 2\sigma + \sigma > 27/16\} = \{s : \sigma > 9/16\}
\]
by specializing $\Omega$ according to the index listed in \Cref{tab:D4_ram_types}, with possible poles at $s=1$ or order $\le 1$ and $s=1/2$ (which is actually outside this region) of order $\le 2$. On this region, The generating series has vertical bounds given by
\[
    \ll (1+|t|)^{\frac{1}{3}\max\{1 - 2\sigma,0\} + \frac{3}{8} + \epsilon} = (1+|t|)^{\frac{3}{8} + \epsilon}.
\]
This is exactly the information needed to apply the power saving Tauberian theorem, \Cref{thm:tauberian_power_saving}.

In our situation, the main terms are given by \cite{cohen-diaz-y-diaz-olivier2002}, $\sigma_a = 1$, $\delta= 7/16$, and $\xi = 3/8$. The power saving exponent is then
\[
    1 - \frac{7/16}{3/8 + 1} + \epsilon = \frac{15}{22} + \epsilon
\]
as claimed.

\subsubsection{Conductor}

Let ${\rm wt}_C$ be the weight function associated to the conductor. Specializing to the line $\underline{s} = ({\rm wt}_C(\tau)s)_{\tau}$ gives a meromorphic continuation for the generating single variable Dirichlet series of conductors of $D_4$-extensions to the region
\[
    \{s : 2\sigma,\sigma,\sigma > 1/2,\ 2\sigma > 19/16,\ \sigma + \sigma > 27/16\} = \{s : \sigma > 27/32\}
\]
by specializing $\Omega$ according to the weight listed in \Cref{tab:D4_ram_types}, with possible poles at $s=1$ or order $\le 2$ and $s=1/2$ (which is actually outside this region) of order $\le 1$. On this region, The generating series has vertical bounds given by
\[
    \ll (1+|t|)^{\frac{1}{3}\max\{1 - 2\sigma,0\} + \frac{3}{8} + \epsilon} = (1+|t|)^{\frac{3}{8} + \epsilon}.
\]
We can now apply \Cref{thm:tauberian_power_saving} where the main terms are given by \cite{altug-shankar-varma-wilson2021,friedrichsen2021secondary} and we may take $\sigma_a = 1$, $\delta= 5/32$, and $\xi = 3/8$. The power saving exponent is then
\[
    1 - \frac{5/32}{3/8 + 1} + \epsilon = \frac{39}{44} + \epsilon
\]
as claimed.

\subsubsection{Octic discriminants}

Specializing to the line $\underline{s} = (\ind_8(\tau)s)_{\tau}$ gives a meromorphic continuation for the generating single variable Dirichlet series of octic $D_4$-extensions to the region
\[
    \{s : 4\sigma,4\sigma,4\sigma > 1/2,\ 6\sigma > 11/16,\ 4\sigma + 4\sigma > 27/16\} = \{s : \sigma > 27/128\}
\]
by specializing $\Omega$ according to the index listed in \Cref{tab:D4_ram_types}, with a possible pole at $s=1/8$ of order $\le 3$. On this region, the generating series has vertical bounds given by
\[
    \ll (1+|t|)^{\frac{1}{3}\max\{1 - 4\sigma,0\} + \frac{3}{8} + \epsilon} \le (1+|t|)^{\frac{41}{96} + \epsilon}.
\]
We can now apply \Cref{thm:tauberian_power_saving} where the main terms are given by \cite{shankar-varma2025octic} and we may take $\sigma_a = 1/4$, $\delta= 5/128$, and $\xi = 41/96$. The power saving exponent is then
\[
    \frac{1}{4} - \frac{5/128}{41/96 + 1} + \epsilon = \frac{61}{274} + \epsilon
\]
as claimed.

\subsection{The main theorem on meromorphic continuation}

Our method of proof closely follows the twisted counting framework developed in \cite{alberts2021,alberts-odorney2021,alberts2024,ALOWW}. Let $T$ be a nontrivial normal subgroup of $G$, and let $q:G\to G/T$ be the corresponding quotient map. This defines a pushforward map
\[
    q_*:\Sur(G_k,G) \to \Sur(G_k,G/T).
\]
This pushforward map was used in \cite{ALOWW} to partition the counting function $\#\Sur(G_k,G;X)$ for the number of surjections with discriminant $\le X$ according to the fibers of $q_*$.

We directly partition $D_k(G;\underline{s})$ in a similar manner.

\begin{definition}
    For each $\pi\in \Sur(G_k,G/T)$, define
    \[
        D_k(T,\pi;\underline{s}) := \sum_{\psi\in q_*^{-1}(\pi)} \inv_{k,G}(\psi,\underline{s})
    \]
    This is a subseries of $D_k(G;\underline{s})$.
\end{definition}

Any normal subgroup $T\normal G$ with quotient map $q:G\to G/T$ then gives rise to the partition
\begin{equation}\label{eq:fibered_sum}
    D_k(G;\underline{s}) = \sum_{\pi\in q_*\Sur(G_k,G)} D_{k}(T,\pi;\underline{s}).
\end{equation}
Our method requires a meromorphic continuation for $D_k(T,\pi;\underline{s})$, uniform upper bounds for $D_k(T,\pi;\underline{s})$ in vertical strips, and a region of absolute convergence for $D_k(G/T,\underline{s})$. The following theorem supposes that we have all three of these ingredients, and uses them to construct a meromorphic continuation of $D_k(G;\underline{s})$.

Given a finite group $G$ and a sequence of proper normal subgroups $T_1,\dots,T_r$, define the function
\[
    p_{G,T_1,...,T_r}(\underline{s}) := \prod_{\tau\in \RTo{k}(G)} \left(\frac{s_\tau - 1}{s_\tau+1}\right)^{b_\tau},
\]
where
\[
    b_\tau = \max\left\{\text{order of the pole }s_\tau = 1\text{ of }D_k(T_j,\pi;\underline{s}) : j=1,2,...,r, \text{ and } \pi\in q_*\Sur(G_k/T_j)\right\}.
\]
We use this to cancel out the polar divisors of $D_k(G;\underline{s})$ when necessary. We remark that it is not immediate that $b_\tau$ is finite, as it is a maximum over an infinite set, but we will see that $b_\tau \le |G|$ when $T_1,\dots,T_r$ are abelian.

\begin{theorem}\label{thm:main}
    Let $k$ be a number field, $G$ a finite group, and $T_1,T_2,...,T_r\normal G$ normal subgroups with corresponding embeddings $\iota_j:T_j\hookrightarrow G$ and quotient maps $q_j:G\to G/T_j$. Suppose $\Omega_1,\Omega_2,\cdots,\Omega_r$ are connected tubular domains with $\bigcup_{j=1}^r \Omega_j$ connected such that for each $(T,\iota,q,\Omega) = (T_j,\iota_j,q_j,\Omega_j)$ we have
    \begin{enumerate}[(1)]
        \item (``Meromorphic continuation of the fibers") For each $\pi \in q_*\Sur(G_k,G)$, the series $D_k(T,\pi;\underline{s})$ has an analytic continuation to $\Omega$, except for possible poles at $s_\tau = 1$ for each tame $\tau \in \RTo{k}(G)$ whose order is bounded independent of $T$ and $\pi$,
        \item (``Uniform upper bounds for the fibers") For each $\pi \in q_*\Sur(G_k,G)$, there is some nonnegative function $f_T(\pi,\underline{s})$ for which,
        \[
            |p_{G,T_1,\dots,T_r}(\underline{s})D_k(T,\pi;\underline{s})| \le f_T(\pi,\underline{s})
        \]
        for each $\underline{s}\in \Omega$, and
        \item (``Criterion for convergence") The series
        \[
            \sum_{\pi\in q_*\Sur(G_k,G)} f_T(\pi,\underline{s})
        \]
        converges uniformly in compact subsets of $\Omega$.
    \end{enumerate}
    Then $p_{G,T_1,\dots,T_r}(\underline{s})D_k(G;\underline{s})$ has a holomorphic continuation to the domain ${\rm Hull}(\bigcup\Omega_j)$.

\end{theorem}

This theorem fits into the inductive philosophy of \cite{ALOWW}, and in fact is directly comparable to \cite[Theorem 2.1]{ALOWW}. If we understand how to count (twisted) $T$-extensions sufficiently well and have sufficiently strong upper bounds for counting $G/T$-extensions, Theorem \ref{thm:main} converts this to a meromorphic continuation for $D_k(G;\underline{s})$. If ${\rm Hull}(\bigcup \Omega_j)$ is sufficiently larger, a proof of Conjecture \ref{conj:number_field_counting} for $G$ over $k$ then follows from a Tauberian theorem.

\begin{remark}
    It is sufficient to replace \Cref{thm:main}(2,3) with the condition that
    \[
        \sum_{\pi\in q_*\Sur(G_k,G)} p_{G,T_1,\dots,T_r}(\underline{s})D_k(T,\pi;\underline{s})
    \]
    converges uniformly in compact subsets of $\Omega$. In practice uniform convergence is proven through upper bounds of the form described in (2,3), and stating \Cref{thm:main} in this way highlights the comparison with \cite[Theorem 2.1]{ALOWW}.
\end{remark}

\begin{proof}[Proof of Theorem \ref{thm:main}]
Let $(T,\iota,q,\Omega) = (T_j,\iota_j,q_j,\Omega_j)$ for some $j$. Part (2) allows us to evaluate the fibered sum in \eqref{eq:fibered_sum} to get
\begin{align*}
    |p_{G,T_1,...,T_r}(\underline{s})D_k(G;\underline{s})| &\ll_{\underline{\sigma}} \sum_{\pi\in q_*\Sur(G_k,G)} f_T(\pi,\underline{s}) < \infty
\end{align*}
uniformly in compact subsets of $\Omega$. This proves that the fibered sum in \eqref{eq:fibered_sum} (times $p_{G,T_1,...,T_r}(\underline{s})$) is uniformly convergent on compact subsets of $\Omega$, thus giving a meromorphic continuation of $p_{G,T_1,...,T_r}(\underline{s})D_k(G;\underline{s})$ to $\Omega$. Condition (1) implies that the poles are necessarily canceled out by $p_{G,T_1,...,T_r}(\underline{s})$.

Overall, this implies $D_k(G;\underline{s})$ has a meromorphic continuation to $\bigcup \Omega_j$ with poles at $s_\tau = 1$ for $\tau\in \bigcup_j (\iota_j)_*\RTo{k}(T_j)$. These poles have order $\le b_\tau$. The union of tubular domains is necessarily a tubular domain, and it is connected by assumption, so Bochner's tube theorem implies $p_{G,T_1,...,T_r}(\underline{s})D_k(G;\underline{s})$ is holomorphic on ${\rm Hull}(\bigcup \Omega_j)$. This proves the meromorphic continuation of $D_k(G;\underline{s})$.
\end{proof}

\section{Ramification Types}\label{sec:ramification_types}

In this section we define ramification types, and develop several of their useful properties. For tame ramification types, our notion is equivalent to numerous existing notions in the literature. Wild ramification types are largely absent from the literature, with the exception being the notion of ``secteroids'' of a stack in characteristic $p$, developed by Darda--Yasuda \cite{darda2025batyrev} to play the role of wild ramification types on stacks.

The notion of ``ramification type'' is the building block for the discriminant and generalizations to other invariants: the power to which a prime $\pk$ divides the discriminant of an extension $F/k$ is determined by the inertia group $I_{\pk}(F/k)$ (and the higher ramification subgroups). The ``ramification type'' of $\pk$ in $F/k$ is a canonical description of the inertia group from which the power of $\pk$ dividing the discriminant can be computed.

We separate the exposition of ramification types into ``tame ramification types" and ``wild ramification types", partitioning the set of $k$-ramification types of a group $G$ as
\[
    \RT_k(G) := \RT_k^{\rm tame}(G) \cup \RT_k^{\rm wild}(G).
\]

\subsection{Tame ramification types}

The goal of defining ``ramification types'' is to give a parametrization of all the ways a prime $\pk$ can ramify with as little reference to the prime $\pk$ itself as possible. This is essentially only achievable for tamely ramified places: the ways in which a prime can tamely ramify depend very little on the prime itself because $I_\pk^{\rm tame}$ is always procyclic of order prime to $\pk$. Heuristically, the ways in which a prime can tamely ramify in a $G$-extension are the same as the choices of a generator $g\in G$ for the tame inertia group of that $G$-extension. This needs to be made independent of the embedding $I_\pk\hookrightarrow G_k$ and the choice of generator $\tau_\pk\in I_\pk^{\rm tame}$.

We capture this idea via $(\Q/\Z)^* = \Hom(\Q/\Z,\mu)$, the Tate dual module of $\Q/\Z$. The module $(\Q/\Z)^*$ is procyclic, does not have a canonical generator, and carries the cyclotomic Galois action just like $I_{\pk}^{\rm tame}$. This suggests that $\Hom(I_\pk^{\rm tame},G)$ should behave ``the same'' as $\Hom((\Q/\Z)^*,G)$. We make this correspondence explicit in Lemma \ref{lem:canonical}, but this is enough to motivate the following definition:

\begin{definition}\label{def:tameRT}
    Let $k$ be a number field and $G$ be a $G_k$-group.
    
    The set of \emph{tame $k$-ramification types of $G$}, denoted $\RT_k^{\rm tame}(G)$, is the set of $G\rtimes G_k$-orbits in $\Hom((\Q/\Z)^*,G)$, where $G$ acts by conjugation on the outputs and $G_k$ acts by the Galois action on $(\Q/\Z)^*$ and $G$. In other words, the action is given by
    \[
        \left((gx).f\right)(a) = g (x.f(x^{-1}.a))g^{-1}.
    \]

    Furthermore, we let $\RTo{k}(G) := \RT_k^{\rm tame}(G) - \{1\}$ denote the set of \emph{nontrivial tame $k$-ramification types of $G$}, where $1$ is the orbit of the trivial homomorphism.
\end{definition}

The semi-direct product $G\rtimes G_k$ is with respect to the Galois action on $G$ as a $G_k$-group, and is chosen precisely so that $(xg).f = x.(g.f)$.

There are a number of different characterizations for tame ramification types in the literature, usually referred to simply as ``ramification types" as the researchers using them are typically able to avoid considering wild primes. The reader may recognize some features of these definitions here, for instance this presentation agrees the presentation just below \cite[Theorem 3.1]{alberts2024} with $G$ is abelian, and the $G\times G_k$-action is reminiscent of \cite[Definition 2.1]{gundlach2022multimalle} for groups with the trivial action. In fact, these and several other existing notions of tame ramification type are all equivalent. For the interested reader, we give proofs of the various equivalencies in \Cref{app:tame_ramification_types}.

For now, we give a short argument justifying that \Cref{def:tameRT} agrees (noncanonically) with Malle's notion of $k$-conjugacy classes \cite{malle2004}, which play the role of tame ramification types in Malle's original conjecture. Using the description as a hom-set carries the benefit of a canonical correspondence to ramification (see Lemma \ref{lem:canonical}) and access to cohomological constructions like pushforward maps. In several examples, we use the LMFDB labels for $\Q$-conjugacy classes interchangeably with tame $\Q$-ramification types.

\begin{proof}[Proof of the correspondence between tame $k$-ramification types and $k$-conjugacy classes]
Let $G$ be a constant group (i.e.~with trivial Galois action) and fix a generator for $a\in (\Q/\Z)^*$, which is equivalent to fixing an isomorphism with $\widehat{\Z}$. The map $f\mapsto f(a)$ is a bijection
\[
    \Hom((\Q/\Z)^*,G) \to G.
\]
The action by $G\times G_k$ passed through this bijection is in fact a combination of $G$ acting by conjugation and $G_k$ acting by the cyclotomic action, that is $x.g = g^{\chi(x)^{-1}}$ for $\chi:G_k\to \Gal(k\Q^{ab}/k) \subseteq \widehat{\Z}^\times$ the cyclotomic character. Malle defined a $k$-conjugacy class to be precisely a minimal subset of $G$ closed under conjugation and the cyclotomic action, so that these are in bijection with $G\times G_k$-orbits of $\Hom((\Q/\Z)^*,G)$. This shows that our notion of ramification type agrees with Malle's original notion, even if the correspondence described here depends on a choice of generator for $(\Q/\Z)^*$.
\end{proof}

The true purpose of Definition \ref{def:tameRT} is to give a canonical description of $\Hom(I_\pk^{\rm tame},G)$ for all tame places with as little reference to $\pk$ as possible. This is achieved by the following lemma:

\begin{lemma}\label{lem:canonical}
    Let $k$ be a number field and $\pk$ a nonarchimedean place of $k$. Let $\iota_\pk:(\Q/\Z)^* \to I_\pk^{\rm tame}$ be the composition
    \[
    \begin{tikzcd}
        (\Q/\Z)^* \rar{(q_e)} &{\displaystyle\lim_{\substack{\longleftarrow\\ (e,\pk) = 1}} \mu_e} \rar{\sim} &{\displaystyle\lim_{\substack{\longleftarrow\\ (e,\pk) = 1}} I_\pk/I_\pk^e} \rar[equals] &I_\pk^{\rm tame}
    \end{tikzcd}
    \]
    with the first map determined by homomorphisms $q_e:(\Q/\Z)^*\to \mu_e$ defined by $q_e(a) = a(1/e)$ and the second map is given by the inverse of $g\mapsto g(\pi_\pk)/\pi_\pk \mod \pi_\pk$ for $\pi_\pk\in k_\pk(\mu_e)$ a uniformizer.

    \begin{enumerate}[(i)]
        \item $\iota_\pk$ is a surjective homomorphism which is independent of the choice of uniformizer, inducing an isomorphism between the prime-to-$\pk$ part of $(\Q/\Z)^*$ and $I_\pk^{\rm tame}$.
        \item If $G$ is a finite $G_k$-group with action $\phi:G_k\to \Aut(G)$ and $\pk \nmid |G|\infty$ with $\phi(I_\pk)=1$, then $\iota_\pk$ induces a $G_{k_\pk}$-equivariant bijection
        \[
        \begin{tikzcd}
            \Hom(I_\pk,G) = \Hom(I_{\pk}^{\rm tame},G) \rar{\iota_\pk^*} & \Hom((\Q/\Z)^*,G),
        \end{tikzcd}
        \]
        which is an isomorphism of groups in case $G$ is abelian.
    \end{enumerate}
\end{lemma}

Lemma \ref{lem:canonical}(i) implies that $\iota_\pk$ is a canonical homomorphism. In particular, we did not need to choose a generator of $I_\pk^{\rm tame}$ or $(\Q/\Z)^*$ in order to define it. Thus, for tame places $\pk$ which are unramified in the action on $G$, Lemma \ref{lem:canonical}(ii) shows that $\iota_\pk^*$ is a canonical $G\times G_{k_\pk}$-equivariant bijection
\[
    \Hom(I_\pk,G) \to \Hom((\Q/\Z)^*,G).
\]
The $G$-orbits of $\Hom(-,G)$ are precisely $H^1(-,G)$, so the map to tame ramification types factors through the coboundary relation. The inflation-restriction sequence implies that $\res_{I_\pk}(H^1(k_\pk,G)) = H^1(I_\pk,G)^{G_{k_\pk}}$, so that $\iota_\pk$ also induces a bijection
\[
    H^1(I_\pk,G)^{G_{k_\pk}} \to H^1((\Q/\Z)^*,G)^{G_{k_\pk}} = \{ [f]\in H^1((\Q/\Z)^*,G) : \Fr_\pk.f = f\}.
\]
Thus, the possible tame ramification types at a particular tame place $\pk$ are in canonical bijection with the Galois orbits of $\Fr_\pk$-fixed points of $H^1((\Q/\Z)^*,G)$.

\begin{proof}[Proof of Lemma \ref{lem:canonical}]
    The homomorphisms $q_e$ are each surjective and $G_k$-equivariant, so that the homomorphism $(q_e)$ to the inverse limit necessarily is as well. The isomorphism $g\mapsto g(\pi_\pk)/\pi_\pk \mod \pi_\pk$ is the standard one used to classify the structure of $I_\pk$, which is known to be independent of the choice of uniformizer \cite[Chapter IV Proposition 7]{serre1979}. This proves part (i). Part (ii) is then immediate from $\pk\nmid |G|\infty$ and $\phi(I_\pk) = 1$.
\end{proof}

Each tame ramification type carries useful internal structure, which we will need to access during our proofs in order to get finer control over the behavior of $D_k(G;\underline{s})$.

\begin{definition}\label{def:rt_internal_structure}
    Let $G$ be a $G_k$-group and $\tau\in \RT_k^{\rm tame}(G)$.
    \begin{enumerate}[(i)]
        \item The \emph{order of $\tau$}, denoted ${\rm ord}(\tau)$, is the unique integer such that $|\im f| = {\rm ord}(\tau)$ for each $f\in \tau$.
        \item We always let $\zeta_\tau$ denote a primitive root of unity of order ${\rm ord}(\tau)$.
        \item The \emph{transitive group associated to $\tau$} is given by the embedding
        \[
            \rho_\tau:(G\rtimes G_k)/C_{G\rtimes G_k}(\tau) \to \Sym(\tau),
        \]
        where $C_{G\rtimes G_k}(\tau)$ is the centralizer, i.e. the set of elements acting trivially on $\tau$.
    \end{enumerate}
\end{definition}

We briefly justify that the order is well-defined: Choose a topological generator $a\in (\Q/\Z)^*$. Then $\im f = \langle f(a)\rangle$. The size of a subgroup of $G$ is invariant under any automorphism, which means it is invariant under conjugation and the Galois action. Moreover, the subgroup itself is invariant under taking invertible powers of the generator, i.e.~is invariant under the cyclotomic action. Thus, we have shown $|\im f|$ is invariant under the $G\rtimes G_k$ action.

These are familiar group theoretic invariants, which become recognizable after passing through the bijection with $k$-conjugacy classes: Suppose $G$ carries the trivial $G_k$-action and let $\tau$ be a $k$-conjugacy class of $G$.
\begin{itemize}
    \item  The order ${\rm ord}(\tau)$ is equal to the order of any element of $g\in \tau$ as a group element, that is the smallest positive integer $n$ for which $g^n = 1$.
    \item The transitive group associated to $\tau$ is precisely $G/C_G(\tau) \times \Gal(k(\zeta_\tau)/k)$ acting on $\tau$ by conjugation (in the left coordinate) and the cyclotomic action (in the right coordinate).
\end{itemize}

\subsection{Wild ramification types}

The specialization in Proposition \ref{prop:specialize} can be made seamless by the inclusion of separate variables for each way a prime can wildly ramify. This makes it possible for $D_k(G;\underline{s})$ to specialize to the generating series of discriminants, rather than just the tame part of discriminant.

\begin{definition}\label{def:wildRT}
    Let $k$ be a number field, $G$ be a $G_k$-group with action $\phi:G_K\to \Aut(G)$, and $S$ a finite set of places of $k$ containing $\{\pk : \pk\mid |G|\infty\text{ or }\phi(I_\pk)\ne 1\}$.
    
    The \emph{$S$-wild $k$-ramification types of $G$} are given by the disjoint union
    \[
        \RT_{k}^{S, {\rm {wild}}}(G) := \coprod_{\pk\in S} \{[f]\in H^1(I_\pk,G) : f(I_\pk)\ne 1\}.
    \]
    We write elements of the disjoint union as pairs $(\pk,[f])$ for $[f]\in H^1(I_\pk,G)$, in order to be clear about which wild place is involved.

    If $S = \{\pk : \pk\mid |G|\infty\text{ or }\phi(I_\pk)\ne 1\}$, we drop the $S$ from the notation and call $\RT_k^{\rm wild}(G)$ the wild $k$-ramification types of $G$.
\end{definition}

We require that the infinite places and places ramified in the Galois action be treated as ``wild places" because $H^1(I_\pk,G)$ behaves differently as a $G_k$-set for these places. The parameter $S$ allows us to include finitely many additional places that may behave differently, which we analogously call $S$-wild. We will be mostly unconcerned with the structure of wild ramification types. The fact that a finite $G_k$-group $G$ comes with only finitely many wild ramification types and that only finitely many primes can be wildly ramified is sufficient information for our purposes, as changing finitely many Euler factors does not change the region of meromorphic continuation or the bounds coming from it.

We collect some notations here concerning ramification types that we will use throughout the paper:

\begin{definition}\label{def:RT}
    The set of $(S,k)$-ramification types
    \[
        \RT_k^S(G) := \RT_{k}^{\rm tame}(G) \cup \RT_k^{S,\rm wild}(G)
    \]
    and the nontrivial $(S,k)$-ramification types
    \[
        \RT_{k,0}^S(G) := \RTo{k}(G) \cup \RT_k^{S,\rm wild}(G).
    \]
    We drop the $S$ from the notation if $S = \{\pk : \pk\mid |G|\infty\text{ or }\phi(I_\pk)\ne 1\}$.
\end{definition}

\subsection{Assigning ramification types}

Now that we have a notion of ramification types, we need a way to identify the ramification type of a particular crossed homomorphism $\pi\in Z^1(k,G)$. We will often refer to disjoint unions of the form
\[
    \coprod_\pk Z^1(I_\pk,G).
\]
When doing so, we write the elements as pairs $(\pk,f)$ for $f\in Z^1(I_\pk,G)$ in order to make it clear which place is being considered.

\begin{definition}\label{def:RTmap}
    Let $k$ be a number field and $G$ a $G_k$-group. Define the \emph{ramification map}
    \[
        \RT^S:\coprod_{\pk} Z^1(I_\pk,G)\to \RT_{k}^S(G)
    \]
    as follows:
    \begin{enumerate}[(a)]
        \item if $\pk\not\in S$, then $\RT^S(\pk,f) = \iota_\pk^*f \in \Hom((\Q/\Z)^*,G)$,
        \item if $\pk\in S$, then $\RT^S(\pk,f) = (\pk,[f])$.
    \end{enumerate}
    Moreover, we abuse notation and write
    \[
        \RT^S:\{\text{places of }k\} \times Z^1(k,G) \to \RT_{k}^S(G)
    \]
    for the function $\RT^S(\pk,\pi) := \RT^S(\pk,\pi|_{I_\pk})$ for some embedding $I_\pk\hookrightarrow G_k$.

    Once again, we drop $S$ from the notation if $S = \{\pk : \pk\mid |G|\infty\text{ or }\phi(I_\pk)\ne 1\}$.
\end{definition}

This describes how the ramification type of $\pi$ at a place $\pk$ is identified. The value of $\RT^S(\pk,\pi)$ is necessarily independent of the choice of embedding $I_\pk\hookrightarrow G_k$, as any two embeddings are conjugate (i.e.~in the same $G$-orbit). 

All but finitely many primes can be at most tamely ramified in a $G$-extension. As a consequence, the analytic behavior of $D_k(G;\underline{s})$ in the tame variables will be the most relevant for Conjecture \ref{conj:number_field_counting}. Gundlach expresses this in terms of a system of basic invariants $(\inv_\tau)$ in \cite{gundlach2022multimalle}, which are the numbers
\[
    \inv_\tau(\pi) = \prod_{\RT(\pk,\pi) = \tau} |\pk|.
\]
The tame part of $\inv_{k,G}(\pi,\underline{s})$ is given precisely by
\[
    \prod_{\tau\in\RTo{k}(G)} \inv_\tau(\pi)^{-s_\tau}.
\]
Thus, our multi-invariant agrees with Gundlach's at all tame places. As the finitely many wild places are not expected to have a significant effect on this picture, our generating multiple Dirichlet series $D_k(G;\underline{s})$ can be reasonably understood to behave similarly to Gundlach's. As only finitely many primes can be wildly ramified, one heuristically expects that $D_k(G;\underline{s})$ is an entire function in the wild variables. For all regions of meromorphicity for $D_k(G;\underline{s})$ we construct in this paper we will see that the singularities are cut-out by equations in the tame variables only.

\subsection{Properties of ramification types}

The order to which a tamely ramified prime $\pk$ divides the discriminant of a $G$-extension is given by Malle's index function \cite{malle2002,malle2004}. We show that the index function is well-defined on ramification types, so that we can define corresponding weights to the discriminant as in Proposition \ref{prop:specialize}.

\begin{lemma}
    Let $G$ be a transitive group of degree $n$ and $\ind_n:S_n\to \Z$ be the index function defined by $\ind_n(g) = n - \#\{\text{orbits of }g\}$. Define the map
    \[
        \ind_n:\Hom((\Q/\Z)^*,G) \to \Z
    \]
    given by $\ind_n(f) = \ind_n(g)$ for some generator $g\in \im f$.

    This map is independent of the choice of generator and factors through the $G\times G_k$-orbits.
\end{lemma}

If we take the weights
\[
    {\rm wt}(\tau) = \begin{cases}
        \ind_n(\tau) & \tau\in \RTo{k}(G)\\
        \nu_\pk(\disc_G(f)) & \tau=(\pk,[f])\in \RT_k^{\rm wild}(G),
    \end{cases}
\]
it follows immediately that $\inv_{k,G}(\pi,({\rm wt}(\tau)s)_{\tau}) = |\disc(\pi)|^{-s}$. To simplify notation, we extend the definition of the index function to wild places by setting
\[
    \ind_n(\tau) := \nu_\pk(\disc_G(f))
\]
when $\tau=(\pk,[f])\in \RT_k^{\rm wild}(G)$.

\begin{proof}
    The number $\ind_n(g)$ is determined solely by the cycle type of $g\in S_n$. If $g,h$ generate the same cyclic subgroup in $G\le S_n$, then $g$ and $h$ necessarily have the same cycle type. Thus, the index is independent of the choice of generator.

    In fact, the $G_k$-action factors through $\hat{\Z}^\times$ via the cyclotomic character. As $\hat{\Z}^\times$ acts on $G$ by invertible powers, it necessarily preserves the cyclic subgroups of $G$. Thus, the index function also factors through the $G_k$-orbits.
    
    Lastly, cycle types are invariant under conjugation. This is precisely the $G$-action on $\Hom((\Q/\Z)^*,G)$, so it necessarily follows that the index factors through the coboundary relation as well.
\end{proof}

Given that our method leverages the fibers of $G$-extensions containing a particular $G/T$-extension for some normal subgroup $T\normal G$, we will need to describe the relationship between ramification types of $T$, $G$, and $G/T$.

Given a Galois equivariant group homomorphism $q:G\to H$, we can define the pushforward of ramification types. The main idea here is that there is a pushforward map
\[
    q_*:\Hom(-,G) \to \Hom(-,H),
\]
which respects the structure used to define ramification types.

\begin{definition}\label{def:pushforward_of_RT}
    Let $G$ and $H$ be $G_k$-groups, and $q:G\to H$ a Galois equivariant homomorphism. Let
    \[
        q_*:\Hom((\Q/\Z)^*,G) \to \Hom((\Q/\Z)^*,H)
    \]
    be the pushforward map. Define the pushforward on ramification types
    \[
        q_*:\RT^S_{k}(G) \to \RT^S_{k}(H)
    \]
    as follows:
    \begin{itemize}
        \item If $\tau\in \RT_{k}^{\rm tame}(G)$ is tame and $f\in \tau$, then $q_*\tau$ is defined to be the Galois orbit of $q_*f\in H^1((\Q/\Z)^*,H)$.
        \item If $\tau = (\pk,[f])\in \RT_{k}^{S, {\rm wild}}(G)$ is wild, define $q_*\tau$ to be the wild ramification type $(\pk,q_*[f])\in \RT_{k}^{S, {\rm wild}}(H)$.
    \end{itemize}
\end{definition}

The pushforward is a $q:G\rtimes G_k\to H\rtimes G_k$ equivariant map, as $q$ is Galois equivariant and
\begin{align*}
    q_*((gx).f)(a) &= q(g.x.f(x^{-1}.a))\\
    &= q(g(x.f(x^{-1}.a)g^{-1})\\
    &=q(g)(x.q(f(x^{-1}.a))) q(g)^{-1}\\
    &= ((q(g)x).q_*f)(a).
\end{align*}
This is sufficient to conclude that the pushforward of tame ramification types is well-defined.

We also need to define a twisting of ramification types. The idea is that if $T\normal G$, then the ramification types of the $G_k$-group $T(\pi)$ twisted by $\pi\in \Sur(G_k,G)$ should induce ramification types of $G$ whose image is contained in $T$.

\begin{definition}\label{def:twist_of_RT}
    Let $G$ be a $G_k$-group and $T\normal G$ a normal $G_k$-equivariant subgroup. Let $\pi\in Z^1(k,G)$ be a crossed homomorphism, and define $T(\pi)$ to be the $G_k$-group with twisted action $x:t \mapsto \pi(x)(x.t)\pi(x)^{-1}$.

    Define the twist of ramification types
    \[
        \tw_{T,\pi}:\RT_{k}^S(T(\pi)) \to \RT_{k}^S(G)
    \]
    as follows:
    \begin{itemize}
        \item If $\tau\in \RT_{k}^{\rm tame}(T(\pi))$ and $f\in \tau\subseteq \Hom((\Q/\Z)^*,T(\pi))$, then $\tw_{T,\pi}(\tau)$ is defined to be the $G\times G_k$-orbit of $\iota_* f\in \Hom((\Q/\Z)^*,G)$ through the embedding $\iota:T\hookrightarrow G$.
        \item If $\tau = (\pk,[f])\in \RT_{k}^{S, {\rm wild}}(T(\pi))$ is wild, define $\tw_{T,\pi}(\tau)$ to be the wild ramification type $(\pk,[f*\pi|_{I_\pk}])$.
    \end{itemize}
    In each of the above, $f*\pi$ is the pointwise product of functions $(f*\pi)(x) = f(x)\pi(x)$.
\end{definition}

The pointwise product of functions $f\mapsto f*\pi$ is shown in \cite[Lemma 1.3]{alberts2021} to induce a bijection $Z^1(k,T(\pi))\to \{ \psi\in Z^1(k,G) : q_*\psi = \pi\}$ which respects the coboundary relation. This implies the twist of wild ramification types is well-defined.

To show that the twist of tame ramification types is well-defined, it will be sufficient to show that it is in fact an equivariant map:

\begin{lemma}
    Let $G$ be a $G_k$-group and $T\normal G$ a $G_k$-equivariant subgroup. Define a map $\phi_\pi:T(\pi)\rtimes G_k \to G\rtimes G_k$ by
    \[
        (t,x) \mapsto (t\pi(x),x).
    \]
    Then
    \begin{enumerate}[(a)]
        \item $\phi_\pi$ is an injective group homomorphism;
        \item $\tw_{T,\pi}$ is a $\phi_\pi$-equivariant map, that is
        \[
            \tw_{T,\pi}((t,x).f) = \phi_\pi(t,x).\tw_{T,\pi}(f).
        \]
    \end{enumerate}
\end{lemma}

\begin{proof}
    For part (a) injectivity is clear, while being a group homomorphism follows from direct computation with the semi-direct product structures:
    \begin{align*}
        \phi_\pi((t_1,x_1)(t_2,x_2)) &= \phi_\pi(t_1\pi(x_1)(x_1.t_2)\pi(x_1)^{-1}, x_1x_2)\\
        &= (t_1\pi(x_1)(x_1.t_2)\pi(x_1)^{-1} \pi(x_1x_2),x_1x_2)\\
        &= (t_1\pi(x_1)(x_1.t_2)(x_1.\pi(x_2)),x_1x_2)\\
        &= (t_1\pi(x_1),x_1)(t_2\pi(x_2),x_2).
    \end{align*}
    Part (b) also follows from direct computation from the two actions on $f$, as
    \begin{align*}
        \tw_{T,\pi}((t,x).f) &= \tw_{T,\pi}(t\pi(x)(x.f)\pi(x)^{-1}t^{-1})\\
        &=\iota_*(t\pi(x)(x.f)\pi(x)^{-1}t^{-1})\\
        &=t\pi(x).(x.\iota_*f)\\
        &= (t\pi(x),x).\iota_*f\\
        &= \phi_\pi(t,x).\tw_{T,\pi}(f).
    \end{align*}
\end{proof}

Despite how they appear in the definition, the tame ramification types $\tau \in \RT_{k}^{\rm tame}(T(\pi))$ and $\tw_{T,\pi}(\tau) \in \RT_k^{\rm tame}(G)$ are not exactly the same. Essentially, $\tw_{T,\pi}$ need not be injective so that the tame ramification types of $T(\pi)$ may be more numerous than those of $G$.

This can happen if the subgroup $\phi_\pi\left(T(\pi)\rtimes G_k\right)$ in $G\times G_k$ has smaller orbits than all of $G\times G_k$. Put another way, let $\chi:G_k\to \hat{\Z}^{\times}$ be the cyclotomic map and denote $\chi^{-1}(x):=\chi(x)^{-1}$. This issue occurs when
\[
    (\pi*\chi^{-1})(G_k) \ne \pi(G_k)\chi^{-1}(G_k).
\]
The ``overlap" between the actions $\pi$ and $\chi$ is equivalent to the field fixed by $\ker\pi$ containing certain extra roots of unity compared to $k$. This is closely related to Kl\"uners counter example to Malle's original conjecture \cite{kluners2005}, as having more ramification types than expected ultimately leads to a larger power of $\log X$ in the asymptotic. See \cite[Section 3.4]{alberts2021} for further discussion of this phenomenon.

Despite the potential for overlap depending on how $\pi$ interacts with $\chi$, the \emph{image} of $\tw_{T,\pi}$ on tame ramification types does not actually depend on $\pi$. We will require the following lemma characterizing both the image and the fibers.

\begin{corollary}\label{cor:rt_twist_structure}
    Let $G$ be a $G_k$-group, $T\normal G$ a normal $G_k$-equivariant subgroup with embedding $\iota:T\hookrightarrow G$, and $\pi\in Z^1(G_k,G)$ a crossed homomorphism.
    \begin{enumerate}[(a)]
        \item $\tw_{T,\pi}\left(\RT_{k}^{\rm tame}(T(\pi))\right) = \iota_*\RT_{k}^{\rm tame}(T)$,
        \item for each $\widetilde{\tau}\in \tw_{T,\pi}^{-1}(\tau)$, ${\rm ord}(\widetilde{\tau}) = {\rm ord}(\tau)$,
        \item Suppose further that $\pi(G_k)T=G$. Then $\#\tw_{T,\pi}^{-1}(\tau) = \frac{|\tau|}{|\widetilde{\tau}|}$ for each $\widetilde{\tau}\in \tw_{T,\pi}^{-1}(\tau)$.
    \end{enumerate}
\end{corollary}

\begin{proof}[Proof of Corollary \ref{cor:rt_twist_structure}]
    Part (a) follows from the definition $\tw_{T,\pi}(f) = \iota_*f$.

    Choose some $f\in \widetilde{\tau}$. Then $|\im f| = |\im \iota_*f|$, where $\iota_*f\in \tau$. This implies ${\rm ord}(\widetilde{\tau}) = {\rm ord}(\tau)$.

    We now consider part (c). The twisting map being $\phi_\pi$-equivariant implies that the orbits in $\iota_*\tw_{T,\pi}^{-1}(\tau)$ partition $\tau$, that is
    \[
        \tau = \coprod_{\widetilde{\tau}\in \tw_{T,\pi}^{-1}(\tau)} \iota_*\widetilde{\tau}.
    \]
    Given that $\pi(G_k)T = G$, it follows that for every $g\in G$ there exists an $x\in G_k$ for which $gx \in \phi_\pi(T(\pi)\rtimes G_k)$. In particular, $\phi_\pi(T(\pi)\rtimes G_k) G_k = G\rtimes G_k$, which implies that $G_k$ acts transitively on $\{\iota_*\widetilde{\tau} : \widetilde{\tau}\in \tw^{-1}_{T,\pi}(\tau)\}$ by the cyclotomic action. Given that the cyclotomic action commutes with conjugation, it necessarily follows that the cyclotomic action further induces a bijection between $\iota_*\widetilde{\tau}$ and $x.(\iota_*\widetilde{\tau})$ as subsets of $\Hom((\Q/\Z)^*,G)$. This implies $|\iota_*\widetilde{\tau}|=|\widetilde{\tau}|$ is independent of the choice of $\widetilde{\tau}\in \tw_{T,\pi}^{-1}(\tau)$, which implies the result.
\end{proof}

\section{The Fiber Series}\label{sec:fiber_series}

The goal of this section is to determine a region $\Omega_T$ for each abelian normal subgroup $T\normal G$ for which $(T,\iota,q,\Omega_T)$ satisfies the conditions laid out in Theorem \ref{thm:main}. This is done in the most general form by Theorem \ref{thm:meromorphic_continuation_of_fiber_series}. We prove this result by following the ``twisted number field counting'' program of study in \cite{alberts2021,alberts-odorney2021,alberts2024,ALOWW}, where the terms in the fiber series $D_k(T,\pi;\underline{s})$ are parametrized by crossed homomorphisms.

The results of this section are the most technical of the paper. The criterion for convergence in Theorem \ref{thm:main} requires uniform control of the fiber series $D_k(T,\pi;\underline{s})$ to the left of its first polar divisors, a significantly stronger type of uniformity than required in \cite{ALOWW}. More than this, our method really proves something new when the convex hull taken in Theorem \ref{thm:main} is nontrivial, that is when $\bigcup \Omega_{T}$ is non-convex to start with. This is a delicate condition, and is a hurdle which many of our unconditional examples only barely clear.

For this reason, we find that it is beneficial to have the most complete information about the regions $\Omega_T$ in the most general setting. We prove several simplifications in \Cref{sec:unconditional} that make some sacrifices in the strength of the error bounds in exchange for making the scope of our results clearer, notably including shortcuts for determining when the convex hull contains the poles of interest. Readers looking for statements to use solely in applications might refer to those results in lieu of Theorem \ref{thm:meromorphic_continuation_of_fiber_series}.

\subsection{Meromorphic continuation and vertical bounds}

The primary purpose of this section is to give an explicit meromorphic continuation of the fiber series $D_k(T,\pi;\underline{s})$. We summarize the features we will use in the following result.

\begin{theorem}\label{thm:meromorphic_continuation_of_fiber_series}
Let $k$ be a number field, $G$ a finite group, $T\normal G$ abelian with quotient map $q:G\to G/T$ and inclusion map $\iota:T\to G$, and $\pi\in q_*\Sur(G_k,G)$.
\begin{enumerate}[(i)]
    \item The fiber series $D_k(T,\pi;\underline{s})$ has a meromorphic continuation to the region
    \[
        \{\underline{s} : \text{if }\tau\in \iota_*\RTo{k}(T)\text{ then }\sigma_\tau > 1/2\},
    \]
    with possible poles at $s_\tau = 1$ of order $\le |\tw_{T,\widetilde{\pi}}^{-1}(\tau)|$ for each $\tau\in \iota_*\RTo{k}(T)$ and any choice of lift $\widetilde{\pi}\in q_*^{-1}(\pi)$.
    \item Moreover, for each $\epsilon,\delta > 0$ we have the bound
    \[
        |D_k(T,\pi;\underline{s})| \ll |H^1_{ur}(k,T(\pi))| \cdot \inv_{k,G/T}(q_*\underline{\sigma}+\epsilon) \cdot \prod_{\tau\in \iota_*\RTo{k}(T)}\left(\prod_{\widetilde{\tau}\in\tw_{T,\widetilde{\pi}}^{-1}(\tau)}\max_{\chi}\left\{|L_{F(\pi,\widetilde{\tau})}(s_\tau,\chi)|,1\right\}\right)
    \]
    for each $\underline{s}$ in the region cut out by $\sigma_\tau > 1/2+\delta$ for each $\tau\in \iota_*\RTo{k}(T)$, where
    \begin{enumerate}[(a)]
        \item $q_*\underline{\sigma}$ is the tuple of variables indexed by $\overline{\tau}\in \RT_k^{\rm tame}(G/T)$ and defined by
        \[
            (q_*\underline{\sigma})_{\overline{\tau}} := \min_{\tau\in q_*^{-1}(\overline{\tau})}\sigma_\tau,
        \]
        \item $\widetilde{\pi}\in q_*^{-1}(\pi)$ is any choice of lift,
        \item $F(\pi,\widetilde{\tau})$ is the field fixed by a stabilizer of the Galois action on $\widetilde{\tau}$ (defined up to conjugation), via the $G_k$-action on $\Hom((\Q/\Z)^*,T(\pi))$,
        \item the maximum is taken over Hecke characters $\chi$ over $F(\pi,\widetilde{\tau})$ of order dividing $\exp(T)$ which are unramified away from the primes dividing $|T|\infty$, and
        \item the implied constant depends only on $[k:\Q]$, $|T|$, $\epsilon$, and $\delta$.
    \end{enumerate}
\end{enumerate}
\end{theorem}

We then use Theorem \ref{thm:meromorphic_continuation_of_fiber_series} to produce the required inputs for Theorem \ref{thm:main}. Using subconvexity bounds for the Hecke $L$-functions, we construct a subdomain $\Omega_T$ of the region of meromorphicity for $D_k(T,\pi;\underline{s})$ on which the criterion for convergence holds for the sum over all $\pi$.

The proof will take several steps. We recall the basics of the ``twisted number field counting framework'' in Subsection \ref{subsec:param_fibers} and show how this can be modified to express the fiber series as a sum over crossed homomorphisms. Then, we prove \Cref{thm:explicit_continuation} which gives an explicit meromorphic continuation for the generating multiple Dirichlet series of a Galois module. This result is essentially a (slightly more explicit) multivariable version of the work in \cite[Corollary 3.4]{alberts2024}. Finally, we go back from the results for Galois modules to the fiber series $D_k(T,\pi;\underline{s})$ to prove Theorem \ref{thm:meromorphic_continuation_of_fiber_series}.

\subsection{Parametrizing the fibers}\label{subsec:param_fibers}

Let $T\normal G$ be a normal subgroup with quotient map $q:G\to G/T$. As done in \cite{ALOWW}, we partition according to the fibers of the pushforward $q_*:\Sur(G_k,G) \to \Sur(G_k,G/T)$ as in \eqref{eq:fibered_sum}. In this section we give a description for the generating series of the fibers, $D_k(T,\pi;\underline{s})$, in terms of cohomology following the framework of \cite{alberts2021}.
    
It was proven in \cite[Lemma 1.3]{alberts2021} that the map $f\mapsto f*\widetilde{\pi}$ is a bijection from
\[
    \{f\in Z^1(k,T(\widetilde{\pi})) : f*\widetilde{\pi}\text{ surjective}\} \to q_*^{-1}(\pi),
\]
where $\widetilde{\pi}:G_k\to G$ is any choice of lift for $\pi$ (i.e. $\widetilde{\pi}\in q_*^{-1}(\pi)$), $T(\widetilde{\pi})$ is the group $T$ together with the Galois action $x.g = \widetilde{\pi}(x) g\widetilde{\pi}(x)^{-1}$, and $f*\widetilde{\pi}:G_k\to G$ is the pointwise product $(f*\widetilde{\pi})(x) = f(x)\widetilde{\pi}(x)$.

In the case that $T$ is abelian, the conjugation action on $T$ factors through the quotient $G/T$. For this reason, we will often abuse notation and write $T(\pi)$ for $T(\widetilde{\pi})$ as the Galois module depends only on $\pi$.

It is then immediate that
\[
    D_k(T,\pi;\underline{s}) = \sum_{\substack{f\in Z^1(k,T(\pi))\\ f*\widetilde{\pi}\text{ surjective}}} \inv_{k,G}(f*\widetilde{\pi},\underline{s}).
\]
While this expression depends on the choice of $\widetilde{\pi}$, we recall that the series $D_k(T,\pi;\underline{s})$ does not.

It will be convenient to drop the surjectivity condition and, as much as possible, remove the dependence on $\widetilde{\pi}$. We do this by defining a generating multiple Dirichlet series for $G_k$-groups in analogy with Definition \ref{def:DkGs}.

\begin{definition}
    Let $k$ be a number field, $G$ a $G_k$-group, and $S$ a finite set of places of $k$ containing $\{\pk : \pk\mid |G|\infty\text{ or }\pk\text{ ramifies in the }G_k\text{-action}\}$. For any $f\in Z^1(k,G)$, define
    \[
        \inv^S_{k,G}(f,\underline{s}) = \prod_{\tau\in \RT_{k,0}^S(G)} \prod_{\RT^S(\pk,f)=\tau} |\pk|^{-s_\tau},
    \]
    where $|\pk|$ is the norm of $\pk$ down to $\Q$ and $\underline{s} = (s_\tau)$ is a tuple of variables indexed by $(S,k)$-ramification types.

    We then define
    \[
        \Detale{S}{k}(G;\underline{s}) = \sum_{f\in Z^1(k,G)} \inv^S_{k,G}(f,\underline{s}).
    \]
\end{definition}

This series has nicer arithmetic structure than $D_k(T,\pi;\underline{s})$, and will be easier to study directly. Using an inclusion-exclusion argument, we can relate $D_k(T,\pi;\underline{s})$ to $\Detale{S}{k}(T(\widetilde{\pi});\underline{s})$.

\begin{proposition}\label{prop:mobius_inversion}
    Let $k$ be a number field, $G$ a finite group, $T\normal G$ with quotient map $q:G\to G/T$. Further, let $\pi\in q_*\Sur(G_k,G)$ and $\widetilde{\pi}\in q_*^{-1}(\pi)$. Then
    \begin{align}\label{eq:mobius_inversion}
        D_k(T,\pi;\underline{s}) = \sum_{\substack{H\le G\\HT = G}}\mu_{G,T}(H) \Detale{S(\pi)}{k}( (T\cap H)(\widetilde{\pi});\tw_{T\cap H,\widetilde{\pi}}^*(\underline{s})),
    \end{align}
    where $\mu_{G,T}$ is the M\"obius function of the lattice of subgroups $\{H\le G : HT = G\}$, $S(\pi) = \{\pk : \pk\mid |G|\infty\text{ or }\pi(I_\pk)\ne 1\}$, and $\tw_{T\cap H,\widetilde{\pi}}^*(\underline{s})$ is the tuple of variables
    \[
        \tw_{T\cap H,\widetilde{\pi}}^*(\underline{s})_\tau = s_{\tw_{T\cap H,\widetilde{\pi}}(\tau)}
    \]
    indexed by $\tau\in \RT_{k,0}^{S(\pi)}((T\cap H)(\widetilde{\pi}))$.
\end{proposition}

A meromorphic continuation of the righthand side gives a meromorphic continuation of the left-hand side, so it suffices to study $\Detale{S}{k}(G;\underline{s})$ for $G_k$-groups $G$. We remark that an $(S(\pi),k)$-ramification type need not be the same as a $k$-ramification type. While $\widetilde{\pi}$ induces the action on $T(\widetilde{\pi})$, it is possible for the action to factor through some further quotient. There may be primes ramified in $\pi$ (and so in $\widetilde{\pi}$) which are not ramified in the action of $T(\widetilde{\pi})$.

\begin{proof}
    Given $\widetilde{\pi}\in \Sur(G_k,G)$ and $f\in Z^1(k,G)$, we note that $\im(f*\widetilde{\pi})$ necessarily belongs to the lattice $\{H\le G : HT = G\}$. The fact that $HT=G$ is a group implies $T\cap H$ is a normal subgroup in $H$. Thus, we can write
    \[
        \sum_{f\in Z^1(k,T(\pi))} \inv_{k,G}(f*\widetilde{\pi},\underline{s}) = \sum_{\substack{H\le G\\HT=G}} \sum_{\substack{f\in Z^1(k,(T\cap H)(\widetilde{\pi}))\\\im(f*\widetilde{\pi}) = H}} \inv_{k,G}(f*\widetilde{\pi},\underline{s}).
    \]
    It follows from M\"obius inversion that
    \[
        D_k(T,\pi;\underline{s}) = \sum_{\substack{H\le G\\HT = G}}\mu_{G,T}(H)\sum_{f\in Z^1(k,(T\cap H)(\widetilde{\pi}))} \inv_{k,G}(f*\widetilde{\pi},\underline{s}).
    \]
    Finally, we note that the definition of $\inv$ as a product over ramification types implies
    \[
        \inv_{k,G}(f*\widetilde{\pi},\underline{s}) = \inv_{k,(T\cap H)(\widetilde{\pi})}(f, \tw_{T\cap H,\widetilde{\pi}}^*(\underline{s})).
    \]
    This concludes the proof.
\end{proof}

\subsection{Explicit meromorphic continuation}\label{sec:explicit_merom}

We are now ready to construct an explicit meromorphic continuation of $\Detale{S}{k}(T(\pi);\underline{s})$ in analogy with the single variable case in \cite[Corollary 3.4]{alberts2024}. Despite the heavy notation, the structure of the argument follows \cite{alberts-odorney2021} and \cite{alberts2024} quite closely: using Poisson summation as in \cite{alberts-odorney2021} we express $\Detale{S}{k}(T(\pi);\underline{s})$ as a finite linear combination of Euler products. These Euler products are described explicitly by appealing to cohomological results from \cite{alberts2024}. We carefully track the dependence on more parameters, but otherwise the proof is analogous to that of \cite[Corollary 3.4]{alberts2024}.

\begin{theorem}\label{thm:explicit_continuation}
    Let $T\normal G$ be an abelian normal subgroup with quotient map $q:G\to G/T$, $\pi\in q_*\Sur(G_k,G/T)$, and $S$ a finite set of places of $k$ containing $\{\pk : \pk \mid |G|\infty \text{ or }\pi(I_\pk) \ne 1\}$. Let $\underline{s} = (s_\tau)$ be a tuple of complex variables with real part $\underline{\sigma}=(\sigma_\tau)$ and imaginary part $\underline{t}=(t_\tau)$ each indexed by $\tau\in\RT_{k,0}(T(\pi))$ the nontrivial ramification types of $T(\pi)$. Then
    \[
        \Detale{S}{k}(T(\pi);\underline{s}) = \frac{|T|}{|H^0(k,T(\pi))|}\sum_{h\in H^1_{ur^\perp}(k,T(\pi)^*)} Q(\underline{s},h) B(\underline{s},h) \prod_{\tau\in \RTo{k}(T(\pi))} L(s_\tau,\rho_h|_{\C\tau}),
    \]
    where
    \begin{enumerate}[(a)]
        \item $H^1_{ur^{\perp}}(k,T(\pi)^*)$ is the dual Selmer group to $H^1_{ur}(k,T(\pi))$ and it satisfies
        \[
            |H^1_{ur^{\perp}}(k,T(\pi)^*)| \ll_{[k:\Q],|T|} |H^1_{ur}(k,T(\pi))|.
        \]
        \item $Q(\underline{s},h)$ is an entire function given by
        \begin{align*}
            Q(\underline{s},h) = &\prod_{\pk\in S}\left(\frac{1}{|H^0(k_\pk,T(\pi))|}\sum_{\tau\in \RT_k^{S,{\rm wild}}(T(\pi))}\sum_{\substack{f\in H^1(k_\pk,T(\pi))\\ \RT(\pk,f|_{I_\pk})=\tau}} \langle f,h\rangle |\pk|^{-s_\tau}\right)\\
            &\cdot \prod_{\pk\in S} \prod_{\substack{\tau\in \RT_k^{S,{\rm wild}}(T(\pi))\\\tau=(\pk,[f])}}\det\left(I - {\rm tr}(\rho_h|_{\C \tau}(\Fr_\pk|\C\tau^{I_\pk}))|\pk|^{-s_\tau}\right),
        \end{align*}
        which satisfies
        \[
            |Q(\underline{s},h)| \le O_{[k:\Q],|T|}(1)^{|S|} \inv^S_{k,G/T}\left(\pi,q_*\underline{\sigma}\right)
        \]
        where $q_*\underline{\sigma}$ is the tuple of real variables defined as in Theorem \ref{thm:meromorphic_continuation_of_fiber_series}.
        \item $B(\underline{s},h)$ is an absolutely convergent multiple Dirichlet series on the region
        \[
            \{\underline{s} : \text{if }\tau\in\RTo{k}(T(\pi)),\text{ then }\sigma_\tau > 1/2\},
        \]
        which is bounded uniformly in this region by
        \[
            |B(\underline{s},h)| \ll_{[k:\Q]} \prod_{\tau\in \RTo{k}(T(\pi))} |\sigma_\tau - 1/2|^{O_{[k:\Q],|T|}(1)}.
        \]
        \item $\rho_h:G_k\to {\rm GL}(V)$ is the Galois representation defined in \cite[Definition 5.1]{alberts2024} acting on $V:=\C[\Hom((\Q/\Z)^*,T(\pi))]$. In particular, $\rho_h|_{\C\tau}$ is induced from a Hecke character $\chi$ over the field fixed by a stabilizer of the Galois action on $\tau$, denoted $F(\pi,\tau)$ in Theorem \ref{thm:meromorphic_continuation_of_fiber_series}, which is unramified away from the primes dividing $|T|\infty$ and satisfies $\chi^{\exp(T)} = 1$ for $\exp(T)$ the exponent of $T$ as a group.
        
        If $h=0$, then $\rho_h|_{\C\tau}$ is induced from the trivial character over $F(\pi,\tau)$.
    \end{enumerate}
\end{theorem}

It was sufficient information for the purposes of  \cite{alberts2024} to know that the Hecke characters in (d) are unramified at all primes $\pk\nmid |T|\disc(\pi)\infty$, although we will see that it is immediate from the same arguments that such Hecke characters are unramified at all $\pk\nmid |T|\infty$.

The information given above can immediately be distilled to a meromorphic continuation for $\Detale{S}{k}(T(\pi);\underline{s})$.

\begin{corollary}\label{cor:explicit_continuation}
    Let $T\normal G$ be an abelian normal subgroup with quotient map $q:G\to G/T$, and $\pi\in q_*\Sur(G_k,G)$. Then $\Detale{S}{k}(T(\pi);\underline{s})$ has a meromorphic continuation to the set
    \[
        \{\underline{s} : \text{if }\tau\in\RTo{k}(T(\pi)),\text{ then }{\rm Re}(s_\tau) > 1/2\},
    \]
    with at most simple poles at $s_\tau=1$ for each $\tau\in \RTo{k}(T(\pi))$.

    Moreover, for each $\underline{s}$ in this region
    \begin{align*}
        |\Detale{S}{k}(T(\pi);\underline{s})| \ll_{[k:\Q],|T|}& |H^1_{ur}(k,T(\pi))|\cdot O_{[k:\Q],|T|}(1)^{|S|}\cdot \inv^S_{k,G/T}\left(\pi,q_*\underline{\sigma}\right)\\
        &\cdot \max_{h\in H^1_{ur^{\perp}}(k,T(\pi))}\prod_{\tau\in\RTo{k}(T(\pi))} |L(s_\tau,\rho_h|_{\C\tau})|\cdot |\sigma_\tau-1/2|^{O_{[k:\Q],|T|}(1)}.
    \end{align*}
\end{corollary}

\begin{proof}[Proof of Corollary \ref{cor:explicit_continuation}]
    It is clear from Theorem \ref{thm:explicit_continuation} that $\Detale{S}{k}(T(\pi);\underline{s})$ is meromorphic on this region with possible singularities at $s_\tau = 1$ for each $\tau\in \RTo{k}(T(\pi))$ coming from the $L$-functions. The bound is immediate from combining the bounds in Theorem \ref{thm:explicit_continuation}(a,b,c).
\end{proof}

This statement and \Cref{thm:meromorphic_continuation_of_fiber_series} are extremely similar. When we go to prove \Cref{thm:meromorphic_continuation_of_fiber_series}, we do so by stitching together the results for $\Detale{S}{k}(T(\pi);\underline{s})$ according to \Cref{prop:mobius_inversion} resulting in similar statements for $D_k(T,\pi;\underline{s})$.

\begin{remark}
    It is possible to confirm that $s_\tau = 1$ is a simple pole of $\Detale{S}{k}(T(\pi);\underline{s})$ for each $\tau\in \RTo{k}(T(\pi))$ using the the ideas in joint work of the first author with O'Dorney \cite{alberts-odorney2021,alberts-odorney2023} to show that the poles of the Hecke $L$-functions at the trivial character do not cancel out. We have opted not to include this step, as we are losing control over the order of the poles by the end of our argument regardless.
\end{remark}

\subsection{Proof of Theorem \ref{thm:explicit_continuation}}

We will follow along the program of study in \cite{alberts-odorney2021,alberts2024} to prove Theorem \ref{thm:explicit_continuation}. These papers do essentially exactly what we need in a single variable setting, so it will suffice to confirm that the arguments still work in the multivariable setting.

We first use Poisson summation to express $\Detale{S}{k}(T(\pi);\underline{s})$ as a finite sum of Euler products as in \cite{alberts-odorney2021}. Throughout this section let $s_1 = 0$, where $1$ is the trivial tame ramification type. This allows us to write more concise formulas.

\begin{lemma}
    The series $\Detale{S}{k}(T(\pi);\underline{s})$ is equal to
    \[
        \frac{|T|}{|H^0(k,T(\pi)^*)|} \sum_{h\in H^1_{ur^{\perp}}(k,T(\pi)^*)} \prod_\pk\left(\frac{1}{|H^0(k_\pk,T(\pi))|}\sum_{\tau\in\RT_{k}^S(T(\pi))}\sum_{\substack{f\in H^1(k_\pk,T(\pi))\\\RT^S(\pk,[f])=\tau}} \langle f,h\rangle |\pk|^{-s_\tau}\right),
    \]
    where $T(\pi)^* = \Hom(T(\pi),\mu)$ is the Tate dual group, $H^1_{ur^\perp}(k,T(\pi)^*)$ is the dual Selmer group to $H^1_{ur}(k,T(\pi))$, and $\langle,\rangle$ is the (local) Tate pairing.
\end{lemma}

\begin{proof}
    As the ramification types $\RT^S(\pk,f)$ depend only on the class $[f]\in H^1(k,G)$, we can alternatively write
    \[
        \Detale{S}{k}(T(\pi);\underline{s}) = \frac{|T|}{|H^0(k,T(\pi))|}\sum_{[f]\in H^1(k,T(\pi))} \inv^S_{k,T(\pi)}(f,\underline{s}).
    \]
    The locally compact abelian groups $H^1(k,T(\pi))$ and $H^1(k,T(\pi)^*)$ are dual to each other via the Tate pairing $\langle, \rangle$, where $T(\pi)^* = \Hom(T(\pi),\mu)$ is the Tate dual module. The function $\Detale{S}{k}(T(\pi);\underline{s})$ then appears to be the left-hand side of a Poisson summation, which is precisely how these types of sums are addressed in \cite{alberts-odorney2021}.

    The function $\inv^S_{k,T(\pi)}(-,\underline{s})$ can naturally be extended to a function from $H^1(\A_k,T(\pi)) \to \C$. This function is multiplicative in the sense of \cite[Definition 3.1]{alberts-odorney2021} and it is periodic with respect to the compact open subgroup
    \[
        Y_{ur}:=\prod_{\pk\mid \infty} H^1(k_\pk,T(\pi)) \times \prod_{\pk\nmid \infty} H^1_{ur}(k,T(\pi))
    \]
    because it depends only on ramification, not splitting type. \cite[Proposition 4.1]{alberts-odorney2021} applies to a single variable analog of this function with $L_\pk = H^1(k_\pk,T(\pi))$, and it is straight forward to see that the same argument shows that $\inv^S_{k,T(\pi)}(-,\underline{s})$ is multiplicative and satisfies the hypotheses of \cite[Theorem 2.3]{alberts-odorney2021}. Thus, \cite[Theorem 2.3]{alberts-odorney2021} implies
    \[
        \Detale{S}{k}(T,\pi;\underline{s}) = \frac{|T|}{|H^0(k,T(\pi)^*)|}\sum_{h\in H^1(k,T(\pi)^*) \cap Y_{ur}^{\perp}} \widehat{\inv}_{k,T(\pi)}^S(h,\underline{s}),
    \]
    where $\widehat{\inv}_{S,k,T(\pi)}(-,\underline{s})$ is the Fourier transform given by
    \[
        \widehat{\inv}_{k,T(\pi)}^S(h,\underline{s}) = \frac{1}{|H^0(k_\pk,T(\pi))|}\prod_\pk\left(\sum_{\tau\in\RT_{k}^S(T(\pi))}\sum_{\substack{f\in H^1(k_\pk,T(\pi))\\\RT^S(\pk,[f])=\tau}} \langle f,h\rangle |\pk|^{-s_\tau}\right).
    \]
    This concludes the proof, as $H^1(k,T(\pi)^*) \cap Y_{ur}^{\perp}$ is the definition of the dual Selmer group $H^1_{ur^{\perp}}(k,T(\pi)^*)$.
\end{proof}

It is remarked in \cite{alberts-odorney2021} that these Euler products are Frobenian, which is used to construct a meromorphic continuation. The work in \cite{alberts2024} goes much further in describing these Euler products in \cite[Theorem 3.1]{alberts2024}. The following is a multivariable version of this result:

\begin{lemma}\label{lem:tame_euler_factors}
    Suppose $\pk\not\in S$ and $h\in H^1_{ur^{\perp}}(k,T(\pi)^*)$. Then
    \[
        \frac{1}{|H^0(k_\pk,T(\pi))|}\sum_{\tau\in\RT_{k}^S(T(\pi))}\sum_{\substack{f\in H^1(k_\pk,T(\pi))\\\RT^S(\pk,[f])=\tau}}\langle f,h\rangle |\pk|^{-s_\tau} = \sum_{\tau\in\RT_{k}^{\rm tame}(T(\pi))}{\rm tr}(\rho_h|_{\C\tau}(\Fr_\pk)) |\pk|^{-s_\tau}
    \]
    for $\rho_h$ the Galois representation defined in \cite[Definition 5.1]{alberts2024}.
\end{lemma}

\begin{remark}
    We note the role that ramification types play in this setting, which is particularly evident in Lemma \ref{lem:tame_euler_factors}. The series $\Detale{S}{k}(T(\pi);\underline{s})$ has a meromorphic continuation past its rightmost poles because each whole ramification type is assigned a single complex variable. If one separates a single ramification type with different variables, the analog to Lemma \ref{lem:tame_euler_factors} will necessarily have rational, not integer, coefficients. This would result in requiring the removal of a branch cut before being able to give a continuation of $\Detale{S}{k}(T(\pi);\underline{s})$, which is incompatible with our method.
\end{remark}

\begin{proof}[Proof of Lemma \ref{lem:tame_euler_factors}]
    As $\pk$ is tame, the only ramification types on the left-hand side with a nonzero summand are necessarily tame. We also note that $T$ being abelian implies $\Hom((\Q/\Z)^*,T(\pi)) = H^1((\Q/\Z)^*,T(\pi))$. The proof of \cite[Theorem 3.1]{alberts2024} applies without change after replacing $\ind(o)s$ for a Galois orbit $o\subset H^1((\Q/\Z)^*,T(\pi))$ with the complex variable $s_{o}$ (and relabeling Galois orbits $o$ with the variable $\tau$ for consistency).
\end{proof}

We have now built up the necessary tools in the multivariate setting to prove Theorem \ref{thm:explicit_continuation}

\begin{proof}[Proof of Theorem \ref{thm:explicit_continuation}]
Putting these lemmas together and factoring out the places in $S$, we have proven that
\[
    \Detale{S}{k}(T(\pi);\underline{s}) = \frac{|T|}{|H^0(k,T(\pi)^*)|}\sum_{h\in H^1_{ur^{\perp}}(K,T(\pi)^*)} Q(\underline{s},h)\prod_{\pk} \left(\sum_{\tau\in\RT_k^{\rm tame}(T(\pi))}{\rm tr}(\rho_h|_{\C\tau}(\Fr_\pk)) |\pk|^{-s_\tau}\right)
\]
by the definition of $Q(\underline{s},h)$ accounting for the places in $S$. Decomposing the Euler factors as in \cite[Corollary 3.3]{alberts2024}, we find that
\[
    \Detale{S}{k}(T(\pi);\underline{s}) = \frac{|T|}{|H^0(k,T(\pi)^*)|}\sum_{h\in H^1_{ur^{\perp}}(K,T(\pi)^*)} Q(\underline{s},h)B(\underline{s},h)\prod_{\tau\in\RTo{k}(T(\pi))}L(s_\tau,\rho_h|_{\C \tau}),
\]
where $B(\underline{s},h)$ is an absolutely convergent Euler product on $\Omega_T$, and in fact
\[
    B(\underline{s},h) = \prod_{\pk} \left( 1 + \sum_{\tau\in\RTo{k}(T(\pi))}O_{|T|}(|\pk|^{-2\sigma_\tau})\right)
\]
so that the bound follows. This confirms the decomposition as well as part (c).

The dual Selmer group to $H^1_{ur}(k,T(\pi))$ can be bounded using the Greenberg--Wiles identity \cite[Theorem (8.7.9)]{neukirch-schmidt-wingberg2013cohomology} to prove part (a):
\begin{align*}
    |H^1(k,T(\pi)^*) \cap Y_{ur}^\perp| &= |H^1_{ur}(k,T(\pi))| \frac{|H^0(k,T(\pi)^*)|}{|H^0(k,T(\pi))|}\prod_{\pk\mid \infty}\frac{|H^0(k_\pk,T(\pi))|}{|H^1(k_\pk,T(\pi))|}\\
    &\ll_{[k:\Q],|T|} |H^1_{ur}(k,T(\pi))|.
\end{align*}

Meanwhile, the wild parts can be bounded by noting that $|\det(I - M|\pk|^{-s})| \ll_n 1$ for ${\rm Re}(s) > 0$ and $M$ an $n\times n$ matrix with $|\det(M)| = 1$. Thus,
\begin{align*}
    |Q(\underline{s},h)| &\ll_{|T|} \prod_{\pk\in S}\left(\frac{1}{|H^0(k_\pk,T(\pi))|}\sum_{\tau\in \RT_{k}^{S,{\rm wild}}(T(\pi))}\sum_{\substack{f\in H^1(k_\pk,T(\pi))\\ \RT^S(\pk,f|_{I_\pk})=\tau}} |\langle f,h\rangle| |\pk|^{-s_\tau}|\right)\\
    &\ll_{|T|} \prod_{\pk\in S}\left(\frac{1}{|H^0(k_\pk,T(\pi))|}\sum_{\tau\in \RT_{k}^{S,{\rm wild}}(T(\pi))}\sum_{\substack{f\in H^1(k_\pk,T(\pi))\\ \RT^S(\pk,f|_{I_\pk})=\tau}} |\pk|^{-\sigma_\tau}\right)\\
    &= \prod_{\pk\in S}\left(\sum_{\tau\in \RT_{k}^{S,{\rm wild}}(T(\pi))}\sum_{\substack{f\in \res_{I_\pk}H^1(k_\pk,T(\pi))\\ \RT^S(\pk,f)=\tau}} |\pk|^{-\sigma_\tau}\right).
\end{align*}
We then bound $\sigma_\tau \ge (q_*\underline{\sigma})_{q_*\tau}$ by the definition of $q_*\underline{\sigma}$. This gives the bound
\begin{align*}
    |Q(\underline{s},h)| &\le O_{[k:\Q],|T|}(1)^{|S|} \prod_{\overline{\tau}\in\RT_{k,0}^S(G/T)} \prod_{\substack{\pk\in S\\\RT^S(\pk,\pi) = \overline{\tau}}} p^{-(q_*\sigma)_{\overline{\tau}}}.
\end{align*}
We note that the contribution of any individual place is bounded above by $1$, and so can be removed. Thus, removing the contributions of all places in $S$ which are unramified in $\pi$ gives exactly
\begin{align*}
    |Q(\underline{s},h)| &\le O_{[k:\Q],|T|}(1)^{|S|} \inv^S_{k,G/T}\left(\pi,q_*\underline{\sigma}\right),
\end{align*}
proving part (b).

Finally, part (d) follows from \cite[Definition 1.5]{alberts2024}. The representation is defined by
\[
    \rho_h(g)a = [g.a, (h*\phi)(g)](g.a)
\]
for each $a\in \Hom((\Q/\Z)^*,T(\pi))$ and $g\in G_k$. Here $[,]$ is the ``parametrizing local Tate pairing'' \cite[Definition 4.1]{alberts2024} valued in the group of roots of unity, and
\[
h*\phi:G_k \to T(\pi)^*\rtimes \Gal(F/k)
\]
is the pointwise product $(h*\phi)(g) = (h(g),\phi(g))$, where $F/k$ contains the fields of definition for both $T(\pi)$ and $T(\pi)^*$ and $\phi:G_k\to \Gal(F/k)$ is the canonical quotient map. This is a monomial representation in the sense of \cite[Definition 2.1]{alberts2024}, which by \cite[Lemma 2.2]{alberts2024} is a direct sum of representations induced from Hecke characters on each Galois orbit $o\subseteq \Hom((\Q/\Z)^*,T(\pi))$ over the field $k_a$ fixed by the Galois stabilizer of some $a\in o$. The Galois orbits are precisely the tame ramification types, so relabeling $o$ with $\tau$ and $k_a$ with $F(\pi,\tau)$ gives us the desired description.

Moreover, the fact that $h\in H^1_{{ur}^{\perp}}(k,T(\pi)^*)$ implies that $h$ is unramified at all primes not dividing $|T|\infty$ (see, for example, \cite[Theorem 2.17(e)]{darmon-diamond-taylor1995}). This implies that the Hecke character $\chi$ factors through $\Gal(L/F(\pi,\tau))$ for some cyclic extension $L/F(\pi,\tau)$ of degree equal to the order of $\chi$ which is unramified away from the primes dividing $|T|\infty$.
\end{proof}

\subsection{Proof of Theorem \ref{thm:meromorphic_continuation_of_fiber_series}}

We first construct a region of meromorphicity for each individual fibers series.

\begin{proof}[Proof of Theorem \ref{thm:meromorphic_continuation_of_fiber_series}(i)]
    This follows directly from Proposition \ref{prop:mobius_inversion} decomposing $D_k(T,\pi;\underline{s})$ into a sum of $\Detale{S}{k}((T\cap H)(\pi);\tw_{T\cap H,\widetilde{\pi}}^*\underline{s})$ and Corollary \ref{cor:explicit_continuation} giving the meromorphic continuation.
    
    Generically, we expect the poles to be simple poles. However, in the presence of extra roots of unity ramification types of $G$ can split into the union of several ramification types of $T(\pi)$. This is captured by the quantity $|\tw_{T,\widetilde{\pi}}^{-1}(\tau)|$. More explicitly, the upper bound on the order of the pole comes from the overlap between the poles $s_{\tw_{T\cap H,\widetilde{\pi}}(\tau)}=1$ for several $\tau\in \RT_{k,0}((T\cap H)(\pi))$.
\end{proof}

Next, we distill the explicit description of $\Detale{S}{k}(T(\pi);\underline{s})$ in \Cref{thm:explicit_continuation} into vertical bounds for $D_k(T,\pi;\underline{s})$.

\begin{proof}[Proof of Theorem \ref{thm:meromorphic_continuation_of_fiber_series}(ii)]
    Proposition \ref{prop:mobius_inversion} and the vertical bounds given in Corollary \ref{cor:explicit_continuation} together give an upper bound of the form
    \begin{align*}
        &\ll |H^1_{ur}(k,T(\pi))|\cdot \inv_{k,G/T}(\pi,q_*\underline{\sigma}+\epsilon) \max_{\substack{H\le G\\ HT=G}}\left\{\prod_{\widetilde{\tau}\in \RTo{k}((T\cap H)(\pi))} \max_{\chi} |L_{F(\pi,\widetilde{\tau})}(s_{\tw_{T\cap H,\widetilde{\pi}}(\widetilde{\tau})} , \chi)|\right\}\\
        &\ll|H^1_{ur}(k,T(\pi))|\cdot \inv_{k,G/T}(\pi,q_*\underline{\sigma}+\epsilon) \max_{\substack{H\le G\\ HT=G}}\left\{\prod_{\widetilde{\tau}\in \RTo{k}((T\cap H)(\pi))} \max_{\chi} \left\{|L_{F(\pi,\widetilde{\tau})}(s_{\tw_{T\cap H,\widetilde{\pi}}(\widetilde{\tau})} , \chi)|,1\right\}\right\},
    \end{align*}
    where we note that the $\chi$ range over Hecke characters that induce to $\rho_h|_{\C\widetilde{\tau}}$. These are Hecke characters on $F(\pi,\widetilde{\tau})$ of order $\le \exp(T)$ which are unramified away from the primes dividing $|T|\infty$ by Theorem \ref{thm:explicit_continuation}(d).

    Let $\iota_H:T\cap H \to T$ be the inclusion map. The inclusion map and twists of ramification types commute via the following diagram:
    \[
    \begin{tikzcd}
        \RTo{k}((T\cap H)(\pi)) \drar[swap]{\tw_{T\cap H,\widetilde{\pi}}} \arrow{rr}{(\iota_H)_*} &&\RTo{k}(T(\pi))\dlar{\tw_{T,\widetilde{\pi}}}\\
        &\RTo{k}(G)
    \end{tikzcd}
    \]
    Thus, for any ramification type $\widetilde{\tau}\in \RTo{k}((T\cap H)(\pi))$ it follows that
    \[
        L_{F_{\pi,\widetilde{\tau}}}(s_{\tw_{T\cap H,\widetilde{\pi}}(\widetilde{\tau})} , \chi) = L_{F_{\pi,\widetilde{\tau}}}(s_{\tw_{T,\widetilde{\pi}}((\iota_H)_*\widetilde{\tau})} , \chi).
    \]
    Finally, we claim that the pushforward of the inclusion map on tame ramification types is injective in this case. Indeed, $(\iota_H)_*$ is injective on the underlying hom sets. As both $T$ and $T\cap H$ are abelian, the action by conjugation is trivial, so it suffices to show that $(\iota_H)_*$ is injective on $G_k$-orbits, which is clear as $\Hom((\Q/\Z)^*,(T\cap H)(\pi))$ is in fact a submodule of $\Hom((\Q/\Z)^*,T(\pi))$.

    In particular, the product of $L$-functions can be written as
    \begin{align*}
        &\prod_{\widetilde{\tau}\in \RTo{k}((T\cap H)(\pi))} \max_{\chi}\left\{|L_{F(\pi,\widetilde{\tau})}(s_{\tw_{T\cap H,\widetilde{\pi}}(\widetilde{\tau})} , \chi)|,1\right\}\\
        &= \prod_{\widetilde{\tau}\in (\iota_H)_*\RTo{k}((T\cap H)(\pi))} \max_{\chi}\left\{|L_{F(\pi,\widetilde{\tau})}(s_{\tw_{T,\widetilde{\pi}}(\widetilde{\tau})} , \chi)|,1\right\}\\
        &\le \prod_{\widetilde{\tau}\in \RTo{k}(T(\pi))} \max_{\chi}\left\{|L_{F(\pi,\widetilde{\tau})}(s_{\tw_{T,\widetilde{\pi}}(\widetilde{\tau})} , \chi)|,1\right\},
    \end{align*}
    where the last inequality comes from multiplying in each extra terms for $\tau \in \RTo{k}(T(\pi)) - \iota_*\RTo{k}((T\cap H)(\pi))$. For this step to work as intended, it is necessary to take the maximum of each $L$-function with $1$.
    
    This product is now independent of $H$, so the maximum over $H$ can be dropped. Factoring this product according to the possible images $\tau = \tw_{T,\widetilde{\pi}}(\widetilde{\tau}) \in \iota_*\RTo{k}(T)$ and plugging this into the upper bound concludes the proof.
\end{proof}

\section{Bounds for $h^1_{ur}(k,N,T,\pi)$ Using Class Group Torsion}\label{sec:bounds}

In order to meromorphically continue the series $D_k(G;\underline{s})$ and extract information about its polar behavior we will need to account for the contributions coming from two main sources: the appropriate class group torsion and the various (Hecke) $L$-functions that correspond to each of our complex variables in $\underline{s}.$  For the class group torsion, the bounds will have to be in terms of the conductor of the map $\pi$ along which we fiber the series, while for $L$-functions we will need to make explicit the dependence both on the discriminant of the fields involved, and on the imaginary part of the complex variables involved.

Assume the group $G$ is concentrated in a nilpotent normal subgroup $N$. Denote by $q_N: G \to G/N$ the projection on the quotient by $N.$ By \eqref{eq:fibered_sum} we can write the generating multiple Dirichlet series for $G$ as a sum over fibers

\[
    D_k(G;\underline{s}) = \sum_{\pi \in (q_N)_* \Sur(G_k, G)} D_k (N, \pi; \underline{s}).
\]
For each normal subgroup $T\normal G$ contained in $N$ we can further decompose the fiber multiple Dirichlet series corresponding to $\pi$ as
\[
    D_k(N, \pi;\underline{s}) = \sum_{\substack{\psi\in (q_T)_*\Sur(G_k,G)\\(q_{T,N})_*\psi = \pi}} D_k (T, \psi; \underline{s}),
\]
where $q_{T,N}:G/T\to G/N$ is the quotient map.

Following \cite[Theorem 6.1]{ALOWW}, we know that when $T$ is abelian the $\psi$-dependence on $D_k(T,\psi;\underline{s})$ is approximately $|H^1_{ur}(k,T(\psi))|$, which can be bounded in terms of class group torsion.

In order to prove similar results for groups concentrated in a nilpotent normal subgroup $N$, we need to understand enough of the $\pi$-dependence on the series $D_k(N,\pi;\underline{s})$. If $G$ is in fact semiconcentrated in the abelian normal subgroups $T_1,T_2,...,T_r$ which are all contained in $N$, we expect this dependence to come from $|H_{ur}^1(k,T(\psi))|$ for various $\psi$ lying over $\pi$. The goal of this section is to show that references to $\psi$ can essentially be removed and give a bound in terms of $\pi$ alone.

\begin{definition}\label{def:h1ur}
    Let $G$ be a finite group with normal subgroups $T,N\normal G$ for which $T$ is abelian, $T$ contained in the hypercenter of $N$, and $q:G\to G/N$ is the quotient map.

    Write
    \[
        1=: Z_0(N) \subseteq Z(N) =: Z_1(N) \subseteq Z_2(N) \subseteq Z_3(N) \subseteq \cdots \subseteq Z_{\infty}(N) = N
    \]
    for the upper central series for $N$, concluding with the hypercenter.

    For any $\pi \in q_*\Sur(G_k,G)$ define
    \[
        h_{ur}^1(k,N,T,\pi) := \prod_{i=1}^{\infty} |H^1_{ur}(k, (T\cap Z_{i}(N)/ T\cap Z_{i-1}(N))(\pi))|.
    \]
\end{definition}

This is well-defined, as $N$ acts trivially by conjugation on $Z_i(N)/Z_{i-1}(N)$ by construction of the upper central series. Thus, for any $\psi\in q_*^{-1}(\pi)$ for $q:G\to G/N$ the quotient map, the action on $(T\cap Z_{i}(N)/T\cap Z_{i-1}(N))(\psi)$ necessarily factors through $\pi$.

We note that nilpotent groups are equal to their own hypercenter, so if we take $N$ to be nilpotent it suffices to have $T\subseteq N$.

For a single Galois module $M$, the quantity $|H^1_{ur}(k,M(\pi))|$ is closely related to the size of certain class group torsion. This is made more explicit in \cite[Corollary 1.14]{ALOWW}, and is the key step in converting bounds on class group torsion to new results towards new cases of \Cref{conj:number_field_counting} for concentrated groups. As $h^1_{ur}(k,N,T,\pi)$ is a product of over several Galois modules, we can cite the bounds for individual Galois modules to bound $h^1_{ur}(k,N,T,\pi)$ in terms of class group torsion as well.

\begin{proposition}\label{prop:h1ur_bound}
    Let $G$ be a finite group with normal subgroups $T,N\normal G.$ Assume $T$ is abelian and contained in the hypercenter of $N$. Denote by $q:G\to G/N$ the quotient map.

    Then for any $\pi \in q_*\Sur(G_k,G)$ and any $\psi\in q_*^{-1}(\pi)$,
    \[
        |H^1_{ur}(k,T(\psi))| \ll_{[k:\Q],|G|,\epsilon} h^1_{ur}(k,N,T,\pi) \cdot |\disc(\psi)|^{\epsilon}
    \]
\end{proposition}

\begin{proof}
    This essentially follows directly from \cite[Lemma 4.3]{ALOWW}. By applying this result to the submodule $T\cap Z_1(N) \le T$ (where we let $H$ be the entire Galois group defining the action on $T(\psi)$), we get the bound
    \[
        |H^1_{ur}(k,T(\psi))| \le |H^1_{ur}(k,(T\cap Z_1(N))(\psi))| \cdot |H^1_{ur}(k,(T/T\cap Z_1(N))(\psi))| \cdot [T:T\cap Z_1(N)]^{|G| + \omega(\disc(\psi)) - 1}.
    \]
    By iterating this process along the upper central series, this produces the bound
    \[
        |H^1_{ur}(k,T(\psi))| \le \prod_{i=1}^{\infty}|H^1_{ur}(k,(T\cap Z_i(N)/T\cap Z_{i-1}(N))(\psi))| \cdot \left(\prod_{j=1}^{\infty}[T:T\cap Z_j(N)]\right)^{|G| + \omega(\disc(\psi)) - 1},
    \]
    where we note that the products are secretly finite, as the groups are all finite so the upper central series terminates after finitely many steps.

    The result then follows from the fact that the action on $(T\cap Z_i(N)/T\cap Z_{i-1}(N))(\psi)$ factors through $\pi$ and that
    \[
        \left(\prod_{i=1}^{\infty}[T:T\cap Z_j(N)]\right)^{|G| + \omega(\disc(\psi)) - 1} \ll_{[k:\Q],|G|,\epsilon} |\disc(\psi)|^{\epsilon}.
    \]
\end{proof}

\section{Unconditional Results}\label{sec:unconditional}

In general there is no guarantee that the convex hull will be large enough to conclude new counting asymptotics. Producing the largest possible region requires keeping track of the most information possible in terms of subconvexity bounds for Hecke $L$-functions.

We recall that a number field $F$ and group $G$ satisfy the hybrid subconvexity bound $H(\gamma,\alpha,\beta;F, G)$ if
\[
    |L_F(s,\chi)| \ll_{\epsilon} \left(\disc(F/\Q)\mathfrak{f}(\chi)\right)^{\alpha \max\{1-\sigma,0\}+\epsilon} \left((1+|t|)^{[F:\Q]}\right)^{\beta\max\{1-\sigma,0\}+\epsilon}
\]
for each $\sigma > \gamma$ and for each Hecke $L$-function in $\mathcal L(F,G)$. All hybrid bounds of this shape that we are currently aware of are valid for $\sigma > 1/2$, so for simplicity we will work with $H(1/2,\alpha,\beta;F,G)$ in this section. We remark that it would suffice to have $\gamma < 1$ for our application.

\subsection{From hybrid subconvexity bounds to a region of meromorphicity}

We require hybrid subconvexity bounds for each Hecke $L$-function occurring in the bounds from \Cref{thm:meromorphic_continuation_of_fiber_series}. More specifically, given a collection of abelian normal subgroups $T_1,T_2,\dots,T_r$ of $G$ we assume there is a sequence of constants $\alpha_\tau,\beta_\tau \ge 0$ indexed by $\tau\in \RTo{k}(G)$ such that, for each $T \in \{ T_1, \dots, T_r\}$ and each corresponding subfield $F(\pi,\widetilde{\tau})$ appearing in \Cref{thm:meromorphic_continuation_of_fiber_series} the hypothesis  $H(1/2,\alpha_\tau,\beta_\tau;F(\pi,\widetilde{\tau}), T)$ holds. 

We recall that the components of these $L$-functions are defined as follows.
\begin{itemize}
    \item $F(\pi,\widetilde{\tau})$ is the field fixed by a stabilizer of the Galois action on $\widetilde{\tau}\in \tw_{T,\widetilde{\pi}}^{-1}(\tau)$, where $\pi\in (q_j)_*\Sur(G_k,G)$ with $q_j:G\to G/T_j$ the quotient map and $\widetilde{\pi}\in (q_j)_*^{-1}(\pi)$ is some lift of $\pi$.

    \item $\chi$ is a Hecke character defined over $F(\pi,\widetilde{\tau})$ of order dividing $\exp(T_j)$ which is unramified away from primes dividing $|T|\infty$.
\end{itemize}

It is important that the constants $\alpha_\tau,\beta_\tau$ do not depend on $\pi$. It is possible to allow hybrid subconvexity bounds that depend on $\widetilde{\tau}$ instead of $\tau$, but this significantly complicates the region of meromorphicity and tends to be unnecessary for practical examples. We ensure the uniformity of the constants $\alpha_\tau,\beta_\tau$ by taking the maximum of all the such constants for all the intermediate subfields that appear. 

We can greatly simplify the expression of the region of meromorphicity by defining a matrix $M = (M_{\tau,\kappa})$, indexed by pairs of nontrivial tame ramification types $\tau,\kappa\in \RTo{k}(G)$, by
\[
    M_{\tau,\kappa} := \alpha_{\kappa} [k(\zeta_\kappa):k]\ind_{|\kappa|/[k(\zeta_\kappa):k]}((q_{\kappa})_* \tau),
\]
where $q_\kappa:G\to G/C_{G}(\kappa)$ is the quotient map and $\zeta_\kappa$ is a primitive root of unity with the same order as $\kappa$. We recall that $G/C_G(\kappa)$ carries a natural action on $\kappa$ with orbits of size $|\kappa|/[k(\zeta_\kappa):k]$, and we use this permutation action to define the index function in the above expression. Note that $M$ depends on $k$, $G$, and the tuple of constants $\underline{\alpha}$, all of which have been suppressed in the notation.

For an abelian normal subgroup $T\normal G$, define the tubular region $\Omega_T$ cut out by
\begin{align*}
    \sigma_\tau &> 1/2 && \text{if }\tau\in \iota_*\RTo{k}(T)\\
    \sigma_\tau &> 1 && \text{if }\tau \in \RTo{k}(G) - \iota_*\RTo{k}(T)\\
    \sigma_\tau + \sum_{\kappa\in \iota_*\RTo{k}(T)} M_{\tau,\kappa}\sigma_\kappa &> 1 + \sum_{\kappa\in \iota_*\RTo{k}(T)} M_{\tau,\kappa} && \text{if }\tau \in \RTo{k}(G) - \iota_*\RTo{k}(T).
\end{align*}
We are now ready to state our main unconditional result.

\begin{theorem}\label{thm:general_unconditional_nilpotent}
    Let $k$ be a number field, $G$ a finite nilpotent group, and $T_1,T_2,\dots T_r$ some abelian normal subgroups of $G$. Assume we have (sub)convexity bounds corresponding to $\alpha_\tau,\beta_\tau$ as stated above. Then, for the tubular regions $\Omega_{T_1},\Omega_{T_2},\dots,\Omega_{T_r}$ defined above,
    \begin{enumerate}[(i)]
        \item $D_k(G;\underline{s})$ is meromorphic on the convex hull ${\rm Hull}\left(\bigcup_{j=1}^r \Omega_{T_j}\right)$, with poles at $s_\tau = 1$ for each $\tau \in \bigcup_{j} \iota_*\RTo{k}(T_j)$ of order at most
        \[
            b_\tau = \max_{j=1,\dots,r}\left(\max_{\pi\in (q_j)_*\Sur(G_k,G)} |\tw_{T_j,\widetilde{\pi}}^{-1}(\tau)|\right) \le [k(\zeta_\tau):k].
        \]
        \item For an arbitrary inertial invariant of the group $G$ and its corresponding weight function, assume that the point
        \[
        \left(\frac{\wt(\tau)}{a_\inv(G)}\right)_{\tau}
        \]
        is contained in the convex hull ${\rm Hull}\left(\bigcup_{j} \Omega_{T_j}\right)$. Then there exists a nonzero ineffective polynomial $P=P_{k,G, \inv}$ and an effective constant $\delta =\delta_{k,G, \inv} > 0$ such that
        \[
            \#\mathcal{F}_{\inv, k}(G;X) = X^{1/a_\inv(G)}P(\log X) + O\left(X^{1/a_\inv(G) - \delta}\right),
        \]
        where
        \[
            \max_{\substack{j=1,\dots,r\\\pi\in (q_j)_*\Sur(G_k,G)}}\left\{\sum_{\substack{\tau\in \iota_*\RTo{k}(T_j)\\\wt(\tau) = a_\inv(G)}} |\tw_{T_j,\widetilde{\pi}}^{-1}(\tau)|\right\} - 1\le \deg P \le  \left(\sum_{\substack{\tau\in \bigcup_j\iota_*\RTo{k}(T_j)\\\wt(\tau) = a_\inv(G)}} b_\tau\right) - 1.
        \]
    \end{enumerate}
\end{theorem}

Part (i) of this theorem is proven by verifying the conditions of \Cref{thm:main}. The ``meromorphic continuation of the fibers'' and ``uniform upper bounds for the fibers'' conditions both follow directly from \Cref{thm:meromorphic_continuation_of_fiber_series}, so the focus of our proof is on verifying the ``criterion for convergence''. Part (ii) is then achieved by specializing to the relevant complex line as in \Cref{prop:specialize}.

\begin{proof}[Proof of \Cref{thm:general_unconditional_nilpotent}]
    We first prove (i). The main content of this proof is bounding the Hecke $L$-functions appearing in Theorem \ref{thm:meromorphic_continuation_of_fiber_series}(ii) in order to verify the ``criterion for convergence'' in \Cref{thm:main}.

    Fix a choice of $(T,q,\Omega)=(T_j,q_j,\Omega_j)$. We assumed a shape of hybrid subconvexity bounds for such $L$-functions in terms of $\alpha_\tau$ and $\beta_\tau$, namely that for each $\sigma_\tau > 1/2$ 
    \[
        |L_{F(\pi,\widetilde{\tau})}(s_{\tau},\chi)| \ll_{\epsilon} \left(\disc(F(\pi,\widetilde{\tau})/\Q)\mathfrak{f}(\chi)\right)^{\alpha_\tau\max\{1-\sigma_{\tau},0\}+\epsilon} \left(1+|t_\tau|\right)^{\beta_\tau[F(\pi,\widetilde{\tau}):\Q]\max\{1-\sigma_{\tau},0\}+\epsilon}.
    \]
    It is now clear that, in the product over $\tau$ in Theorem \ref{thm:meromorphic_continuation_of_fiber_series}, we have produced some $t$-aspect bound for $D_k(T,\pi;\underline{s})$. The criterion for convergence will imply that the same $t$-aspect bound carries through to $D_k(G,\underline{s})$ on $\Omega_T$, and by \Cref{cor:independent_variable_multi_phragmen_lindelof} gives a $t$-aspect bound on the entire convex hull. For the sake of brevity, we now omit the $t$-aspect in our proof of the criterion for convergence.

    Choose a lift $\widetilde{\pi}\in q_*^{-1}(\pi)$, and let $F(\pi)$ be the field fixed by $\widetilde{\pi}^{-1}(\Stab_G(f))$ for some $f\in \tau$. We claim that
    \begin{equation}\label{eq:subcon_intermediate}
     |L_{F(\pi,\widetilde{\tau})}(s_{\tau},\chi)| \ll_{|G|,t_{\tau},\epsilon} \disc(F(\pi)/\Q)^{\alpha_\tau\frac{[k(\zeta_\tau):k]}{|\tw_{T,\widetilde{\pi}}^{-1}(\tau)|}\max\{1-\sigma_{\tau},0\}+\epsilon}.    
    \end{equation}
    The product over $\widetilde{\tau}\in \tw_{T,\widetilde{\pi}}^{-1}(\tau)$ then cancels with the $1/|\tw_{T,\widetilde{\pi}}^{-1}(\tau)|$ in the exponent, eliminating the twisting map from the subconvexity bound for $D_k(T,\pi;\underline{s})$.

    The conductor $\mathfrak{f}(\chi) \ll_{|T|} 1$ is bounded by $\chi$ being unramified away from primes dividing $|T|\infty$, and so can be absorbed by the implied constant. Thus, it suffices to compute the discriminant of $F(\pi,\widetilde{\tau})$. By definition, $F(\pi,\widetilde{\tau})$ is the field fixed by a stabilizer of the Galois action on $\widetilde{\tau}\subseteq \Hom((\Q/\Z)^*,T(\pi))$. The $\phi_{\widetilde{\pi}}$-equivariance of the inclusion map implies that $F(\pi,\widetilde{\tau})$ is equivalently the field fixed by the preimage of the stabilizer under the composite map
    \[
        \begin{tikzcd}
            G_k \rar{\widetilde{\pi} \times 1} &G\times G_k \rar &G/C_G(\iota_*\widetilde{\tau})\times \Gal(k(\zeta_\tau)/k),
        \end{tikzcd}
    \]
    where we note that ${\rm ord}(\tau)={\rm ord}(\widetilde{\tau})$ in \Cref{cor:rt_twist_structure}(b) implies we can take $\zeta_\tau = \zeta_{\iota_*\widetilde{\tau}}$. Going further, the proof of \Cref{cor:rt_twist_structure}(c) includes the statement that $G_k$ acts transitively on $\{\iota_*\widetilde{\tau}: \widetilde{\tau}\in \tw^{-1}_{T,\widetilde{\pi}}(\tau)\}$ via the cyclotomic action. As the cyclotomic action commutes with conjugation and $\tau$ is the disjoint union of $\iota_*\widetilde{\tau}$, it follows that $C_G(\iota_*\widetilde{\tau}) = C_G(\tau)$. Thus, $F(\pi,\widetilde{\tau})$ is precisely the field fixed by the preimage of a stabilizer under the composite homomorphism
    \[
        \begin{tikzcd}
            G_k \rar{\widetilde{\pi} \times 1} &G\times G_k \rar &G/C_G(\tau)\times \Gal(k(\zeta_\tau)/k).
        \end{tikzcd}
    \]
    In particular, $F(\pi,\widetilde{\tau})$ is independent of the choice of $\widetilde{\tau}\in \tw^{-1}_{T,\widetilde{\pi}}(\tau)$. The fact that $T$ is abelian implies the kernel is independent of the choice of lift $\widetilde{\pi}$.
    
    The field $F(\pi)$ is equivalent to the subfield of $F(\pi,\widetilde{\tau})$ fixed by $(\widetilde{\pi}\times 1)^{-1}(\Stab_G(f)\times G_k)$. In particular, $F(\pi,\widetilde{\tau}) = F(\pi)(\zeta_\tau)$, which implies $F(\pi,\widetilde{\tau})/F(\pi)$ is ramified only at primes dividing ${\rm ord}(\tau)$ and has discriminant bounded solely in terms of $|T|$. We conclude that
    \[
        \disc(F(\pi,\widetilde{\tau})/\Q) \ll_{|T|} \disc(F(\pi)/\Q)^{[F(\pi,\widetilde{\tau}):F(\pi)]}.
    \]
    It now suffices to understand the degree. Let $N\le G/C_G(\tau)\times \Gal(k(\zeta_\tau)/k)$ be the image of $(\widetilde{\pi}\times 1)(G_k)$, so that $N$ acts faithfully and transitively on $\iota_*\widetilde{\tau}$. This implies
    \[
        [F(\pi,\widetilde{\tau}):k] = |\iota_*\widetilde{\tau}|=\frac{|N|}{|(\Stab_G(f)/C_G(\tau)) \times 1|}
    \]
    for some $f\in \tau$. Meanwhile, by construction $[F(\pi):k]=[G:\Stab_G(f)]$ for some $f\in \tau$. Given that $G/C_G(\tau)\times \Gal(k(\zeta_\tau)/k)$ acts faithfully and transitively on $\tau$, we can write
    \[
        |\tau| = [G:\Stab_G(f)][k(\zeta_\tau):k].
    \]
    \Cref{cor:rt_twist_structure}(c) then implies
    \[
        [F(\pi,\widetilde{\tau}):F(\pi)] = \frac{|\iota_*\widetilde{\tau}|}{[G:\Stab_G(f)]} = \frac{|\tau|}{|\tw_{T,\widetilde{\pi}}^{-1}(\tau)|[G:\Stab_G(f)]} = \frac{[k(\zeta_\tau):k]}{|\tw_{T,\widetilde{\pi}}^{-1}(\tau)|}.
    \]
    This proves \eqref{eq:subcon_intermediate}.
    
    The ramification type of any tame prime in $F(\pi)$ is given by
    \[
        (q_\tau)_*\RT(\pk,\pi).
    \]
    The power to which this prime divides the discriminant is precisely
    \[
        \ind_{|\tau|/[k(\zeta_\tau):k]}((q_\tau)_*\RT(\pk,\pi)).
    \]
    Partitioning the discriminant in terms of the ramification types of $G$, we conclude that
    \[
        \prod_{\widetilde{\tau}\in \tw_{T,\widetilde{\pi}}^{-1}(\tau)} \max_{\chi}\{|L_{F(\pi,\widetilde{\tau})}(s_{\tau},\chi)|,1\} \ll_{t_{\tau},k,|G|,\epsilon} \prod_{\kappa\in\RTo{k}(G)} \prod_{\substack{\pk\\\RT(\pk,\pi)\in q_*\kappa}}|\pk|^{M_{\kappa,\tau}\max\{1-\sigma_{\tau},0\}+\epsilon}.
    \]
    The other components of the bound in \Cref{thm:meromorphic_continuation_of_fiber_series} can be dealt with more directly. Given that $G$ is nilpotent, it follows that $T(\pi)$ is necessarily a nilpotent $G_k$-module in the sense of \cite[Corollary 1.14]{ALOWW}(i), which implies 
    \begin{equation}\label{eq:aloww_bound_H1}
          |H^1_{ur}(k,T(\pi))| \ll_{\epsilon} |\disc(\pi)|^{\epsilon}
    \end{equation}
    is negligible, i.e.~it can be absorbed by the $\inv_{k,G/T}(q_*\underline{\sigma}+\epsilon)$ term. This term is directly bounded by
    \[
        |\inv_{k,G/T}(q_*\underline{\sigma}+\epsilon)| \le \prod_{\overline{\kappa}\in \RTo{k}(G/T)} \prod_{\substack{\pk\\ \RT(\pk,\pi)=\overline{\kappa}}} |\pk|^{-\min\{\sigma_{\kappa} : q_*\kappa = \overline{\kappa}\}}.
    \]
    All together, this gives the bound
    \begin{align*}
        |D_k(T,\pi;\underline{s})| &\ll_{t,k,|G|,\epsilon} \prod_{\tau\in \iota_* \RTo{k}(T)}\prod_{\overline{\kappa}\in \RTo{k}(G/T)} \prod_{\substack{\pk\\ \RT(\pk,\pi)=\overline{\kappa}}} |\pk|^{-\min\{\sigma_{\kappa} : q_*\kappa = \overline{\kappa}\} + M_{\kappa,\tau}\max\{1-\sigma_\tau,0\}+\epsilon}\\
        &=\prod_{\overline{\kappa}\in \RTo{k}(G/T)} \prod_{\substack{\pk\\ \RT(\pk,\pi)=\overline{\kappa}}} |\pk|^{-\min\{\sigma_{\kappa} : q_*\kappa = \overline{\kappa}\} + \sum_{\tau} M_{\kappa,\tau}\max\{1-\sigma_\tau,0\}+\epsilon}.
    \end{align*}
    One directly confirms that the inequalities cutting out $\Omega_T$ are precisely those for which $D_k(T,\pi;\underline{s})$ is meromorphic by \Cref{thm:meromorphic_continuation_of_fiber_series} and which cause each exponent to be $<-1$ (for $\epsilon$ sufficiently small). Thus, on $\Omega_T$ we find that
    \[
        \sum_{\pi\in q_*\Sur(G_k,G)}|D_k(T,\pi;\underline{s})| \le D_k(G/T,\underline{x})
    \]
    for some tuple $\underline{x}$ satisfying $x_{\overline{\tau}} > 1$ for each $\overline{\tau} \in \RTo{k}(G/T)$. As $G/T$ is nilpotent \Cref{prop:nilpotent_absolute_convergence} implies $D_k(G/T;\underline{x})$ convergences absolutely, concluding the proof of the criterion for convergence. The meromorphic continuation of $D_k(G;\underline{s})$ then follows from \Cref{thm:main}. The bounds on the orders of the poles follow precisely from the bounds in \Cref{thm:meromorphic_continuation_of_fiber_series} as follows: the fibers $D_k(T_j,\pi;\underline{s})$ have a pole at $s_\tau = 1$ of order at most $|\tw_{T_j,\widetilde{\pi}}^{-1}(\tau)|$ for each $\tau\in \iota_*\RTo{k}(T)$, so the infinite sum $D_k(G;\underline{s})$ has a pole at $s_\tau=1$ of order at most $\max_{\pi} |\tw_{T_j,\widetilde{\pi}}^{-1}(\tau)|$. Note that when $\tau \not\in \iota_*\RTo{k}(T_j)$, we have $\max_{\pi} |\tw_{T_j,\widetilde{\pi}}^{-1}(\tau)| = 0$. On the other hand, if $\tau \in \iota_*\RTo{k}(T_j)$, then $b_\tau \leq \max_{\pi} |\tw_{T_j,\widetilde{\pi}}^{-1}(\tau)|.$ The required upper bound follows taking the maximum over normal subgroups.

    For part (ii), we specialize to the line $s_\tau = \wt(\tau)s$ for the specific weight function corresponding to our choice of invariant, as in \Cref{prop:specialize}. The poles at $s_\tau = 1$ become poles at $s = 1/\wt(\tau)$, and thus the rightmost pole is at $s=1/a_{\inv}(G)$. This corresponds to the point $s_\tau = \wt(\tau)/a_\inv(G)$, which we assumed is in the region of meromorphicity for $D_k(G,\underline{s})$. Thus we conclude that $s=1/a_\inv(G)$ is contained in the region of meromorphicity of the specialization. Applying \Cref{thm:tauberian_power_saving} proves the asymptotic.

    The upper bound on the degree of the polynomial is precisely one less than the sum of the orders of the poles at $s_\tau = 1$ for which $\wt(\tau) = a_\inv(G)$, which is an upper bound for the order of the pole of the specialization. The lower bound follows from determining a lower bound for the size of a single fiber corresponding to an individual abelian normal subgroup and taking the maximum over the choice of normal subgroup and fiber, as every member of this fiber gives us a $G$-extension. For the discriminant ordering, this is exactly what is done in \cite[Corollary 1.7]{alberts2021}. For a more general invariant, the same argument applies following from \cite[Theorem 5.3]{alberts2021} (or \cite[Theorem 1.2]{alberts-odorney2021}), which is stated for a general invariant and from which the proof of \cite[Corollary 1.7]{alberts2021} immediately generalizes.
\end{proof}

\subsection{A two-dimensional simplification}\label{subsec:2dim} Computing the relevant convex hull is challenging when a large collection of variables is involved. Unfortunately, we can see in \Cref{thm:general_unconditional_nilpotent} that the equations cutting out the region of meromorphicity become quite complicated in terms of as many variables as $G$ has nontrivial tame ramification types.

The problem becomes vastly more manageable when the group $G$ happens to be semiconcentrated in just two abelian subgroups. In that case, one has to take the convex hull of just two regions corresponding to the two subgroups. Especially favorable is the situation when each of the two abelian subgroups contains all but one of the tame ramification types of minimal index. In such a situation, the problem of computing the convex hull essentially reduces to two dimensions.

\begin{proposition}\label{prop:shortcut_convex_hull}
    Let $k$ be a number field, $G$ a nilpotent group, and $T_1,T_2$ two abelian normal subgroups of $G$. Assume the same hybrid subconvexity bounds as in \Cref{thm:general_unconditional_nilpotent} and let $\Omega_{T_j}$ denote the same tubular regions.
    
    Let $\inv$ be an inertial invariant with corresponding weight function $\wt$, let $\mathcal{M} = \{\tau\in \RTo{k}(G) : \wt(\tau) = a_\inv(G)\}$ be the tame ramification types of minimum weight, and suppose there exist two distinct elements $\tau,\kappa\in \mathcal{M}$ such that
    \begin{align*}
        \mathcal{M} - \{\kappa\} &\subseteq \iota_*\RTo{k}(T_1)&\text{and}&&
        \mathcal{M} - \{\tau\} &\subseteq \iota_*\RTo{k}(T_2).
    \end{align*}
    If
    \begin{equation}\label{eq:friendly}
            \alpha_{\tau}\alpha_{\kappa}[k(\zeta_{\tau}):k][k(\zeta_{\kappa}):k]\ind_{|\tau|/[k(\zeta_{\tau}):k]}((q_\tau)_*\kappa)\ind_{|\kappa|/[k(\zeta_{\kappa}):k]}((q_\kappa)_*\tau) < 1,
    \end{equation}
    then the point $(\wt(\tau)/a_{\inv}(G))_\tau$ is contained in the convex hull ${\rm Hull}(\Omega_{T_1}\cup\Omega_{T_2})$. In particular, there exists a nonzero ineffective polynomial $P=P_{k,G,\inv}$ and a constant $\delta = \delta_{k,G, \inv} > 0$ such that
    \[
        \#\mathcal{F}_{\inv, k}(nTd;X) \sim X^{1/a_\inv(G)}P(\log X) + O\left(X^{1/a_\inv(G) - \delta}\right)
    \]
    as in \Cref{thm:general_unconditional_nilpotent}(ii).
\end{proposition}

The condition \eqref{eq:friendly} is straightforward, if occasionally tedious, to check, and is sufficient to prove \Cref{conj:number_field_counting}. As discussed above, the proof of \Cref{prop:shortcut_convex_hull} is done by reducing to two dimensions; namely, to the $(s_\tau,s_\kappa)$-space.

\begin{proof}[Proof of \Cref{prop:shortcut_convex_hull}]
    We summarize the results of \Cref{thm:general_unconditional_nilpotent} in this case. If $\theta \in \mathcal{M}$ and $\theta \ne \tau,\kappa$, that means $\theta$ is a tame ramification type coming from both $T_1$ and $T_2$. As $T_1$ and $T_2$ are abelian, any element in their intersection commutes will all elements of their union. In particular, this means $(q_\gamma)_*\theta$ equals the identity for any $\gamma\in \iota_*\RTo{k}(T_1)\cup \RTo{k}(T_2)$. Thus,
    \[
        M_{\gamma,\theta} = 0
    \]
    whenever $\theta \in \iota_*\RTo{k}(T_1)\cap \iota_*\RTo{k}(T_1)$ and $\gamma \in \iota_*\RTo{k}(T_1)\cup \iota_*\RTo{k}(T_1)$.

    Fix some $\epsilon$. We intersect the regions $\Omega_{T_1}$ and $\Omega_{T_2}$ with the orthant $\Omega_{\mathcal{M},\epsilon,\delta}$ cut out by
    \begin{align*}
        \sigma_\theta &> 1-\delta &&\theta \in \mathcal{M}\\
        \sigma_\theta &> 1 + \epsilon && \theta\in \RTo{k}(G)-\mathcal{M},
    \end{align*}
    where $\delta$ is to be specified later. 
    
    The intersection $\Omega_{T_1}\cap \Omega_{\mathcal{M},\epsilon,\delta}$ is cut out by
    \begin{align*}
        \sigma_{\theta} &> 1-\delta && \text{if }\theta\in \mathcal{M} - \{\tau,\kappa\}\\
        \sigma_{\theta} & > 1 + \epsilon && \text{if }\theta \in\RTo{k}(G) - \mathcal{M}\text{ or } \theta=\kappa\\
        \sigma_{\theta} + \sum_{\gamma\in \RTo{k}(T_1)}M_{\theta,\gamma}\sigma_\gamma &> 1 + \sum_{\gamma\in \RTo{k}(T_1)}M_{\theta, \gamma} && \text{if }\theta \in\RTo{k}(G) - \iota_*\RTo{k}(T_1).
    \end{align*}
    The last line can be simplified a fair amount, in fact most choices of $\theta$ can be made redundant. If $\gamma\not\in\mathcal{M}$, then $M_{\theta,\gamma} \sigma_\gamma > M_{\theta,\gamma}$. So, the sum can be restricted to $\gamma\in \mathcal{M} - \{\kappa\}$. If $\theta\ne\kappa$ in the last line, then $\theta \not\in \mathcal{M}$ so that
    \begin{align*}
        \sigma_{\theta} + \sum_{\substack{\gamma\in \mathcal{M}\\\gamma\ne \kappa}}M_{\theta,\gamma}\sigma_\gamma &> 1 + \epsilon + \sum_{\substack{\gamma\in \mathcal{M}\\\gamma\ne \kappa}}M_{\theta,\gamma}(1-\delta)\\
        &=1 + \sum_{\substack{\gamma\in \mathcal{M}\\\gamma\ne \kappa}}M_{\theta,\gamma} + \epsilon - \delta\left(\sum_{\substack{\gamma\in \mathcal{M}\\\gamma\ne \kappa}}M_{\theta,\gamma}\right).
    \end{align*}
    We can choose $\delta > 0$ sufficiently small to ensure this is larger than $1+\sum M_{\theta,\gamma}$. Finally, when $\theta=\kappa$ we know $M_{\kappa,\gamma}$ vanishes for all $\gamma\in \mathcal{M}$ except $\gamma = \tau$.

    To conclude, we have found that for every $\epsilon > 0$ there exists a $\delta > 0$ such that $\Omega_{T_1}$ contains the region $U_{T_1,\epsilon,\delta}$ cut out by
    \begin{align*}
        \sigma_{\theta} &> 1-\delta && \text{if }\theta\in \mathcal{M} - \{\tau,\kappa\}\\
        \sigma_{\theta} & > 1 + \epsilon && \text{if }\theta \in\RTo{k}(G) - \mathcal{M}\\
        \sigma_{\tau} & > 1-\delta\\
        \sigma_{\kappa} & > 1 + \epsilon\\
        \sigma_{\kappa} + M_{\kappa,\tau}\sigma_\tau &> 1 + M_{\kappa, \tau}.
    \end{align*}
    Flipping the roles of $\tau$ and $\kappa$, the same argument implies that for every $\epsilon > 0$ there exists a $\delta > 0$ such that $\Omega_{T_2}$ contains the region $U_{T_2,\epsilon,\delta}$ cut out by
    \begin{align*}
        \sigma_{\theta} &> 1-\delta && \text{if }\theta\in \mathcal{M} - \{\tau,\kappa\}\\
        \sigma_{\theta} & > 1 + \epsilon && \text{if }\theta \in\RTo{k}(G) - \mathcal{M}\\
        \sigma_{\tau} & > 1+\epsilon\\
        \sigma_{\kappa} & > 1-\delta\\
        \sigma_{\tau} + M_{\tau,\kappa}\sigma_\kappa &> 1 + M_{\tau,\kappa}.
    \end{align*}
    The first two equations cutting out $U_{T_1,\epsilon,\delta}$ and $U_{T_2,\epsilon,\delta}$ are the same, while the last three are given solely in terms of $\sigma_{\tau}$ and $\sigma_{\kappa}$. Thus, it suffices to determine the convex hull in the projection to the $(\sigma_{\tau},\sigma_{\kappa})$-plane.
    
    This two dimensional convex hull is then contained the intersection of the orthant $\sigma_\tau,\sigma_\kappa>1-\delta$ with the region to the right of the segment given by the two corner points
    \[
        \left( 1-\delta, 1 + \delta M_{\kappa, \tau} \right) \textrm{ and } \left(1 + \delta M_{\tau, \kappa}, 1-\delta\right),
    \]
    of the boundary of $U_{T_1,\epsilon,\delta}$ and $U_{T_2,\epsilon,\delta}$ respectively

    The line between these two points cuts out the half-plane
    \begin{align*}
        \left(1+M_{\kappa, \tau}\right)\sigma_\kappa + \left( 1 + M_{\tau, \kappa} \right)\sigma_\tau &> \left(1+M_{\kappa, \tau}\right) + \left( 1 + M_{\tau, \kappa}\right) + \delta\left(M_{\tau,\kappa}M_{\kappa,\tau} - 1\right).
    \end{align*}
    It is clear from this expression that the convex hull of $U_{T_1,\epsilon,\delta}\cup U_{T_2,\epsilon,\delta}$ projected onto the $(\sigma_\tau,\sigma_\kappa)$-plane contains the point $(1,1)$ if and only if
    \[
        M_{\tau,\kappa}M_{\kappa,\tau} < 1,
    \]
    which is exactly \eqref{eq:friendly}. Containing $(1,1)$ is equivalent to containing an open neighborhood of the orthant $\sigma_\tau,\sigma_\kappa\ge 1$ for this region, so by \Cref{thm:general_unconditional_nilpotent}(i) we have shown that $D_k(G;\underline{s})$ is meromorphic on the orthant $\Omega_{\mathcal{M},\epsilon,\delta'}$ for some sufficiently small $\delta' > 0$.

    Consider the point $(\wt(\tau)/a_{\inv}(G))_\tau$. If $\tau\in \mathcal{M}$, then
    \[
        \frac{\wt(\tau)}{a_\inv(G)} \ge 1 > 1 - \delta'.
    \]
    If $\tau\not\in \mathcal{M}$ is tame, then
    \[
        \frac{\wt(\tau)}{a_{\inv}(G)} > 1.
    \]
    Choosing $\epsilon$ sufficiently small implies this is $ > 1+\epsilon$. There are no conditions on the wild ramification types. This implies the point is in $\Omega_{\mathcal{M},\epsilon,\delta'}$ for $\epsilon > 0$ sufficiently small, and therefore lies in the convex hull ${\rm Hull}(\Omega_{T_1}\cup \Omega_{T_2})$. Together with \Cref{thm:general_unconditional_nilpotent}, this concludes the proof.
\end{proof}

\subsection{Analogous results for twisted counting}

We can prove analogous unconditional results for nilpotent normal subgroups $N\normal G$. The proofs are the same except we need to keep track of the $\pi$ dependence given by the factors $h_{ur}^1(k,N,T;\pi)$. We state our results only for the discriminant ordering, thus making them easier to compare with the existing literature. However, our methods extend to all inertial invariants.

\begin{theorem}\label{thm:unconditional_concentrated}
    Let $k$ be a number field and $G= nTd$ a transitive group which is semiconcentrated in abelian normal subgroups $T_1,T_2,...,T_r$.

    Suppose there exists at least on $G$-extension of $k$ and that there exists a nilpotent normal subgroup $N\normal G$ which contains each of $T_1,T_2,\dots T_r$, with quotient map $q:G\to G/N$, and a $\theta \ge 0$ such that
    \begin{align}\label{eq:sum_h1ur_uncond}
        \sum_{\pi\in q_*\Sur(G_k,G;X)} \max_{j} h_{ur}^1(k,N,T_j,\pi) \ll_{n,k} X^{\theta},
    \end{align}
    where $h^1_{ur}(k,N,T_j,\pi)$ is given by Definition \ref{def:h1ur}.
    
    Assume the point $(\ind_n(\tau)/a(nTd))_\tau$ lies in the convex hull ${\rm Hull}(\bigcup_j \Omega_{T_j})$ where $\Omega_{T_1},\dots,\Omega_{T_r}$ are defined as in \Cref{thm:general_unconditional_nilpotent}.
    \begin{enumerate}[(i)]
        \item If $\theta < 1/a(N)$ then there exist ineffective constants $c\in \R_{>0}$ and $b\in \Z_{>0}$ such that
        \[
            \#\mathcal{F}_{\disc,k}(nTd;X) \sim cX^{1/a(N)}(\log X)^{b-1},
        \]
        with $b = \max_{\pi\in q_*\Sur(G_k,G)} b(\pi)$ with $b(\pi)$ defined as in \Cref{thm:unconditional_nilpotent_fibers}. In particular,
        \begin{equation}\label{eq:bbounds_nilpotent}
        \max_{\substack{j=1,\dots,r\\\psi\in (q_j)_*\Sur(G_k,G)}}\left\{\sum_{\substack{\tau\in \iota_*\RTo{k}(T_j)\\\ind_n(\tau) = a(G)}} |\tw_{T_j,\widetilde{\psi}}^{-1}(\tau)|\right\} \le b \le  \sum_{\substack{\tau\in \bigcup_j\iota_*\RTo{k}(T_j)\\\ind_n(\tau) = a(G)}} \left(\max_{\substack{j=1,\dots,r\\\psi\in (q_j)_*\Sur(G_k,G)}}|\tw_{T_j,\widetilde{\psi}}^{-1}(\tau)|\right).
        \end{equation}
        \item If $\theta \ge 1/a(N)$ then
        \[
            \#\mathcal{F}_{\disc,k}(nTd;X) \ll_{n,k} X^{\theta + \epsilon}.
        \]
    \end{enumerate}
\end{theorem}

This result is a combination of the ideas in \Cref{thm:general_unconditional_nilpotent} and the inductive method in \cite{ALOWW}. Specifically, we can directly verify the hypotheses of \cite[Theorem 2.1]{ALOWW} for the subgroup $N\normal G$.

The first hypothesis is ``precise counting of fibers'':

\begin{theorem}\label{thm:unconditional_nilpotent_fibers}
    Let $k$ be a number field, $G = nTd$ a transitive group, and $N\normal G$ a nilpotent normal subgroup with quotient map $q:G\to G/N$. Suppose there exists at least one $G$-extension of $k$ and there exist abelian normal subgroups $T_1,T_2,\dots T_r\normal G$ for which
    \[
        \{g\in N : \ind_n(g) = a(N)\} \subseteq \bigcup_{j=1}^r T_j \subseteq N.
    \]

    If the point $(\ind_n(\tau)/a(nTd))_\tau$ lies in the convex hull ${\rm Hull}(\bigcup_j \Omega_{T_j})$ where $\Omega_{T_1},\dots,\Omega_{T_r}$ are defined as in \Cref{thm:general_unconditional_nilpotent}, then for each $\pi \in q_*\Sur(G_k,G)$ there exist ineffective constants $c(\pi) \in \R_{>0}$ and $b(\pi) \in \Z_{>0}$ for which
    \[
        \#\{\psi\in q_*^{-1}(\pi) : |\disc(\psi)|\le X\} \sim c(\pi) X^{1/a(N)}(\log X)^{b(\pi)-1},
    \]
    in particular satisfying \Cref{conj:twisted_number_field_counting}. Moreover, the exponent $b$ satisfies the bounds
    \[
        \max_{\substack{j=1,\dots,r\\\psi\in (q_j)_*\Sur(G_k,G)\\q_*(\psi) = \pi}}\left\{\sum_{\substack{\tau\in \iota_*\RTo{k}(T_j)\\\ind_n(\tau) = a(G)}} |\tw_{T_j,\widetilde{\psi}}^{-1}(\tau)|\right\} \le b(\pi) \le  \left(\sum_{\substack{\tau\in \bigcup_j\iota_*\RTo{k}(T_j)\\\ind_n(\tau) = a(G)}} b_\tau(\pi)\right),
    \]
    where
    \[
        b_\tau(\pi) = \max_{j=1,\dots,r}\left(\max_{\substack{\psi\in (q_j)_*\Sur(G_k,G)\\q_*(\psi) = \pi}}|\tw^{-1}_{T_j,\widetilde{\psi}}(\tau)|\right).
    \]
\end{theorem}

\begin{proof}
    The proof is completely analogous to part (ii) of \Cref{thm:general_unconditional_nilpotent}, except we use the bound
    \[
        |H^1_{ur}(k,T(\psi))| \ll_{[k:\Q],|G|,\epsilon} h^1_{ur}(k,N,T,\pi)|\disc(\psi)|^{\epsilon}
    \]
    proven in \Cref{prop:h1ur_bound}, which is independent of $\psi\in (q_{T_j,N})_*^{-1}(\pi)$ (recall $q_{T_j,N}:G/T_j\to G/N$ is the quotient map). We bound each $h^1_{ur}(k,N,T_j,\pi)$ above by
    \[
        \max_j h^1_{ur}(k,N,T_j,\pi),
    \]
    which is now independent of $j\in \{1,2,\dots,r\}$. This gives the bound
    \begin{align*}
        \sum_{\psi\in (q_{T_j,N})_*^{-1}(\pi)} |D_k(T,\psi;\underline{s})| \ll_{\underline{t}}&\max_j h^1_{ur}(k,N,T_j,\pi) \sum_{\psi\in (q_{T_j,N})_*^{-1}(\pi)}\inv_{k,G/T_j}(\psi,(q_{T_j})_*\underline{\sigma}+\epsilon),
    \end{align*}
    where we note that $\abs{\disc(\psi)}^\epsilon$ is negligible and can be absorbed by $\inv_{k,G/T_j}(\psi,(q_{T_j})_*\underline{\sigma}+\epsilon).$
    
    As $N$ is nilpotent, the number of $N$ extensions of a given discriminant $D$ is bounded by $D^{\epsilon}$. Moreover, by definition $\inv_{k,G/T_j}(\psi,(q_{T_j})_*\underline{\sigma}+\epsilon)$ is divisible by $\inv_{k,G/N}((q_N)_*\pi+\epsilon)$, so we can bound
    \begin{align*}
        \sum_{\psi\in (q_{T_j,N})_*^{-1}(\pi)} |D_k(T,\psi;\underline{s})| \ll_{\underline{t}}& \max_j h^1_{ur}(k,N,T_j,\pi) \inv_{k,G/N}((q_N)_*\pi, q_*\underline{\sigma}+\epsilon)\\
        &\cdot \sum_{\underline{n}} \prod_{\overline{\tau}\in \iota_*\RTo{k}(N/T_j)} n_{\overline{\tau}}^{-\epsilon - (q_{T_j})_*(\underline{\sigma})_{\overline{\tau}}},
    \end{align*}
    which converges absolutely if $(q_{T_j})_*(\underline{\sigma})_{\overline{\tau}} > 1$ for each $\overline{\tau}\in \iota_*\RTo{k}(N/T_j)$, or equivalently $\sigma_\tau > 1$ for each $\tau \in \iota_*(\RTo{k}(N) - \RTo{k}(T_j))$. This is sufficient to apply \Cref{thm:main}, from which the proof (including the bounds for $b$) is identical to that of \Cref{thm:general_unconditional_nilpotent}(ii).
\end{proof}

For the second hypothesis, we require the notion of pushforward discriminant from \cite[Equation (5.2)]{ALOWW}, which we recall here.

\begin{definition}\label{def:push_disc}
    Let $q : G \to H$ be a group homomorphism. The {\em pushforward discriminant} of a map $\pi \in \Hom(G_k, H)$ is
    \[
        q_*\disc(\pi) = \prod_{\pk} {\pk}^{q_*f_{\pk}(\pi_{\pk}) }
    \]
    where
    \[
        q_*f_{\pk}(\pi_{\pk}) = \begin{cases}
        \min \left\{ f_{\pk} (\psi_{\pk}); \psi_{\pk} \in \Hom(G_{k_\pk}, G) \textrm{ with } q \circ \psi_\pk = \pi_\pk\right\} & \textrm{ if a lift $\psi_\pk$ of $\pi_\pk$ exists,}\\
        &\\
        \infty & \textrm{ otherwise,}
        \end{cases}
    \]
    and $f_{\pk}(\psi_\pk)$ is the local Artin conductor of $\psi_{\pk}$, satisfying $\disc(\psi_\pk) = p^{f_\pk(\psi_\pk)}$.
\end{definition}

The second hypothesis is an ``upper bounds on fibers'':

\begin{theorem}\label{thm:unconditional_nilpotent_fibers_uniform}
    Let $k$, $G = nTd$, $N$, $q$, and $T_1,T_2,\dots,T_r$ be as in \Cref{thm:unconditional_nilpotent_fibers}. 

    If the point $(\ind_n(\tau)/a_{\disc}(nTd))_\tau$  lies in the convex hull ${\rm Hull}(\bigcup_j \Omega_{T_j})$ where $\Omega_{T_1},\dots,\Omega_{T_r}$ are defined as in \Cref{thm:general_unconditional_nilpotent}, then for each $\pi \in q_*\Sur(G_k,G)$
    \[
        \#\{\psi\in q_*^{-1}(\pi) : |\disc(\psi)|\le X\}  = O_{n,[k:\Q],\epsilon}\left(\frac{\max_j h^1_{ur}(k,N,T_j,\pi)}{(q_*\disc_G(\pi))^{1/a_{\disc}(N) - \epsilon}}X^{1/a_{\disc}(N)}(\log X)^{b(\pi)-1}\right),
    \]
    where $b(\pi) \in \Z_{>0}$ is the same ineffective constant from \Cref{thm:unconditional_nilpotent_fibers}, $q_*\disc_G$ is the pushforward discriminant and $h^1_{ur}(k,N,T_j,\pi)$ is the quantity defined in \Cref{def:h1ur}.
\end{theorem}

\begin{proof}
    The proof of this result is now almost identical to that of \cite[Theorem 6.1]{ALOWW}. When we specialize to the single variable series $D_k(N,\pi;\ind(\tau)s)$, the vertical bound is now of the form
    \[
        |D_k(N,\pi;\ind(\tau)s)| \ll_{[k:\Q],|G|,t,\epsilon} \max_j h^1_{ur}(k,N,T_j,\pi)\inv_{k,G/N}(q_*\pi, q_*(\ind(\tau)\sigma)+\epsilon)\prod_{\tau} \left(\frac{s_\tau+1}{s_\tau-1}\right)^{b_\tau(\pi)}.
    \]
    The term $\inv_{k,G/N}(q_*\pi, q_*(\ind(\tau)\sigma)+\epsilon)$ is (up to a constant given by the wild places) bounded above by $|q_*\disc(\pi)|^{-s+\epsilon}$, which is immediate from the definition of the pushforward discriminant (\Cref{def:push_disc}). Bounding above by a smoothed count gives 
    \begin{align*}
        \#\{\psi\in q_*^{-1}(\pi) : |\disc(\psi)|\le X\} &\ll \sum_{\psi\in q_*^{-1}(\pi)} \exp\left(1 - \frac{|\disc(\psi)|}{X}\right)\\
        &= \frac{e}{2\pi i}\int_{{\rm Re}(s)=2} D_k(N,\pi;\ind(\tau)s) \Gamma(s) X^s ds.
    \end{align*}
    Shifting the contour integral to ${\rm Re}(s) = 1/a(N) - \epsilon$ for $\epsilon > 0$ sufficiently small, the superpolynomial decay of $\Gamma(s)$ implies the upper bound
    \[
        \ll \underset{s=1/a(N)}{\rm Res}\left(D_k(G,\ind(\tau)s) \Gamma(s) X^s\right) + O_{[k:\Q],|G|,\epsilon}\left(\frac{\max_j h^1_{ur}(k,N,T_j,\pi)}{|q_*\disc(\pi)|^{1/a(N)-\epsilon}} X^{1/a(N) - \epsilon}\right).
    \]
    It now suffices to bound the residue using the Cauchy integral formula. The residue is of the form $X^{1/a(N)}$ times a polynomial in $\log X$ whose coefficients are linear combinations of Laurent coefficients of $D_k(N,\pi;\ind(\tau)s) \Gamma(s)$. Let $C$ be a circle of radius $\epsilon$ centered at $1/a(N)$. Using the Cauchy integral formula, we can bound each such Laurent coefficient above by
    \begin{align*}
        \frac{1}{2\pi i}\int_{C_{\epsilon}} \frac{|D_k(N,\pi;\ind(\tau)s) \Gamma(s)|}{(s-1/a(N))^{m+1}} ds &\ll_{[k:\Q],|G|,\epsilon} \frac{\max_j h^1_{ur}(k,N,T_j,\pi)}{|q_*\disc(\pi)|^{1/a(N)-\epsilon}} \int_{C_\epsilon} \epsilon^{O_{m}(1)}(1+\epsilon)^{O_\beta(1)}ds\\
        &\ll_{[k:\Q],|G|,\epsilon} \frac{\max_j h^1_{ur}(k,N,T_j,\pi)}{|q_*\disc(\pi)|^{1/a(N)-\epsilon}},
    \end{align*}
    where $\beta = (\beta_\tau)_{\tau}$ is the $t$-aspect of the subconvexity bounds for the Hecke $L$-functions appearing in $D_k(T,\psi;\underline{s})$. Utilizing this to bound the coefficients of the residue gives the desired upper bound.
\end{proof}

The conclusions of \Cref{thm:unconditional_concentrated} follow immediately from \cite[Theorem 2.1]{ALOWW}, \Cref{thm:unconditional_nilpotent_fibers}, and \Cref{thm:unconditional_nilpotent_fibers_uniform}.

\begin{proof}[Proof of \Cref{thm:unconditional_concentrated}]
    We check the hypotheses of \cite[Theorem 2.1]{ALOWW} for the normal subgroup $N\normal G$. ``Precise counting of fibers'' is given by \Cref{thm:unconditional_nilpotent_fibers}. ``Upper bounds on fibers'' is given by \Cref{thm:unconditional_nilpotent_fibers_uniform}.

    Lastly, the ''criterion for convergence'' requires us to prove that
    \[
        \sum_{\pi\in q_*\Sur(G_k,G;X)} \frac{\max_j h_{ur}^1(k,N,T_j,\pi)}{(q_*\disc_G(\pi))^{1/a_{\disc}(N)-\epsilon}}
    \]
    converges. Equation \eqref{eq:sum_h1ur_uncond} and Abel summation implies that this sum converges precisely when $\theta < 1/a_{\disc}(N)$. This proves part (i).

    Part (ii) follows from taking a the sum of the upper bounds in \Cref{thm:unconditional_nilpotent_fibers_uniform} over $q_*\disc_G(\pi)\le X$, again using Abel summation.
\end{proof}

In case the group $G$ happens to be semiconcentrated in just two abelian subgroups we can use our two dimensional simplification for computing the convex hull (\Cref{prop:shortcut_convex_hull}) and get the following simplified version of \Cref{thm:unconditional_concentrated}.

\begin{proposition}\label{prop:shortcut_convex_hull_twisted}
    Let $k$ be a number field, $G$ a transitive group with label $nTd$ for which there exists at least one $G$-extension of $k$, $N\normal G$ a nilpotent normal subgroup, and $T_1,T_2$ two abelian normal subgroups of $G$ contained in $N$. Assume the same hybrid subconvexity bounds as in \Cref{thm:general_unconditional_nilpotent} and let $\Omega_{T_j}$ denote the same tubular regions.
    
    Let $\mathcal{M} = \{\tau\in \RTo{k}(G) : \ind_n(\tau) = a(G)\}$ be the tame ramification types of minimum index, and suppose there exist two distinct elements $\tau,\kappa\in \mathcal{M}$ such that
    \begin{align*}
        \mathcal{M} - \{\kappa\} &\subseteq \iota_*\RTo{k}(T_1)&\text{and}&&
        \mathcal{M} - \{\tau\} &\subseteq \iota_*\RTo{k}(T_2).
    \end{align*}
    If
    \begin{equation}\label{eq:friendly_twisted}
            \alpha_{\tau}\alpha_{\kappa}[k(\zeta_{\tau}):k][k(\zeta_{\kappa}):k]\ind_{|\tau|/[k(\zeta_{\tau}):k]}((q_\tau)_*\kappa)\ind_{|\kappa|/[k(\zeta_{\kappa}):k]}((q_\kappa)_*\tau) < 1,
    \end{equation}
    then the point $(\ind_n(\tau)/a(G))_\tau$ is contained in the convex hull ${\rm Hull}(\Omega_{T_1}\cup\Omega_{T_2})$. In particular, there exists a nonzero ineffective polynomial $P$ and an effective constant $\delta = \delta_{k,G} > 0$ such that
    \[
        \#\mathcal{F}_{\disc,k}(G;X) \sim cX^{1/a(G)}P(\log X) + O\left(X^{1/a(G) - \delta}\right)
    \]
    as in \Cref{thm:unconditional_concentrated}(i).
\end{proposition}

The proof is identical to that of \Cref{prop:shortcut_convex_hull}.

\section{Proofs of the Unconditional Examples}\label{sec:unconditional_corollaries}

\subsection{$Q_8\rtimes C_2$ as 8T11 and 16T11}\label{sec:16T11}

We now prove \Cref{thm:Q8sxC2} and \Cref{thm:16T11_over_k} for $G = Q_8 \rtimes C_2$. We adopt the conjugacy class notation used in the LMFDB (\cite[\href{https://www.lmfdb.org/GaloisGroup/16T11}{Transitive Group 16T11}]{lmfdb} in degree $16$ and \cite[\href{https://www.lmfdb.org/GaloisGroup/8T11}{Transitive Group 8T11}]{lmfdb} in degree $8$).

\begin{table}[ht]
\begin{tabular}{llllll}
     Label & Size& Order& Degree $8$ & Degree $16$ & Degree $8$ \\
     &&&Index&Index&Representative \\
     \hline
    1A    &   $1$    &   $1$    &   $0$   &   $0$    &   $()$\\
    2A    &   $1$    &   $2$    &   $4$   &   $8$    &   $(1,5)(2,6)(3,7)(4,8)$ \\
    2B    &   $2$    &   $2$    &   $4$   &   $8$    &   $(1,6)(2,5)(3,8)(4,7)$\\ 
    2C    &   $2$    &   $2$    &   $2$   &   $8$    &   $(1,5)(3,7)$\\
    2D    &   $2$    &   $2$    &   $4$   &   $8$    &   $(1,4)(2,7)(3,6)(5,8)$\\
    4A1   &   $1$    &   $4$    &   $6$   &   $12$    &  $(1,3,5,7)(2,4,6,8)$\\
    4A-1  &   $1$    &   $4$    &   $6$   &   $12$    &   $(1,7,5,3)(2,8,6,4)$\\
    4B    &   $2$    &   $4$    &   $6$   &   $12$    &   $(1,8,5,4)(2,7,6,3)$\\
    4C    &   $2$    &   $4$    &   $6$   &   $12$    & $(1,3,5,7)(2,8,6,4)$   \\
    4D    &   $2$    &   $4$    &   $6$   &   $12$    & $(1,6,5,2)(3,8,7,4)$ \\
\end{tabular}
\caption{Conjugacy Classes of $Q_8\rtimes C_2$}
\label{tab:16T11}
\end{table}

We note that over $\Q$ the conjugacy classes $4A1$ and $4A\text{-}1$ are in the same cyclotomic orbit (as $4A\text{-}1$ is the inverse of $4A1$) and so their union is a single tame $\Q$-ramification type. We write $4A$ for the union $4A1\cup 4A\text{-}1$. The other conjugacy classes are distinct ramification types, so that $Q_8\rtimes C_2$ has eight nontrivial tame $\Q$-ramification types.

We will consider three normal subgroups generated by the conjugacy classes $2B$, $2C$, and $2D$. 

\begin{align*}
    T_B &= \langle 2B \rangle = 1A\cup 2A \cup 2B \cong C_2\times C_2\\
    T_C &= \langle 2C \rangle = 1A\cup 2A \cup 2C \cong C_2\times C_2\\
    T_D &= \langle 2D \rangle = 1A\cup 2A \cup 2D \cong C_2\times C_2.
\end{align*}

For each $X \in \{B, C, D\}$ the quotient $G/T_X$ is isomorphic to $C_2\times C_2$, which has three nontrivial tame ramification types. We will focus on the picture for $T_B$, as the other two subgroups work similarly by symmetry.

By \Cref{thm:meromorphic_continuation_of_fiber_series}, the fiber series $D_\Q(T_B,\pi;\underline{s})$ has a meromorphic continuation to the region cut out by $\sigma_{2A} > 1/2$ and $\sigma_{2B} > 1/2$ for each $\pi\in (q_B)_*\Sur(G_\Q,Q_8\rtimes C_2)$. Noting that $\tw_{T_B,\widetilde{\pi}}$ is injective in this case and $2A$ is central, \Cref{thm:meromorphic_continuation_of_fiber_series} gives the vertical bound
\begin{align*}
    D_{\Q}\left(T_B,\pi,\underline{s} \right) \ll_{\epsilon} &|H^1_{ur}(\Q,T_B(\pi))| \cdot \inv_{\Q,G/T_B}(\pi,q_*\underline{\sigma}+\epsilon)\\
    &\cdot \max_{\chi_0}\{|L(s_{2A}, \chi_0)|,1\} \cdot \max_{\chi}\{|L_{F(\pi,\tau)}(s_{2B},\chi)|,1\},
\end{align*}
where the first maximum is over rational Dirichlet characters of conductor dividing $8$ and the second maximum is over quadratic Hecke character $\chi$ over $F(\pi,\tau)$ unramified away from $\{2,\infty\}$, where $F(\pi,\tau)$ is the field fixed by $\Stab_{G_\Q}(g)$ for some $g\in 2B$, which in this case is exactly $\pi^{-1}(C_G(T_B))$ the preimage of the centralizer of $T_B$ under $\pi$. The centralizer is given by the group
\[
    C_G(T_B) = 1A \cup 2A \cup 4A \cup 2B \cup 4B
\]
of order $8$. Thus, $F(\pi,\tau)$ is a quadratic field.

The $H^1_{ur}$ piece can be bounded by $\disc(\pi)^{\epsilon}$ using \cite[Corollary 1.14(i)]{ALOWW}. Using the Weyl strength subconvexity bound for the Dirichlet $L$-functions in the $s_{2A}$ variable \cite[Theorem 1.1]{petrow2023fourth}, and otherwise using Yang's subconvexity bounds \cite[Corollary 1.2]{yang2023burgess} for the Hecke $L$-functions, we conclude that
\begin{align*}
    D_{\Q}\left(T_B,\pi,\underline{s} \right) \ll_{\epsilon} &\inv_{\Q,G/T_B}(\pi,q_*\underline{\sigma}+\epsilon)\abs{\disc(F(\pi,\tau)/\Q)}^{\frac{3}{8}\max\{1-\sigma_{2B},0\} + \epsilon}\\
    &\cdot(1+|\underline{t}|)^{\frac{1}{3}\max\{1-\sigma_{2A},0\} + \frac{3}{4}\max\{1-\sigma_{2B},0\}+\epsilon}.
\end{align*}
on compact subsets of the region of meromorphicity.

We separate the tame primes that divide the discriminant $\disc(F(\pi,\tau)/\Q)$, which are precisely those for which $\RT(p,\pi) \in \{2C, 2D, 4C, 4D\}$. By expanding out $\inv_{\Q,G/T_B}$ and $\disc(F(\pi,\tau)/\Q)$ in terms of the two tame ramification types of $G/C_G(T_B)$, the conductor aspect of the upper bound is
\begin{align*}
   D_{\Q}\left( T_B,\pi,\underline{s} \right)& \ll_{\underline{t},\epsilon}  \prod_{\RT(p, \pi) \in q\{4A, 4B\}} p^{-\min\{\sigma_{4A},\sigma_{4B}\} + \epsilon} \\
   &\hspace*{0.1cm}
   \cdot \prod_{\RT(p, \pi) \in q\{2C,2D,4C,4D\}} p^{-\min\{\sigma_{2C},\sigma_{2D},\sigma_{4C},\sigma_{4D}\} + \frac{3}{8}\max\{1 - \sigma_{2B},0\} + \epsilon} 
\end{align*}
We find that the sum over $\pi$ converges absolutely on compact subsets of the region $\Omega_B$ cut out by the following equations:
\begin{align*}
    &\sigma_{2A},\sigma_{2B} > 1/2; &&\sigma_{2C},\sigma_{2D},\sigma_{4A},\sigma_{4B},\sigma_{4C},\sigma_{4D} > 1;\\
    &\sigma_{2C} + \frac{3}{8}\sigma_{2B} > \frac{11}{8};&
    &\sigma_{2D} + \frac{3}{8}\sigma_{2B} > \frac{11}{8};\\
    &\sigma_{4C} + \frac{3}{8}\sigma_{2B} > \frac{11}{8};&
    &\sigma_{4D} + \frac{3}{8}\sigma_{2B} > \frac{11}{8}.
\end{align*}
This contains the region
\[
    \Omega_B \supseteq U_B:=\left\{\underline{s} :\substack{\displaystyle \sigma_{2A},\sigma_{2B}> 1/2;\ \sigma_{2C},\sigma_{2D},\sigma_{4A}> 1;\ \sigma_{4B},\sigma_{4C},\sigma_{4D} > 19/16\\ \displaystyle \sigma_{2C} + \frac{3}{8}\sigma_{2B} > \frac{11}{8};\ \sigma_{2D} + \frac{3}{8}\sigma_{2B} > \frac{11}{8}}\right\}.
\]
All together, we have shown that $D_\Q(Q_8\rtimes C_2;\underline{s})$ is meromorphic on $U_B$ with possible simple polar divisors at $s_{2A}=1$ and $s_{2B}=1$, and on this region
\[
    |D_\Q(Q_8\rtimes C_2;\underline{s})|\ll (1+|\underline{t}|)^{\frac{1}{3}\max\{1-\sigma_{2A},0\} + \frac{3}{4}\max\{1-\sigma_{2B},0\} + \epsilon} \ll (1+|\underline{t}|)^{\frac{1}{3}\max\{1-\sigma_{2A},0\} + \frac{3}{8} + \epsilon}.
\]
By symmetry, $D_\Q(Q_8\rtimes C_2;\underline{s})$ has a meromorphic continuation to $U_B\cup U_C\cup U_D$ with possible simple polar divisors at $s_{2A} = 1$, $s_{2B} = 1$, $s_{2C} = 1$, and $s_{2D} = 1$, where these regions are given by
\[
    U_C:=\left\{\underline{s} :\substack{\displaystyle \sigma_{2A},\sigma_{2C}> 1/2;\ \sigma_{2B},\sigma_{2D},\sigma_{4A} > 1;\ \sigma_{4B},,\sigma_{4C},\sigma_{4D} > 19/16\\ \displaystyle \sigma_{2B} + \frac{3}{8}\sigma_{2C} > \frac{11}{8};\ \sigma_{2D} + \frac{3}{8}\sigma_{2C} > \frac{11}{8}}\right\}.
\]
and
\[
    U_D:=\left\{\underline{s} :\substack{\displaystyle \sigma_{2A},\sigma_{2D}> 1/2;\ \sigma_{2B},\sigma_{2C},\sigma_{4A} > 1;\ \sigma_{4B},\sigma_{4C},\sigma_{4D} > 19/16\\ \displaystyle \sigma_{2C} + \frac{3}{8}\sigma_{2D} > \frac{11}{8};\ \sigma_{2B} + \frac{3}{8}\sigma_{2D} > \frac{11}{8}}\right\}.
\]
Moreover, the same bound in vertical strips $\ll(1+|t|)^{\frac{1}{3}\max\{1-\sigma_{2A},0\}+\frac{3}{8}+\epsilon}$ holds on the entire union $U_B \cup U_C \cup U_D$. By Bochner's tube theorem, the same holds in the convex hull. We claim that the convex hull contains the set
\[
    \Omega = \left\{\underline{s} :\substack{\displaystyle \sigma_{2A},\sigma_{2B},\sigma_{2C},\sigma_{2D} > 1/2;\ \sigma_{4A} > 1;\\\displaystyle \sigma_{4B}, \sigma_{4C} , \sigma_{4D} > 19/16;\\ \displaystyle \sigma_{2B} + \sigma_{2C} + \sigma_{2D} > 23/8;\\
    \displaystyle \sigma_{2B} + \sigma_{2C} > 27/16;\\ \displaystyle  \sigma_{2B} + \sigma_{2D} > 27/16; \\ \displaystyle  \sigma_{2C} + \sigma_{2D} > 27/16\\
    }\right\}.
\]
The equations $\sigma_{2A} > 1/2$, $\sigma_{4A} > 1$, and $\sigma_{4B}, \sigma_{4C} , \sigma_{4D} > 19/16$ are part of the defining data of each of $U_B$, $U_C$, and $U_D$, and so must be for the convex hull of their union. Thus, it suffices to construct the convex hull of the three dimensional shape cut out by the equations in $(\sigma_{2B},\sigma_{2C},\sigma_{2D})$. The same vertical bound applies to every point on the convex by \Cref{cor:independent_variable_multi_phragmen_lindelof}.

To assist with visualizing the convex hull, we plotted a three dimensional picture of this region in the variables $(\sigma_{2B},\sigma_{2C},\sigma_{2D}) = (x,y,z)$ using Desmos \cite{desmos}. An image of this graph is included in \Cref{fig:16T11}, and the Desmos graph itself can be found at \url{https://www.desmos.com/3d/ndwrqldvhr}.

\begin{figure}[!ht]
    \includegraphics[scale=0.5]{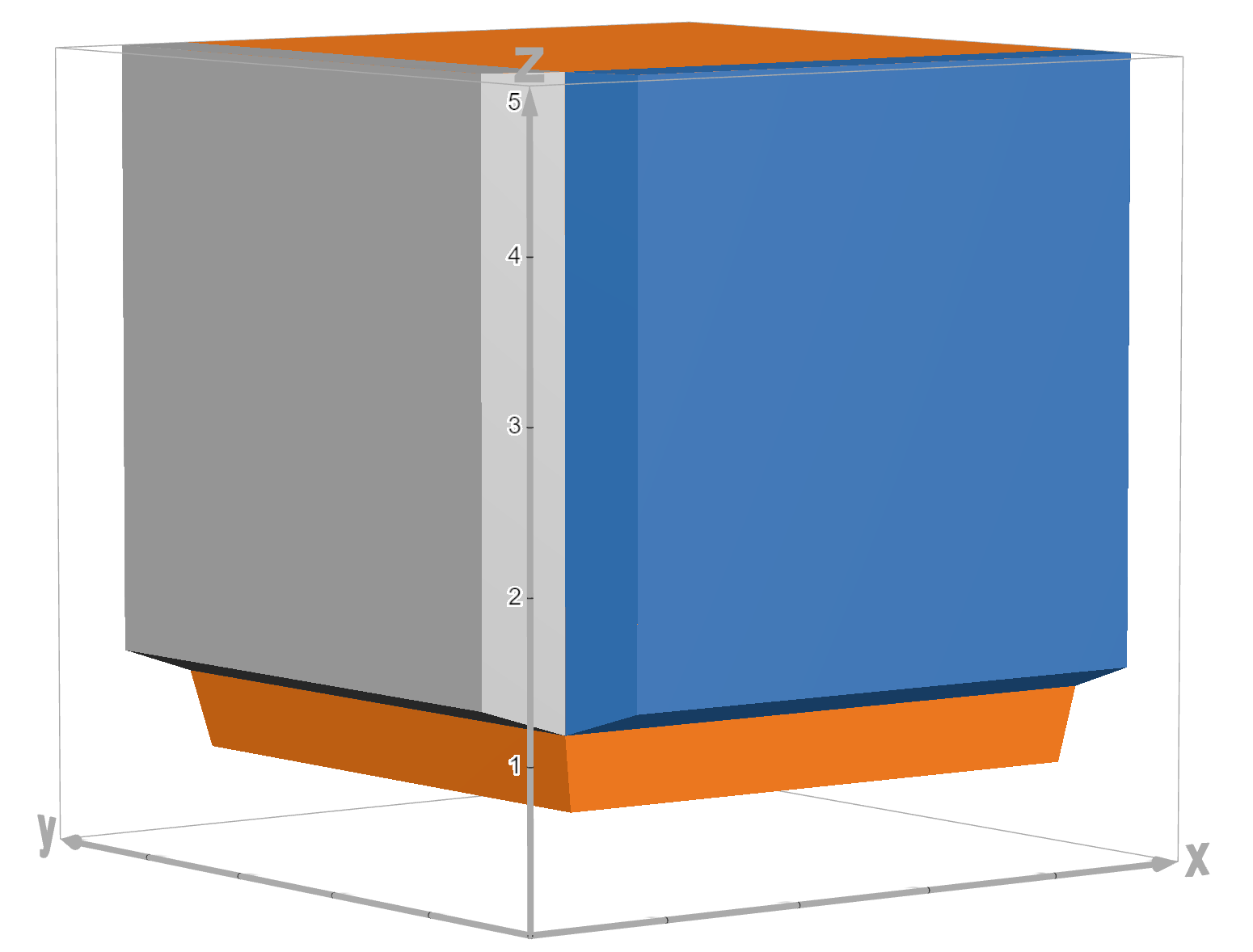}
    \caption{\label{fig:16T11} Projection of $U_B\cup U_C \cup U_D$ onto\\the $(\sigma_{2B},\sigma_{2C},\sigma_{2D}) = (x,y,z)$ space}
\end{figure}

From the image, it is clear to see that the convex hull of this region is the orthant $\sigma_{2B},\sigma_{2C},\sigma_{2D} > 1/2$ with the edges and corner sliced off. We can compute these by computing the convex hull between the various corners of $U_{B}$, $U_{C}$, and $U_{D}$ (these are the outer corners in \Cref{fig:16T11}).

Those three corners are $(1/2,19/16,19/16)$, $(19/16,1/2,19/16)$, and $(19/16,19/16,1/2)$. Slicing off the $\sigma_{2D} = 1/2$ edge comes from the convex hull of $(1/2,19/16,19/16)$ and $(19/16,1/2,19/16)$, which is precisely the equation
\[
    \sigma_{2B} + \sigma_{2C} > 27/16.
\]
A similar procedure applies to cutting off the $\sigma_{2B}=1/2$ and $\sigma_{2C}=1/2$ edges.

Finally, cutting of the corner $\sigma_{2B}=\sigma_{2C}=\sigma_{2D}=1/2$ from the orthant we need the convex hull of all three points, which is precisely
\[
    \sigma_{2B} + \sigma_{2C} + \sigma_{2D} > 23/8.
\]
This concludes the proof that $\Omega$ is contained in the convex hull.

We now check the individual transitive representations by specializing to the relevant complex lines $\underline{s} = (\ind_n(\tau)s)$.

\subsubsection{Degree eight representation}

In this case, $D_\Q(Q_8\rtimes C_2;(\ind_8(\tau)s)_\tau)$ is meromorphic on
\[
    \left\{s :\substack{\displaystyle 4\sigma,2\sigma > 1/2;\ 6\sigma > 1;\ 6 \sigma > 19/16;\\
    \displaystyle 4\sigma + 2\sigma + 4\sigma > 23/8;\\
    \displaystyle 4\sigma + 2\sigma > 27/16;\\
    \displaystyle 4\sigma + 4\sigma > 27/16
    }
    \right\} = \left\{ s : \sigma > 23/80\right\},
\]
with a possible pole at $s=1/2$ of order $\le 1$ in this region. Moreover, this series satisfies the vertical bounds
\[
    |D_\Q(Q_8\rtimes C_2;(\ind_8(\tau)s)_\tau)| \ll (1+|t|)^{\frac{1}{3}\max\{1-4\sigma,0\} + \frac{3}{8} + \epsilon} \ll (1+|t|)^{\frac{3}{8} + \epsilon}.
\]
By applying \Cref{thm:tauberian_power_saving} with $\sigma_a = 1/2$, $\delta = 17/80$, and $\xi = 3/8$ we conclude that the power saving exponent is
\[
    \frac{1}{2} - \frac{17/80}{3/8+1}+\epsilon = \frac{19}{55}+\epsilon.
\]

\subsubsection{Degree sixteen representation}

In this case, $D_\Q(Q_8\rtimes C_2;(\ind_{16}(\tau)s)_\tau)$ is meromorphic on
\[
    \left\{s :\substack{\displaystyle 8\sigma > 1/2;\ 12\sigma > 1;\ 12 \sigma > 19/16;\\
    \displaystyle 8\sigma + 8\sigma + 8\sigma > 23/8;\\
    \displaystyle 8 \sigma + 8 \sigma > 27/16
    }\right\} = \left\{ s : \sigma > 23/192\right\},
\]
with a possible pole at $s=1/8$ of order $\le 4$ in this region. Moreover, this series satisfies the vertical bounds
\[
    |D_\Q(Q_8\rtimes C_2;(\ind_8(\tau)s)_\tau)| \ll (1+|t|)^{\frac{1}{3}\max\{1-8\sigma,0\} + \frac{3}{8} + \epsilon} \ll (1+|t|)^{\frac{7}{18} + \epsilon}.
\]
By applying \Cref{thm:tauberian_power_saving} with $\sigma_a = 1/8$, $\delta = 1/192$, and $\xi = 7/18$ we conclude that the power saving exponent is
\[
    \frac{1}{8} - \frac{1/192}{7/18+1}+\epsilon = \frac{97}{800}+\epsilon.
\]

\subsubsection{Degree sixteen representation over an arbitrary number field}

The exact same arguments as above apply when $\Q$ is replaced by another number field $k$ to prove \Cref{thm:16T11_over_k}. There are essentially only two changes:
\begin{itemize}
    \item If $\Q(i) \subseteq k$, then $4A$ and $4A\text{-}1$ are separate $k$-conjugacy classes, and must be assigned separate complex variables. This does not have a significant impact on the region of meromorphicity, as neither $4A$ or $4A\text{-}1$ are minimum index conjugacy classes.
    \item The $t$-aspect bounds get worse as $[k:\Q]$ gets larger. This will reduce the power savings that can be achieved by \Cref{thm:tauberian_power_saving}, but will otherwise have no effect on the existence of an asymptotic with at least some nonzero power savings. With some effort, it is certainly possible to trace through the argument for arbitrary $k$ to give an expression for the error bound in terms of $[k:\Q]$.
\end{itemize}

\subsection{The nilpotent groups in \Cref{thm:intro_nilpotent_groups_with_shortcut}}

Each of the transitive groups listed in \Cref{thm:intro_nilpotent_groups_with_shortcut} share several common properties. For $G$ being each of these transitive groups,
\begin{enumerate}
    \item there exist abelian normal subgroups $T_1,T_2\normal G$ for which $G$ is semi-concentrated in $T_1,T_2$,
    \item $G$ has precisely three minimum index ramification types: one in $\iota_*\RTo{k}(T_1\cap T_2)$, one in $\iota_*\RTo{k}(T_1) - \iota_*\RTo{k}(T_2)$, and one in $\iota_*\RTo{k}(T_2) - \iota_*\RTo{k}(T_1)$.
\end{enumerate}
The first two properties are exactly what we need to apply the $2$-dimensional shortcut \Cref{prop:shortcut_convex_hull}. We check \eqref{eq:friendly} using \verb|Magma| \cite{MAGMA,codeurl} to confirm the asymptotic with power saving error term for each group, where we take $\alpha_\tau = 3/8$ for each $\tau$ corresponding to the unconditional subconvexity bounds for Hecke $L$-functions proven by Yang \cite[Corollary 1.2]{yang2023burgess}.

The bounds for the degree of the polynomial follow from \Cref{thm:general_unconditional_nilpotent}, where we need to determine $|\tw_{T,\widetilde{\pi}}^{-1}(\tau)|$ for each minimum index tame ramification type.

For the discriminant ordering, minimum index elements always have prime order. This can be checked directly, as $\ind_n(g^d) \le \ind(g)$ with equality if and only if $\langle g^d \rangle = \langle g\rangle$ due to how the cycle type of $g^d$ changes relative to $g$. Focusing solely on ramification types of prime order, we can prove the following:

\begin{lemma}\label{lem:prime_ord_tw_inj}
    Let $G$ be a nilpotent transitive group, $T\normal G$, and $\widetilde{\pi}:G_k\to G$ a homomorphism. If $\tau\in \RTo{k}(G)$ has prime order, then $|\tw_{T,\widetilde{\pi}}^{-1}(\tau)| = 1$.
\end{lemma}

In practice, this means that we do not expect any transitive nilpotent groups to be counter examples to Malle's conjecture for the value of $b$ in \Cref{conj:number_field_counting}. The possibility that $|\tw_{T,\widetilde{\pi}}^{-1}(\tau)|>1$ corresponds to $\tau$ splitting in to several ramification types of $T(\widetilde{\pi})$. This corresponds to the counter examples described by Kl\"uners in \cite{kluners2005} when this is due to ``shrinking the cyclotomic action''. When $G$ is nilpotent and we order by the product of ramified primes instead of the discriminant, it is still possible to have $|\tw_{T,\widetilde{\pi}}^{-1}(\tau)|>1$ corresponding to an ``entanglement of Frobenius'' as described by Koymans--Pagano in \cite{koymans-pagano2025}. The observation that this does not occur for nilpotent extensions in the discriminant ordering was previously communicated to the authors by Koymans in the context of their work, and \Cref{lem:prime_ord_tw_inj} is a statement of this fact in the language of ramification types.

In the context of \Cref{thm:intro_nilpotent_groups_with_shortcut}, this implies $b_\tau=1$ for each minimum index ramification type $\tau$. The upper bound for the degree of the polynomial is then given by
\begin{align*}
    \left(\sum_{\substack{\tau\in \iota_*\RTo{k}(T_1)\cup \iota_*\RTo{k}(T_2)\\ \ind_n(\tau)=a(G)}} b_\tau\right) - 1 &= \#\{\tau \in \iota_*\RTo{k}(T_1)\cup \iota_*\RTo{k}(T_2) : \ind_n(\tau) = a(G)\} -1\\
    &= 2,
\end{align*}
and the lower bound is given by
\begin{align*}
    \left(\sum_{\substack{\tau\in \iota_*\RTo{k}(T_j)\\ \ind_n(\tau)=a(G)}} b_\tau\right) - 1 &= \#\{\tau \in \iota_*\RTo{k}(T_j) : \ind_n(\tau) = a(G)\} -1\\
    &= 1.
\end{align*}

\begin{proof}[Proof of \Cref{lem:prime_ord_tw_inj}]
    Say that $\tau$ has prime order $p$. Write $S_p(G)$ for the Sylow $p$-subgroup of $G$. We remark that nilpotent groups are direct products of their Sylow subgroups, so the $G$-action on $\tau$ by conjugation necessarily factors through the quotient map $G\to S_p(G)$. The transitive $G\times G_k$-action on $\tau$ then factors through $S_p(G)\times \Gal(k(\zeta_p)/k)$.
    
    For any $\widetilde{\pi}:G_k\to G$, the action on $x:\tau\mapsto \widetilde{\pi}(x)\tau^{\chi^{-1}(x)}\widetilde{\pi}(x)^{-1}$ therefore factors through $S_p(G)\times \Gal(k(\zeta_p)/k)$. It suffices to prove that the image of this action in $S_p(G)\times \Gal(k(\zeta_p)/k)$ is the whole group, as each $\widetilde{\tau}\in \tw_{T,\widetilde{\pi}}^{-1}(\tau)$ is an orbit under this action, while $\tau$ is a full $S_p(G)\times \Gal(k(\zeta_p)/k)$ orbit.
    
    Given that $[k(\zeta_p):k]$ is coprime to $p$, the kernels of
    \[
    \begin{tikzcd}
        G_k \rar{\widetilde{\pi}} &G \rar[two heads] &S_p(G)
    \end{tikzcd}
    \]
    and
    \[
    \begin{tikzcd}
        G_k \rar{\chi^{-1}} &\widehat{\Z}^{\times} \rar[two heads] &\Gal(k(\zeta_p)/k)
    \end{tikzcd}
    \]
    have coprime index, which directly implies the composition
    \[
    \begin{tikzcd}
        G_k \rar{\widetilde{\pi}\times \chi^{-1}} &G\times \widehat{\Z}^{\times} \rar &S_p(G)\times \Gal(k(\zeta_p)/k)
    \end{tikzcd}
    \]
    is surjective. Thus $\tau$ contains a single orbit under this action, or equivalently $|\tw_{T,\widetilde{\pi}}^{-1}(\tau)|=1$.
\end{proof}

\subsection{Wreath products}

\Cref{thm:wreath_products} applies to transitive groups of the form $G = N\wr B$ where $B$ is transitive of degree $m$ and $N$ is any one of the nilpotent transitive groups in \Cref{thm:intro_nilpotent_groups_with_shortcut}. This is possible because checking the conditions of the two dimensional simplification \Cref{prop:shortcut_convex_hull_twisted} for the normal subgroup $N^m\normal G$ is essentially the same as checking the untwisted version \Cref{prop:shortcut_convex_hull} for $N$.

This is made precise by the following lemma:
\begin{lemma}\label{lem:wreath_shortcut}
    Let $G= N\wr B$ for transitive groups $N,B$ of degrees $n,m$ respectively.
    \begin{enumerate}[(a)]
        \item $a(N) = a(N^m)$
        \item Let $f:N\to N^m$ be the embedding in to the first coordinate, $g\mapsto (g,1,\dots,1)$. Then $f_*$ induces a bijection from
        \[
            \{\tau\in \RTo{k}(N) : \ind_n(\tau) = a(N)\} \rightarrow \{\tau\in \iota_*\RTo{k}(N^m) : \ind_n(\tau) = a(N^m)\}.
        \]
        \item $f_*$ preserves the order of ramification types
        \item For each $\tau,\kappa\in \RTo{k}(N)$,
        \[
            \ind_{|f_*\tau|/[k(\zeta_{f_*\tau}):k]}((q_\tau)_*f_*\kappa) = \ind_{|\tau|/[k(\zeta_{\tau}):k]}((q_\tau)_*\kappa)
        \]
    \end{enumerate}
\end{lemma}

\begin{proof}
    Let $n$ be the degree of $N$. The index of an element $(g_1,g_2,\dots,g_m)\in N^m$ is given by
    \[
        \ind_{nm}(g_1,g_2,\dots,g_m) = \ind_n(g_1) + \ind_n(g_2) + \cdots + \ind_n(g_m).
    \]
    Thus, the minimum index elements of $N^m$ are precisely those $(g_1,g_2,\dots,g_m)$ for which there exists an index $j$ such that $g_i=1$ for $i\ne j$ and $g_j$ is a minimum index element of $N$. As they both have index $a(N)$, this proves part (a).
    
    Given that $B$ acts transitively on the coordinates of $N^m$, every minimum index conjugacy class is conjugate to one of the form $(g,1,1\dots,1)$. This implies $f_*$ is surjective. Injectivity follows from injectivity of $f$, proving part (b).

    Part (c) is immediate, as $g$ and $(g,1,\dots,1)$ have the same order.

    Finally, part (d) follows by noting that $\zeta_\tau = \zeta_{f_*\tau}$ from (c) and that the diagram
    \[
        \begin{tikzcd}
            N \dar{f} \rar{q_\tau} &N/C_N(\tau)\dar{f}\\
            N^m \rar{q_\tau} &N^m/(C_G(f_*\tau)\cap N^m)
        \end{tikzcd}
    \]
    commutes. 
\end{proof}

As Yang's subconvexity results apply to any Hecke $L$-function, we can always take $\alpha_{f_*\tau} = \alpha_\tau = 3/8$. Combined with this lemma, equations \eqref{eq:friendly} and \eqref{eq:friendly_twisted} are equivalent. Given that we checked \eqref{eq:friendly} to prove \Cref{thm:general_unconditional_nilpotent} for $N$, it follows that \eqref{eq:friendly_twisted} holds for $N^m \normal G$. Thus, according to \Cref{thm:unconditional_concentrated} it now suffices to check \eqref{eq:sum_h1ur_uncond}.

\begin{lemma}\label{lem:wreath_h1ur}
    Let $k$ be a number field, $G = N\wr B$ for $N$ a transitive $2$-group, $B$ a transitive group of degree $m$, and $T\normal N$ with $T$ abelian. Then for each $\pi \in q_*\Sur(G_k,G)$,
    \[
        h_{ur}^1(k,N^m,T^m,\pi) \le |\Cl_{k(\pi)}[2]|^{\log_2|T|},
    \]
    where $k(\pi)$ is the field fixed by $\pi^{-1}(\Stab_B(1))$.
\end{lemma}

\begin{proof}
    This follows essentially from \cite[Corollary 1.14(iii)]{ALOWW}, which implies that for a field extension $F/k$ and an abelian group $M$ with trivial $G_F$ action,
    \[
        |H_{ur}^1(k,{\rm Ind}_F^k M)| \ll_{|M|,[F:k]} |\Hom(\Cl_F, M)|.
    \]
    For any abelian group $A$, the action of $B$ on ${\rm Ind}_{\Stab_B(1)}^B(A) \cong A^m$ is precisely the action permuting the coordinates.

    The upper central series of $N^m$ is precisely $Z_i(N)^m$. In particular, the action of $B$ on
    \[
        T^m\cap Z_i(N)^m / T^m\cap Z_{i-1}(N)^m = \left(T\cap Z_i(N)/ T\cap Z_{i-1}(N)\right)^m
    \]
    induced by conjugation is exactly the action permuting the coordinates, and so is module isomorphic to
    \[
        {\rm Ind}_{\Stab_B(1)}^B\left(T\cap Z_i(N)/ T\cap Z_{i-1}(N)\right).
    \]
    For each $\pi\in q_*\Sur(G_k,G)$, this implies
    \[
        \left(T\cap Z_i(N)/ T\cap Z_{i-1}(N)\right)^m(\pi) = {\rm Ind}_{k(\pi)}^k\left(T\cap Z_i(N)/ T\cap Z_{i-1}(N)\right).
    \]
    By combining this bound with \cite[Corollary 1.14(iii)]{ALOWW} and \Cref{def:h1ur}, we can bound
    \begin{align*}
        h_{ur}^1(k,N^m,T^m,\pi) &= \prod_{i=1}^\infty \left\lvert H_{ur}^1\left(k,{\rm Ind}_{k(\pi)}^k\left(T\cap Z_i(N)/ T\cap Z_{i-1}(N)\right)\right) \right\rvert\\
        &\ll_{|T|,m} \prod_{i=1}^\infty |\Hom\left(\Cl_{k(\pi)},T\cap Z_i(N)/ T\cap Z_{i-1}(N)\right)|.
    \end{align*}
    An abelian $2$-group $A$ certainly satisfies $|\Hom(\Cl_F,A)| \le |\Cl_F[2]|^{\log_2|A|}$ since every such homomorphism will factor through the $2$-part of $\Cl_F$. From this we conclude that
    \begin{align*}
        h_{ur}^1(k,N^m,T^m,\pi) &\ll_{|T|,m} \prod_{i=1}^\infty |\Cl_{k(\pi)}[2]|^{\log_2[T\cap Z_i(N):T\cap Z_{i-1}(N)]}\\
        &=|\Cl_{k(\pi)}[2]|^{\sum_i\log_2[T\cap Z_i(N):T\cap Z_{i-1}(N)]}\\
        &=|\Cl_{k(\pi)}[2]|^{\log_2|T|}.
    \end{align*}
\end{proof}

We are now ready to prove the result.

\begin{proof}[Proof of Theorem \Cref{thm:wreath_products}]
    We prove this via \Cref{thm:unconditional_concentrated}.
    
    As discussed above, \Cref{lem:wreath_shortcut} together with the same \verb|Magma| code used to check \eqref{eq:friendly} in the proof of \Cref{thm:general_unconditional_nilpotent} can be used with \Cref{prop:shortcut_convex_hull_twisted} to prove that $(\ind_n(\tau)/a(nTd))_\tau$ is contained in the convex hull ${\rm Hull}(\bigcup_j \Omega_{T_j})$.

    Thus, it suffices to show \eqref{eq:sum_h1ur_uncond} with a value $\theta < 1/a(N)$. \Cref{lem:wreath_h1ur} implies that
    \[
        \sum_{\pi\in q_*\Sur(G_k,G;X)} \max_j h_{ur}^1(k,N^m,T_j^m,\pi) \ll_{|N|,m} \sum_{\pi\in q_*\Sur(G_k,G;X)} \max_j |\Cl_{k(\pi)}[2]|^{\log_2|T_j|}.
    \]
    It is proven in \cite[(5.1)]{ALOWW} that
    \[
        q_*\Sur(G_k,G;X) \subseteq \{\psi\in \Sur(G_k,G) : |q_*\disc_G(\psi)|\le X\}.
    \]
    As $G$ is an imprimitive extension with tower type $(N,B)$, it then follows from \cite[Proposition 5.2]{ALOWW} that
    \[
        q_*\Sur(G_k,G;X) \subseteq \Sur(G_k,G/N,X^{1/n}),
    \]
    where $N$ is a transitive group in degree $n$. This set is in a many-to-one correspondence with $\mathcal{F}_k(B;X^{1/n})$ via the map $\pi \mapsto k(\pi)$. This implies
    \[
        \sum_{\pi\in q_*\Sur(G_k,G;X)} \max_j h_{ur}^1(k,N^m,T_j^m,\pi) \ll_{|N|,m} \sum_{F\in \mathcal{F}_{k}(B;X^{1/n})} \max_j |\Cl_{F}[2]|^{\log_2|T_j|}.
    \]
    The best known general bounds for $2$-torsion in the class group is given by \cite{bhargava-shankar-taniguchi+2017}, implying that
    \[
        |\Cl_{F}[2]| \ll_{[F:\Q],\epsilon} |\disc(F/\Q)|^{\frac{1}{2} - \frac{1}{2[F:\Q]}+\epsilon}.
    \]
    Inputting this bound, we find that
    \[
        \sum_{\pi\in q_*\Sur(G_k,G;X)} \max_j h_{ur}^1(k,N^m,T_j^m,\pi) \ll_{|N|,m,[k:\Q]} \max_j X^{\frac{\log_2|T_j|}{n}\left(\frac{1}{2} - \frac{1}{2m[k:\Q]}\right)}\#\mathcal{F}_{k}(B;X^{1/n}).
    \]
    By \Cref{thm:unconditional_concentrated}(i), it suffices to show that the right-hand side is $\ll X^{1/a(N) - \delta}$ for some positive $\delta$. This is equivalent to
    \[
        \mathcal{F}_{k}(B;X) \ll \min_j X^{\frac{n}{a(N)} - \frac{\log_2|T_j|}{2} + \frac{\log_2|T_j|}{2m[k:\Q]} - \delta}
    \]
    for some positive $\delta$.

    We compute this expression directly using \verb|Magma| to conclude the proof \cite{MAGMA,codeurl}. By way of example, we demonstrate the computation for $N=8T4$. In this case $n=8$, $a(8T4) = 4$, and following the work in \Cref{subsec:D4}, $T_1$ and $T_2$ are both groups of order $4$. Thus
    \begin{align*}
        \frac{n}{a(N)} - \frac{\log_2|T_j|}{2} + \frac{\log_2|T_j|}{2m[k:\Q]} - \delta &= \frac{8}{4} - \frac{2}{2} + \frac{2}{2m[k:\Q]} - \delta\\
        &= 1 + \frac{1}{m[k:\Q]}-\delta.
    \end{align*}
\end{proof}

\subsection{Direct products}

\Cref{thm:direct_products} applies to transitive groups of the form $G = N\times B$ where $B$ is transitive of degree $m$ and $N$ is any one of the nilpotent transitive groups in \Cref{thm:intro_nilpotent_groups_with_shortcut}. The proof is similar to that of \Cref{thm:wreath_products}, and in some sense can be considered easier because conjugation by $1\times B$ acts trivially on the normal subgroup $N\times 1 \normal G$.

\begin{proof}[Proof of \Cref{thm:direct_products}]
    We prove this via \Cref{thm:unconditional_concentrated}.

    Each $\pi \in q_*\Sur(G_k,G)$ has a lift $\widetilde{\pi}:G_k\to N\times B$ given by composing with the section $B \mapsto 1\times B$. As $1\times B$ commutes with $N\times 1$, it follows that $(N\times 1)(\widetilde{\pi})\cong N$ as $G_k$-groups, where we view $N$ as having the trivial $G_k$-action. This implies that the twisting map $\tw_{N,\widetilde{\pi}}:\RTo{k}(N) \to \RTo{k}(G)$ is injective. In particular, this induces a bijection between
    \[
        \{\tau \in \RTo{k}(N) : \ind_n(\tau) = a(N)\} \to \{\tau\in \iota_*\RTo{k}(N\times 1) : \ind_{nm}(\tau) = ma(N)\}
    \]
    which preserves the order of ramification types. (Note that $a(N\times 1) = ma(N)$). This implies $G$ is semiconcentrated in $T_1\times 1$, $T_2\times 1$, where $T_1,T_2\normal N$ are the two abelian groups in which $N$ is semiconcentrated.
    
    We also find that the diagram
    \[
        \begin{tikzcd}
            N \dar \rar{q_\tau} &N/C_N(\tau)\dar\\
            N\times 1 \rar{q_{(\tau,1)}} &N\times 1/(C_G(\tau,1)\cap (N\times 1))
        \end{tikzcd}
    \]
    commutes, as $1\times B$ commutes with $N\times 1$. This implies that for each $\tau,\kappa\in \RTo{k}(N)$,
    \[
        \ind_{|(\tau,1)|/[k(\zeta_{(\tau,1)}):k]}((q_{(\tau,1)})_*(\kappa,1)) = \ind_{|\tau|/[k(\zeta_{\tau}):k]}((q_\tau)_*\kappa).
    \]

    All together, these imply \eqref{eq:friendly} for $N$ and \eqref{eq:friendly_twisted} for $N\times 1$ are equivalent. As \eqref{eq:friendly} was already checked for the proof of \Cref{thm:intro_nilpotent_groups_with_shortcut}, it now follows from \Cref{prop:shortcut_convex_hull_twisted} that $(\ind_{nm}(\tau)/a(G))_\tau$ is in the interior of ${\rm Hull}(\Omega_{T_1}\cup\Omega_{T_2})$.

    Thus, it suffices to show \eqref{eq:sum_h1ur_uncond} with a value $\theta < \frac{1}{ma(N)}$. However, consider the group
    \[
        (T_j\times 1)\cap Z_i(N\times 1) / (T_j\times 1) \cap Z_{i-1}(N\times 1).
    \]
    The group $1\times B$ acts trivially on this group by conjugation, as it is the quotient of a subgroup of $N\times 1$. Moreover, $N\times 1$ acts trivially on this group by conjugation, as it is a subgroup of a factor of the upper central series. This implies $G$ acts trivially on this group by conjugation, so that the Galois module
    \[
        \left((T_j\times 1)\cap Z_i(N\times 1) / (T_j\times 1) \cap Z_{i-1}(N\times 1)\right)(\widetilde{\pi})
    \]
    carries the trivial action. If $M$ is any Galois module with the trivial action, we necessarily have
    \[
        H_{ur}^1(k,M) = \Hom(\Cl_k,M),
    \]
    the size of which depends only on $k$ and $|M|$. Therefore
    \[
        \sum_{\pi\in q_*\Sur(G_k,G;X)} \max_j h_{ur}^1(k,N\times 1,T_j\times 1,\pi) \ll_{k,|N|} \#q_*\Sur(G_k,G;X).
    \]
    It is proven in \cite[(5.1)]{ALOWW} that
    \[
        q_*\Sur(G_k,G;X) \subseteq \{\psi\in \Sur(G_k,G) : |q_*\disc_G(\psi)|\le X\}.
    \]
    As $G$ is an imprimitive extension with tower type $(N,B)$, it then follows from \cite[Proposition 5.2]{ALOWW} that
    \[
        q_*\Sur(G_k,G;X) \subseteq \Sur(G_k,G/N,X^{1/n}),
    \]
    where $N$ is a transitive group in degree $n$. This set is in a many-to-one correspondence with $\mathcal{F}_k(B;X^{1/n})$ via the map $\pi \mapsto k(\pi)$. This implies
    \[
        \sum_{\pi\in q_*\Sur(G_k,G;X)} \max_j h_{ur}^1(k,N^m,T_j^m,\pi) \ll_{k,|N|}\#\mathcal{F}_{k}(B;X^{1/n}).
    \]
    By \Cref{thm:unconditional_concentrated}(i), it suffices to show that the right-hand side is $\ll X^{\frac{1}{ma(N)} - \delta}$ for some positive $\delta$. This is equivalent to
    \[
        \mathcal{F}_{k}(B;X) \ll X^{\frac{n}{ma(N)} - \delta}
    \]
    for some positive $\delta$, concluding the proof.
\end{proof}

\section{Conditional Results}\label{sec:conditional}

It is natural to ask is how far this method can take us in ideal conditions. The most challenging part of applying our main unconditional results \Cref{thm:general_unconditional_nilpotent} and \Cref{thm:unconditional_concentrated} is confirming that the relevant convex hull ${\rm Hull}(\bigcup \Omega_j)$ contains the singularity of interest. Luckily, we do not need the full convex hull in order to prove our main results as the singularity of interest is always on the boundary of the orthant ${\rm Re}(s_\tau) \ge 1$. Thus, it would suffice to show that the convex hull contains an open neighborhood of this closed orthant.

This is still a significant challenge for proving unconditional results, and is currently not possible for many groups. Fortunately, if we assume a sufficient form of the Lindel\"of hypothesis all of these convex hull computations can be boiled down a single result, \Cref{lem:cond_convex_hull}. Throughout this section we will assume some version of the Lindel\"of hypothesis in a half-plane of the form ${\rm Re}(s) > \gamma$ for some $\gamma < 1$; namely we will assume either $H(\gamma,0,*; k, G)$ or $H(\gamma,0,\beta; k, G)$ for some $\beta \ge 0$.

\subsection{A convex hull computation}
Conditional on $H(\gamma,0,*;k,G)$, the convex hull computation can now be boiled down to the following lemma.

\begin{lemma}\label{lem:cond_convex_hull}
    Let $0 < \gamma < 1$ and for each $j = 1,2,...,n$ define
    \[
        U_j = \{ \underline{x}\in \R^n : x_j > \gamma \text{ and } x_k > 1 \text{ for each }k\ne j\}.
    \]
    Then the convex hull of $\bigcup_j U_j$ contains the set
    \[
        U=\left\{\underline{x}\in \R^n : x_j > 1 - \frac{1}{n}(1-\gamma)\text{ for each }j\right\}.
    \]
\end{lemma}

\begin{proof}
    Let $\underline{x}\in U$. For each $j=1,2,...,n$, define $\underline{y_j} = (y_{j,k})_{k=1}^n\in \R^n$ by
    \[
        y_{j,k} = \begin{cases}
            x_j - \left(1 + \frac{1}{n}\right)(1-\gamma) & j = k\\
            x_k + \frac{1}{n}(1-\gamma) & j \ne k.
        \end{cases}
    \]
    One confirms that
    \begin{align*}
        y_{j,k} &> \begin{cases}
            1-\frac{1}{n}(1-\gamma) - \left(1 + \frac{1}{n}\right)(1-\gamma) & j = k\\
            1-\frac{1}{n}(1-\gamma) + \frac{1}{n}(1-\gamma) & j \ne k.
        \end{cases}\\
        &=\begin{cases}
            \gamma & j = k\\
            1 & j \ne k,
        \end{cases}
    \end{align*}
    so that $\underline{y_j}\in U_j$

    We claim that $\underline{x} = \frac{1}{n}\sum \underline{y_j}$, which would prove that $\underline{x}$ is contained in the convex hull and conclude the proof. Indeed, for each $k=1,2,...,n$ we find that
    \begin{align*}
        \frac{1}{n}\sum_{j=1}^n y_{j,k} &= \frac{1}{n}\left( x_k - \left(1 + \frac{1}{n}\right)(1-\gamma) + \sum_{j\ne k} \left(x_k + \frac{1}{n}(1-\gamma)\right)\right)\\
        &= x_k + \frac{1}{n}\left( -\left(1-\frac{1}{n}\right)(1-\gamma) + (n-1)\frac{1}{n}(1-\gamma)\right)\\
        &= x_k.
    \end{align*}
\end{proof}

\subsection{Proofs of our main conditional results}

In the context of \Cref{thm:general_unconditional_nilpotent} and \Cref{thm:unconditional_concentrated}, this is equivalent to assuming that we can take $\alpha_\tau = 0$ for each $\tau$, or equivalently $M_{\tau,\kappa} = 0$. This assumption greatly simplifies the equations cutting out $\Omega_T$, reducing them to just
\begin{align*}
    \sigma_\tau &> \gamma && \text{if }\tau\in \iota_*\RTo{k}(T)\\
    \sigma_\tau &> 1 && \text{if } \tau \in \RTo{k}(G) - \iota_*\RTo{k}(T).
\end{align*}

\begin{proof}[Proofs of \Cref{thm:intro_nilpotent},\Cref{thm:main_concentrated},\Cref{thm:nilpotent_fibers_main},\Cref{thm:nilpotent_fibers_uniform}]
Each of these results follows from the unconditional versions \Cref{thm:general_unconditional_nilpotent},\Cref{thm:unconditional_concentrated},\Cref{thm:unconditional_nilpotent_fibers}, and \Cref{thm:unconditional_nilpotent_fibers_uniform}, as long as we can prove that the point $(\ind_n(\tau)/a(nTd))_\tau$ lies in the convex hull of ${\rm Hull}(\bigcup\Omega_{T_j})$.

Define $U_\tau$ to be the region cut out by
\begin{align*}
    \sigma_\tau &> \gamma\\
    \sigma_\kappa &> 1 && \text{if }\kappa\in \RTo{k}(G)-\{\tau\}.
\end{align*}
For each $\tau \in \bigcup_j\iota_*\RTo{k}(T_j)$, we certainly have $U_\tau\subseteq \bigcup \Omega_{T_j}$ (assuming $H(\gamma,0,*;k,G)$).

Let $U$ be the region cut out by
\begin{align*}
    \sigma_\tau &> 1 - \frac{1}{|\bigcup_j\iota_*\RTo{k}(T_j)|}(1-\gamma) &&\text{if }\tau \in \bigcup_j\iota_*\RTo{k}(T_j)\\
    \sigma_{\tau} &> 1 && \text{if }\kappa\in \RTo{k}(G)-\bigcup_j\iota_*\RTo{k}(T_j).
\end{align*}
\Cref{lem:cond_convex_hull} implies that $U$ is contained in the convex hull ${\rm Hull}(\bigcup_{\tau} U_\tau)$, where the union is over $\tau\in \bigcup_j\iota_*\RTo{k}(T_j)$, and therefore $U\subseteq {\rm Hull}(\bigcup_j\Omega_{T_j})$.

For each tame $\tau$, we find that
\[
    \frac{\ind_n(\tau)}{a(nTd)} \ge 1
\]
by the definition of $a(nTd)$ as the minimum tame index. Moreover, the assumption that $G$ is semi-concentrated in $T_1,T_2,\dots,T_r$ means that all the minimum index ramification types belong to $\bigcup_j\iota_*\RTo{k}(T_j)$. In particular, for any $\tau \in \RTo{k}(G) - \bigcup_j\iota_*\RTo{k}(T_j)$ we necessarily have
\[
    \frac{\ind_n(\tau)}{a(nTd)} \ge \frac{a(nTd)+1}{a(nTd)} > 1.
\]
This implies $(\ind_n(\tau)/a(nTd))_\tau \in U\subseteq {\rm Hull}(\bigcup_j \Omega_{T_j})$, concluding the proof.
\end{proof}

\subsection{Nilpotency class $2$ groups ordered by an arbitrary invariant}

The proof of \Cref{thm:nilpotent_general_invariant} is completely analogous to that of \Cref{thm:intro_nilpotent}. We will show how this follows from the region of meromorphicity we proved for \Cref{thm:general_unconditional_nilpotent}.

\begin{proof}[Proof of \Cref{thm:nilpotent_general_invariant}]
As before, this result follows from \Cref{thm:general_unconditional_nilpotent} as long as we can prove that the point $({\rm wt}(\tau)/a_{\inv}(G))_\tau$ is contained in the convex hull ${\rm Hull}(\bigcup_j\Omega_{T_j})$.

Let $T_1,T_2,\dots,T_r$ be the abelian normal subgroups of $G$. By the same argument as in the proof of \Cref{thm:intro_nilpotent}, we conclude that the tubular region $U$ cut out by
\[
    \begin{cases}
        \sigma_{\tau} > 1 - \frac{1}{|\bigcup \iota_*\RTo{k}(T_j)|}(1-\gamma) & \tau \in \bigcup \iota_*\RTo{k}(T_j)\\
        \sigma_\tau > 1 & \tau \in \RTo{k}(G) - \bigcup \iota_*\RTo{k}(T_j)
    \end{cases}
\]
is contained in the convex hull ${\rm Hull}(\bigcup_j \Omega_{T_j})$.

When $G$ has nilpotency class $2$ and $T_1,T_2,\dots, T_r$ are all the abelian normal subgroups of $G$, it necessarily holds that $G = \bigcup_j T_j$. Thus, we can in fact say that $U$ is cut out solely by the equations in the first row and is in particular an open neighborhood of the orthant cut out by $\sigma_\tau \ge 1$ for each $\tau\in \RTo{k}(G)$. By definition of $a_{\inv}(G)$ as the minimum tame weight, it follows that
\[
    \frac{{\rm wt}(\tau)}{a_{\inv}(G)} \ge 1,
\]
so that no matter what weight is chosen we have shown that $({\rm wt}(\tau)/a_{\inv}(G))_{\tau}\in U$. This concludes the proof.
\end{proof}

\subsection{Further applications to semiconcentrated groups}

Our main results are applicable to groups which are semiconcentrated in abelian normal subgroups. If we assume a number of existing conjectures, then \Cref{conj:number_field_counting} will follow for all of these groups.

\begin{corollary}
    Let $k$ be a number field and $G$ be a transitive group with label $nTd$ which is semiconcentrated in abelian normal subgroups and for which there exists at least one $G$-extension of $k$.

    We assume the following.
    \begin{enumerate}[(a)]
        \item There exists $\gamma < 1$ for which $H(\gamma,0,*; k, G)$ holds.
        \item The $\ell$-torsion conjecture for the size of torsion in class groups $|\Cl_K[\ell]| \ll_{[K:\Q],\ell,\epsilon} |\disc(K/\Q)|^{\epsilon}$, and
        \item Malle's predicted upper bound for $G/N$-extensions, where $N$ is the Fitting subgroup of $G$, with pushforward discriminant bounded above by $X$, which is the bound $\ll X^{1/a(G-N) + \epsilon}$.
    \end{enumerate}
    Then there exist ineffective constants $c\ge 0$ and $b\in \Z_{>0}$ for which 
    \[
        \#\mathcal{F}_{\disc,k}(G;X) \sim c X^{1/a(G)}(\log X)^{b-1}.
    \]
\end{corollary}

The proof is immediate, as the $\ell$-torsion conjecture together with \cite[Lemma 4.1]{ALOWW} implies $h^1_{ur}(k,N,T,\pi)\ll |\disc(\pi)|^{\epsilon}$. Then Malle's predicted upper bound shows that we can take $\theta = 1/a(G-N)$ in \Cref{thm:main_concentrated}, where $N$ is the Fitting subgroup of $G$. The fact that $G$ is semiconcentrated in $T_1,T_2,\dots T_r\subseteq N$ implies that $1/a(G-N) < 1/a(N)$, so the conclusion of \Cref{thm:main_concentrated} gives the desired asymptotic.

This is similar to the content of \cite[Section 8]{ALOWW}, where they note that their results are applicable to groups which are either of the form $S_3\wr B$ or are concentrated in an abelian normal subgroup. The $\ell$-torsion conjecture and Malle's predicted upper bound together are enough to conclude \Cref{conj:number_field_counting} for each of these groups using the main results of \cite{ALOWW}.

Also similar to \cite{ALOWW}, we argue that our framework is, in principle, applicable to any semiconcentrated group. By constructing meromorphic continuations for $D_k(T_j,\pi;\underline{s})$ for new cases of normal subgroups $T_j\normal G$, essentially a complex analytic version of proving new cases of \Cref{conj:twisted_number_field_counting}, with sufficiently small bounds in vertical strips in terms of $\pi$, the method of this paper gives a roadmap for proving \Cref{conj:number_field_counting} for $G$ which is semiconcentrated in such normal subgroups $T_1,T_2,\dots, T_r$.

\section{Application to the Existence of Potential Secondary Terms}

Finally, we want to point out another application of this point of view. We will (again) use the group $G=D_4$ to illustrate how one can use the multiple Dirichlet series point of view to uncover secondary terms in the field counting asymptotics. We keep the discussion in this section {\em unconditional}. 

As is the case with the multiple Dirichlet series prototype \cite{DGH}, our method could produce exact formulas, secondary terms, and power saving error terms if one is able to compute the residues along the various polar divisors. Then one could not only add the contributions at the points where they intersect and find the order of the poles, but get the actual polynomial in $\log X$ coming from the rightmost pole (after specializing to the desired multi-invariant slice), as well as the secondary terms coming from the other poles in the region of meromorphic continuation. If on top of that one has good enough control on the $t$-aspect and class group torsion bounds in a given context, one could get a power saving error term by moving the line of integration to within $\epsilon$ of the boundary.

To illustrate this phenomenon, let us consider the following family of invariants for $D_4$-fields: Let $\gamma > 0$ be a real number. If $F$ is a quartic $D_4$-field with quadratic subfield $k\subseteq F$, define
\[
    \inv_\gamma(F) := \frac{\disc(F/\Q)}{\disc(k/\Q)^{1-\gamma}} = {\rm Nm}(\disc(F/k)) \cdot \disc(k/\Q)^{1+\gamma}.
\]
This is part way between the conductor and quartic discriminant orderings. The weights of each tame ramification type are as follows:
\begin{align*}
    {\rm wt}(2A) &= 2; & {\rm wt}(2B) &= 1+\gamma; & {\rm wt}(2C) &= 1; & {\rm wt}(4A) &= 2+\gamma.
\end{align*}
We then specialize to the complex line $\underline{s} = ({\rm wt}(\tau)s)_{\tau}$, whose projection onto the $(\sigma_{2B},\sigma_{2C})$-plane is the line $(\sigma_{2B}, \sigma_{2C}) = ((1+\gamma)\sigma, \sigma).$ The special case $\gamma = 1/4$ is depicted in \Cref{fig:secondary_term_example}.

\begin{figure}[!ht]
\begin{center}
    \hspace*{1cm} 
    \begin{tikzpicture}[scale=1.5]

    \draw[gray!50, thin, step=1/3] (-1,-1) grid (3,3);
    \draw[very thick,->] (-1,0) -- (3.2,0) node[right] {$\sigma_{2B}$};
    \draw[very thick,->] (0,-1) -- (0,3.2) node[above] {$\sigma_{2C}$};

    \foreach \x in {-1,...,3} \draw (\x,0.05) -- (\x,-0.05) node[below] {\tiny\x};
    \foreach \y in {-1,...,3} \draw (-0.05,\y) -- (0.05,\y) node[right] {\tiny\y};

    \fill[blue!50!cyan,opacity=0.6] (1,3) -- (1,1) -- (3,1) -- (3,3);
    \fill[blue!50!cyan,opacity=0.3] (1/2,3) -- (1/2,7/6) -- (7/6,1/2) -- (3,1/2) -- (3,3);

    \draw[red,very thick,-] (-1,-4/5) -- (3,12/5) node[right] {\inv};

    \draw[black,dashed,thick,-] (1,-1) -- (1,3) node[above] {$s_{2B}=1$};
    \draw[black,dashed,thick,-] (-1,1) -- (3,1) node[below right] {$s_{2C}=1$};
    \end{tikzpicture}
    \caption{\label{fig:secondary_term_example} Continuation in $T_C$-direction}
\end{center}
\end{figure}
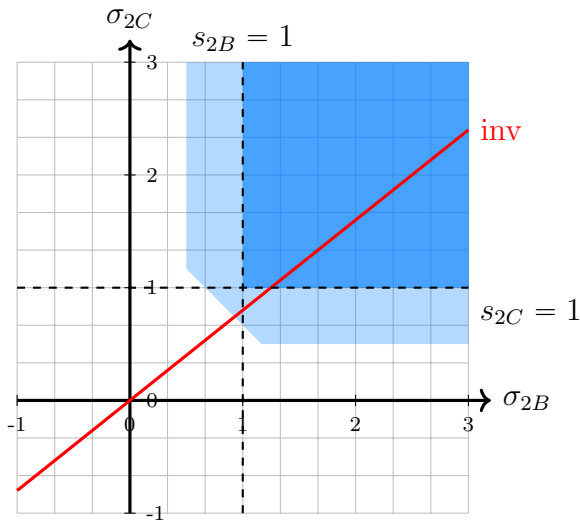

One can see that this complex line appears to cross both the polar divisors $s_{2B}=1$ and $s_{2C}=1$ within the region of meromorphicity, suggesting that we might reveal a secondary term.

We confirm this algebraically for $\gamma$ sufficiently small by specializing $\Omega$ along the complex line, although it turns out the $\gamma = 1/4$ is not small enough. The single variable generating series for this invariant has a meromorphic continuation to
\[
    \{s : 2\sigma > 1,\ (2+\gamma)\sigma > 19/16,\ (1+\gamma)\sigma + \sigma > 27/16\} = \left\{s : \sigma > \max\left\{\frac{27}{32+16\gamma},\frac{1}{2}\right\}\right\},
\]
with possible polar divisors at $s_{2B} = 1/(1+\gamma)$ and $s_{2C} = 1$ of order $\le 1$. On this region, the generating series has vertical bounds given by
\[
    \ll (1+|t|)^{\frac{1}{3}\max\{1 - 2\sigma,0\} + \frac{3}{8} + \epsilon} = (1+|t|)^{\frac{3}{8} + \epsilon}.
\]
We can now apply \Cref{thm:tauberian_power_saving}, where restricting to $\gamma \le 11/8$ we may take $\sigma_a = 1$, $\delta= \frac{5+16\gamma}{32+16\gamma}$, and $\xi = 3/8$. The power saving exponent is then
\[
    1 - \frac{(5+16\gamma)/(32+16\gamma)}{3/8 + 1} + \epsilon = \frac{39+6\gamma}{44+22\gamma} + \epsilon.
\]
Whenever the power savings is smaller than $1/(1+\gamma)$, this \emph{might} reveal a secondary term of order $X^{\frac{1}{1+\gamma}}$, as long as the polar divisor does not get canceled out at the particular point. In the limit as $\gamma\to 0$, we know the polar divisors $s_{2B}=1$ and $s_{2C}=1$ cannot vanish, as the results of \cite{altug-shankar-varma-wilson2021} imply that they must multiply together to give a double pole at $(1,1)$.

We conclude the following result:
\begin{theorem}
    Let $0 < \gamma < \frac{1}{12}(\sqrt{649} - 23) \approx 0.2062$. Then there exist constants $c(\gamma),c'(\gamma)\ge 0$ for which
    \[
        \#\mathcal{F}_{\inv_\gamma,\Q}(D_4;X) = c(\gamma)X + c'(\gamma)X^{\frac{1}{1+\gamma}} + O_{\gamma,\epsilon}\left(X^{\frac{39+6\gamma}{44+22\gamma} + \epsilon}\right),
    \]
    where the error term is smaller than the (potential) secondary term.
    
    Moreover, there exists a positive constant $\gamma_{st} > 0$ such that if we further assume $0 < \gamma < \gamma_{st}$, then $c(\gamma), c'(\gamma) \ne 0$.
\end{theorem}

Unfortunately, one cannot be specific about the values of $c(\gamma)$, $c'(\gamma)$, or $\gamma_{st}$ without a fuller understanding of the polar divisors of $D_{\Q}(D_4;\underline{s})$ and their residues.

A similar phenomenon occurs for every group we consider: for certain invariants the power saving error bounds may reveal potential secondary terms of the form $X^{1/{\rm wt}(\tau)}P(\log X)$ for each tame ramification type $\tau$, coming from the intersection of the complex line $({\rm wt}(\tau)s)_\tau$ with the polar divisor $s_\tau = 1$. These are similar to the potential secondary terms appearing in the main results of \cite{alberts2024} for abelian groups.

More specifically, the region of meromorphicity constructed in \Cref{thm:general_unconditional_nilpotent}(i) implies the following: for $G$ a nilpotent group and $T_1,T_2,\dots, T_r$ abelian normal subgroups, there exists a $\theta > 0$ such that
\[
    \#\mathcal{F}_{\inv,k}(G;X) = \sum_{\tau \in \bigcup_j\iota_*\RTo{k}(T_j)} X^{\frac{1}{{\rm wt}(\tau)}} P_{\inv,k,\tau}(\log X) + O(X^{\theta})
\]
for certain polynomials $P_{\inv,k,\tau}$. There are a few road blocks to being more specific towards a result of this shape:
\begin{itemize}
    \item The value of $\theta$ is determined by $\inv$, the region of meromorphicity, and the best available $t$-aspect subconvexity bounds. It is possible that the $\theta$ produced in this way is too large, so that the error term absorbs any lower order asymptotic terms (as is the case in \Cref{thm:D4} and \Cref{thm:Q8sxC2}).

    \item Our method does not produce an explicit value for the polynomials $P_{\inv,k,\tau}$, and in particular we can not currently eliminate the possibility that $P_{\inv,k,\tau} = 0$ when $\tau$ does not have minimum index, causing the ``potential secondary term'' to vanish.
\end{itemize}
Both of these issues occur for abelian groups in \cite{alberts2024}, albeit to a lesser extent. While some tools are developed in \cite{alberts2024} to show that the secondary terms do not vanish for certain abelian groups, we do not develop any such tools for nilpotent groups. It would be interesting to work towards developing such tools, and possibly produce new secondary terms for number field counting results.

It is reasonable to expect that further meromorphic continuations exist than those we produce in this paper. In the particular case of quartic $D_4$-extensions, Cohen, Diaz y Diaz, and Olivier predict a secondary term of size $X^{1/2}(\log X)^{2}$ in \cite[Section 7]{cohen-diaz-y-diaz-olivier2006}. Thus, we expect to encounter more polar divisors intersecting the point $(s_{2B},s_{2C})=(1/2,1)$ in \Cref{fig:secondary_term_example}. We expect this include a polar divisor at $s_{2A} = 1$, but diagonal polar divisors like $s_{2A}+s_{2B} + s_{2C} + 2s_{4A} = 3/2$ that come from using functional equations are also expected.

\appendix
\section{Equivalence of Definitions for Tame Ramification types}\label{app:tame_ramification_types}

Let $k$ be a number field and $G$ be a group with a $G_k$-action. Recall \Cref{def:tameRT}, where we defined a tame ramification types of $G$, $\tau \in \RT_k^{\rm tame}(G)$, to be a $G\rtimes G_k$ orbits of $\Hom((\Q/\Z)^*,G)$, where $G_k$ acts on $(\Q/\Z)^* = \Hom(\Q/Z,\mu)$ and $G$, and $G$ acts by conjugation.

The notion of (tame) ramification types has played an important role in number field counting, stretching back to Malle's original predictions \cite{malle2002,malle2004}. Essentially, tame ramification types describe the ``ways in which a prime can tamely ramify in a $G$-extension'' without choosing a specific prime.

There are a number of different notions of (tame) ramification type in the literature, tailored to the specific contexts in which they were used.
\begin{enumerate}[(1)]
    \item When $G$ carries the trivial $G_k$-action, Malle considered the \emph{$k$-conjugacy classes} of $G$ in \cite{malle2004}, i.e. minimal subsets of $G$ closed under conjugation and under the cyclotomic action $x.g = g^{\chi(x)}$ for $\chi:G_k\to \hat{\Z}^{\times}$ the cyclotomic character.

    \item When $G$ carries the trivial $G_k$-action, Gundlach \cite{gundlach2022multimalle} describes tame ramification types (which he refers to simply as ramification types) as $G\times G_k$-equivalence classes of pairs $(I,\gamma)$ for a cyclic subgroup $I\le G$ and an isomorphism $\gamma:I\overset{\sim}{\to} \mu_e$. $G$ acts by conjugation and $G_k$ acts on $\mu_e$, and so on $\gamma$, by the cyclotomic action.

    \item In the case that $G$ is a $G_k$-module (i.e.~the underlying group is abelian), Smith \cite[Equation (3.1)]{smith2022selmer1} describes a ``ramification-measuring homomorphism'' valued in $G(-1)$, where $G(-1) = \Hom(\hat{\Z}(1),G)$ and 
    \[
        \widehat{\Z}(1) = \displaystyle \lim_{\leftarrow } \mu_e.
    \]
    This is generalized by Loughran--Santens \cite{loughran2024malle,tavernier2025counting} to any group scheme $G$. We remark that Loughran--Santens define the ramification type to be the image under the map from $Z(k_\nu,G)$ to the conjugacy classes of $G(-1)$ in \cite[Section 7.1]{loughran2024malle}, which is a finer version of the ramification map in \Cref{def:RTmap}. However, it is important that we assign complex variables to each $G_k$-orbit of conjugacy classes in $G(-1)$, rather than to individual conjugacy classes, to ensure that each singularity of $D_k(G;\underline{s})$ is a pole.

    \item \Cref{def:RT} is a direct extension of the notion of tame ramification type used by the first author in \cite{alberts2024}: In the case that $G$ is a $G_k$-module (i.e.~the underlying group is abelian), the tame ramification types are given by the $G_k$-orbits of $\Hom((\Q/\Z)^*,G)$.

    \item The notion of twisted sectors for a stack $\mathcal{X}$ developed by Darda--Yasuda \cite{darda2022batyrev} specializes to a notion of ramification type on the stack $BG$. Darda--Yasuda prove this is equivalent to Malle's notion of $k$-conjugacy classes in \cite[Example 2.15]{darda2022batyrev}, where their residue map plays the role of the ramification map in \Cref{def:RTmap}.
\end{enumerate}

These five descriptions are equivalent, i.e.~there are bijections between each of them. The purpose of this appendix is to prove that each of these are equivalent to the notion given by \Cref{def:tameRT}, save for the equivalence of (1) and (5) which is proven in \cite[Example 2.15]{darda2022batyrev}.

\begin{remark}
    As will be seen in the proofs, these bijections are not always \emph{canonical}. In some cases they will rely on choosing a generator for tame inertia $I_\pk^{\rm tame}$.

    Smith gave a way to make a consistent choice of generator for each $I_p^{\rm tame}$ all at once in \cite[Section 2]{smith2022selmer1}, depending on a choice of generator $\overline{\zeta}\in \hat{\Z}(1)$. He calls these generators ${\rm Tine}_k(\pk)\in I_\pk^{\rm tame}$. Each choice of generator $\overline{\zeta}\in \hat{\Z}(1)$ corresponds to a $G_{k_\pk}$-equivariant isomorphism $\hat{\Z}(1) \to I_\pk^{\rm tame}$ sending $\overline{\zeta} \mapsto {\rm Tine}_k(\pk)$.
\end{remark}

When keeping track of whether a correspondence is canonical or not, it will be useful to know the following isomorphism:
\begin{lemma}\label{lem:Q/ZtoZ(1)}
    The map $\phi:(\Q/\Z)^*\to \hat{\Z}(1)$ sending
    \[
        a\mapsto ( a(1/e)^{\frac{n}{(n,e)}})_{e=1}^{\infty}
    \]
    is a $G_k$-equivariant isomorphism, where $n$ is the order of $a\in (\Q/\Z)^*$.
\end{lemma}

\begin{proof}
    Basic properties of limits imply that
    \[
        (\Q/\Z)^* = \lim_{\leftarrow} \Hom((1/e)\Z/\Z,\mu).
    \]
    The map $\phi_e:\Hom((1/e)\Z/\Z,\mu) \to \mu_e$ sending
    \[
        a \mapsto a(1/e)
    \]
    is certainly a $G_k$-equivariant isomorphism for each $e\ge 1$, so it induces a $G_k$-equivariant isomorphism $(\Q/\Z)^*\to \hat{\Z}(1)$.

    We identify the image of the isomorphism by considering each of the compositions
    \[
        (\Q/\Z)^*\to \hat{\Z}(1) \to \mu_e.
    \]
    If $a\in (\Q/\Z)^*$ has order $n$, then this composition factors through $\Hom(\Z/n\Z) \to \mu_n \to \mu_e$, where the latter map is precisely $\zeta \mapsto \zeta^{n/(n,e)}$ given by the inverse system defining $\hat{\Z}(1)$.
\end{proof}

\subsection{$k$-conjugacy classes}

\begin{proposition}
    Let $G$ be a finite group with trivial Galois action. Let $a\in (\Q/\Z)^*$ be a generator. Then the map of pointed sets
    \[
        \Hom((\Q/\Z)^*,G) \rightarrow G
    \]
    given by
    \[
        f \mapsto f(a)
    \]
    is a $G\times G_k$-equivariant bijection, where $G$ acts on itself by conjugation, $G_k$ acts on $\mu$, and so on $(\Q/\Z)^*$, by the cyclotomic action $\sigma.a = a^{\chi(\sigma)}$, and $G_k$ acts on $G$ by the inverse cyclotomic acton $\sigma.g = g^{\chi(\sigma)^{-1}}$, where $\chi:G_K\to \hat{\Z}^{\times}$ is the cyclotomic character.

    In particular,
    \begin{enumerate}[(i)]
        \item This map descends to a bijection between $\RT_k^{\rm tame}(G)$ and the set of $k$-conjugacy classes in $G$. If $k$ is linearly disjoint from maximal abelian extension $\Q(\mu)$, this bijection is independent of the choice of generator $a\in (\Q/\Z)^*$.
        \item Let $\pi\in \Sur(G_k,G)$. For prime $p$ which is at most tamely ramified in $\pi$, this map sends $[\iota_p^*\pi|_{I_p}]$ to a $k$-conjugacy class containing a generator of the tame inertia group $\pi(I_p)\subseteq G$.
    \end{enumerate}
\end{proposition}

\begin{proof}
    The fact that $(\Q/\Z)^*$ is procyclic implies this map is a bijection. The $G$-action is given by $g.f(a) = g f(a)g^{-1}$ by definition, so this map is $G$-equivariant. $G_k$ acts on $\mu$, and so on $(\Q/\Z)^*$, by the cyclotomic character so that $\sigma.f(a) = f(\sigma^{-1}.a) = f(a^{\chi(\sigma)^{-1}}) = f(a)^{\chi(\sigma)^{-1}}$.

    For part (i), equivariance of the bijection implies that it factors through the orbits, which are precisely $\RT_k^{\rm tame}(G)$ and the set of $k$-conjugacy classes respectively. Suppose $k$ is linearly disjoint from $\Q(\mu)$ and $a'$ is a different generator for $(\Q/\Z)^*$. The automorphism group of $(\Q/\Z)^*$ is exactly $\hat{\Z}^{\times}$ because $(\Q/\Z)^*$ is (noncanonically) isomorphic to $\hat{\Z}$, so $a' = a^n$ for some $n\in \hat{\Z}^\times$. Thus, $f(a') = f(a)^n$ generates the same cyclic subgroup as $f(a)$, and so belongs to the same $\Q$-conjugacy class. As $k$ and $\Q(\mu)$ are linearly disjoint, the cyclotomic character $\chi:G_\Q\to \hat{\Z}^{\times}$ satisfies $\chi(G_\Q) = \chi(G_k)$, implying that the $k$-conjugacy classes are the same as the $\Q$-conjugacy classes.
    
    Part (ii) follows directly from the properties of $\iota_\pk^*$ proven in \Cref{lem:canonical}.
\end{proof}

\subsection{Gundlach's notion of ramification type}

We now elaborate on Gundlach's definitions: Gundlach defines (tame) ramification types as $G\times G_k$-orbits $[I,\gamma]$ of pairs $(I,\gamma)$ with $I\le G$ cyclic and $\gamma:I\overset{\sim}{\to} \mu_e$ with respect to the following $G\times G_k$-action:
\begin{itemize}
    \item If $g\in G$, then $g.(I,\gamma) = (g I g^{-1}, g.\gamma)$, where $(g.\gamma)(x) = \gamma(g^{-1}xg)$.
    \item If $\sigma\in G_k$, then $\sigma.(I,\gamma) = (I,\sigma \circ \gamma)$.
\end{itemize}

Let $(L,\psi)$ be a Galois extension $L/k$ with isomorphism $\psi:\Gal(L/k) \overset{\sim}{\to} G$. A tame prime $\pk$ is said by Gundlach to have ramification type $[I,\gamma]$ in $(L,\psi)$ if the following hold:
\begin{itemize}
    \item $\pk$ is ramified in $L/K$ with ramification index $e$ and $\psi(I_\pk(L/k)) = I$,
    \item For some $\mathfrak{q}\mid \pk$, $\beta_{\mathfrak{q}\mid \pk}:I\to (\mathcal{O}_L/\mathfrak{q})^{\times}$ is the natural injection defined by $g\mapsto g(\pi_{\mathfrak{q}})/\pi_{\mathfrak{q}} \mod \mathfrak{q}$, and
    \item $\gamma:I \to \mu_e$ is defined so that $\gamma(g)\in \mu_e$ is the unique root of unity for which $\gamma(g) \equiv \beta_{\mathfrak{q}\mid \pk}(g) \mod \mathfrak{q}$ for some prime $\mathfrak{q}\mid \pk$.
\end{itemize}
The orbit $[I,\gamma]$ is independent of both $\mathfrak{q}$ and $\beta$.

\begin{remark}
    None of the $\beta$s in this appendix have anything to do with the $\beta$s in the rest of the paper, in particular Section \ref{sec:unconditional}.
\end{remark}

\begin{proposition}
    Let $G$ be a finite group with trivial Galois action. There is a canonical $G\times G_k$-equivariant bijection of pointed sets
    \[
        \{ (I,\gamma) : I \le G\text{ is cyclic and }\gamma: I \overset{\sim}{\to} \mu_e\} \leftrightarrow \Hom((\Q/\Z)^*,G)
    \]
    given by
    \[
        (I,\gamma) \mapsto q_e^*(\gamma^{-1}),
    \]
    where $q_e:(\Q/\Z)^* \to \mu_e$ is the homomorphism defined by $a\mapsto a(1/e)$ and $e = |I|$.

    Moreover,
    \begin{enumerate}[(i)]
        \item This map descends to a bijection between $G\times G_k$-orbits $[I,\gamma]$ and $\RT_k^{\rm tame}(G)$.
        \item Let $L$ be the field fixed by the kernel of $\pi\in \Sur(G_k,G)$. If $\pk$ has ramification type $[I_\pk,\gamma]$ in $L/k$, then $[q_e^*(\gamma^{-1})]=[\iota_\pk^*\pi|_{I_\pk}]\in \RT_k^{\rm tame}(G)$.
    \end{enumerate}
\end{proposition}

\begin{proof}    
    We first show that the map is a well-defined bijection. This will essentially follow from properties of $q_e^*$, as it is clear by construction that $q_e$ is a homomorphism and is in fact $G_k$-equivariant.

    We first show that this map is well-defined. Given that $(\Q/\Z)^*$ is a procyclic group, we can partition the homomorphisms by their image
    \[
        \Hom((\Q/\Z)^*,G) = \coprod_{\substack{I\le G\\\text{cyclic}}} \Sur((\Q/\Z)^*,I).
    \]
    Setting $e = |I|$, the first isomorphism theorem implies that the pullback along $q_e$ induces a bijection between
    \[
        q_e^*:\Sur(\mu_e,I) \to \Sur((\Q/\Z)^*,I).
    \]
    Given that $|\mu_e| = e = |I|$, any element on the left-hand side must be an isomorphism. Thus,
    \[
        q_e^*(\gamma^{-1}) \in \Sur((\Q/\Z)^*,I) \subseteq \Hom((\Q/\Z)^*,G).
    \]
    This implies the map is well-defined.

    The map is bijective because we can demonstrate an inverse:
    \[
        f\leftrightarrow ( f((\Q/\Z)^*), ((q_e^*)^{-1}f)^{-1}),
    \]
    where we note that the fact that $(\Q/\Z)^*$ is procyclic and $|I|=e$ implies that $q_e^*$ is a bijection via the First Isomorphism Theorem.

    We now check that this bijection respects the action. We can check each component of the action separately. An element $g\in G$ acts on $(I,\gamma)$ by
    \[
        g.(I,\gamma) = (g I g^{-1}, g.\gamma).
    \]
    Consider that $(g.\gamma)(x) = \gamma(g^{-1}xg)$ implies that $(g.\gamma)^{-1}(\zeta) = g\gamma^{-1}(\zeta)g^{-1}$. Thus, for any $a\in (\Q/\Z)^*$, we evaluate
    \begin{align*}
        q_e^*((g.\gamma)^{-1})(a) &= (g.\gamma)^{-1}(q_e(a))\\
        &= g\gamma(q_e(a))^{-1} g^{-1}\\
        &= g q_e^*(\gamma^{-1})(a) g^{-1}.
    \end{align*}
    This is precisely the action on $\Hom((\Q/\Z)^*,G)$, $(g.f)(a) = g f(a) g^{-1}$, showing that the bijection is $G$-equivariant.

    Likewise, consider $\sigma \in G_k$ which acts as
    \[
        \sigma.(I,\gamma) = (I, \sigma \circ \gamma).
    \]
    Consider that
    \begin{align*}
        q_e^*((\sigma\circ\gamma)^{-1}) &= q_e^*(\gamma^{-1}\circ \sigma^{-1})\\
        &=\gamma^{-1}\circ \sigma^{-1} \circ q_e.
    \end{align*}
    Given that $q_e$ is $G_k$-equivariant implies that for each $a\in (\Q/\Z)^*$
    \begin{align*}
        q_e^*((\sigma\circ\gamma)^{-1})(a) &= \gamma^{-1}\circ q_e(\sigma^{-1}.a)\\
        &=q_e^*(\gamma^{-1})(\sigma^{-1}.a).
    \end{align*}
    This is exactly the $G_k$-action on $\Hom((\Q/\Z)^*,G)$, $(\sigma.f)(a) = f(\sigma^{-1}.a)$.

    Part (i) follows from equivariance, as it is stating a bijection between $G\times G_k$ orbits.
        
    We now consider part (ii). Suppose that $\pk$ has ramification type $[I,\gamma]$ in $L/K$ and let $\pi\in \Sur(G_k,G)$ correspond to $L$ under the Galois correspondence.

    Then by definition
    \begin{align*}
        \iota_{\pk}^*\pi_{I_\pk} &= \pi|_{I_\pk} \circ \iota_{\pk}\\
        &= \pi|_{I_{\pk}} \circ (q_e)_{(e,{\rm Nm}\pk)=1}.
    \end{align*}
    Choose a prime $\mathfrak{q}\mid \pk$, which is equivalent to fixing a homomorphism $I_{\mathfrak{p}}\hookrightarrow G_K$, where $I_\mathfrak{q}$ is what we label as the image of this map. Identifying $I_\pk$ with the representative image $I_\mathfrak{q}$, it now suffices to show that the diagram
    \[
        \begin{tikzcd}
            I_{\mathfrak{q}} \arrow[rr,"{\pi|_{I_\mathfrak{p}}}"]\drar &&I\\
            &I_\mathfrak{q}/I_\mathfrak{q}^e \urar\drar& \\
            (\Q/\Z)^* \arrow[rr,"q_e"]\arrow[uu,"{\iota_\pk}"] \urar &&\mu_e\arrow[uu,swap,"{\gamma^{-1}}"]
        \end{tikzcd}
    \]
     commutes, where $I_\mathfrak{q} \to I_\mathfrak{q}/I_\mathfrak{q}^e$ is the canonical quotient map, $I_\mathfrak{q}/I_\mathfrak{q}^e\to I$ is the isomorphism induced by $\pi|_{I_\pk}$, $I_\mathfrak{q}/I_\mathfrak{q}^e\to \mu_e$ is the canonical isomorphism sending $g\mapsto g(\pi_\mathfrak{q})/\pi_\mathfrak{q}$ for some uniformizer $\pi_\mathfrak{q}\in L_\mathfrak{q}$, and $(\Q/\Z)^*\to I_\mathfrak{q}/I_\mathfrak{q}^e$ is the composition of $q_e$ with the inverse of the canonical isomorphism from $I_\mathfrak{q}/I_\mathfrak{q}^e\to \mu_e$. By definition, the top, left, and bottom triangles all commute so it suffices to consider the right triangle.

     The composition $I_\mathfrak{q}/I_\mathfrak{q}^e \to I \to \mu_e$ sends $g\mapsto \beta_{\mathfrak{q}\mid \pk}(g) \mod \mathfrak{q}$. But, $\beta_{\mathfrak{q}\mid \pk} (g) = g(\pi_\mathfrak{q})/\pi_\mathfrak{q}$ by definition, so that this map agrees with the canonical isomorphism $I_\mathfrak{q}/I_\mathfrak{q}^e\to \mu_e$.
\end{proof}

\subsection{Smith's ramification-preserving homomorphism}
Smith defines
\[
    \mathcal{R}_{\pk,k}: H^1(G_k,G) \mapsto G(-1)^{G_{k_\pk}}
\]
by sending $\phi \mapsto \left(\overline{\zeta}\mapsto \phi({\rm Tine}_k(\pk))\right)$.

\begin{proposition}
    Let $k$ be a number field and $G$ a $G_k$-module. Then there is a canonical $G\times G_k$-equivariant isomorphism $\Hom((\Q/\Z)^*,G) \cong G(-1)$.

    Moreover,
    \begin{enumerate}[(i)]
        \item This map descends to a bijection between $\RT_k^{\rm tame}(G)$ and the $G_k$-orbits of $G(-1)$.
        \item The map $\mathcal{R}_{\pk,k}$ is independent of the choice of $\overline{\zeta}$.
        \item For each tame prime $\pk$, the diagram
        \[
        \begin{tikzcd}
            \Hom(G_k,G) \rar[equals]\dar & H^1(G_k,G) \rar{\mathcal{R}_{\pk,k}} & G(-1)\dar \\
            \Hom(I_\pk,G) \rar{(\iota_{\pk}^*)^{-1}} & \Hom((\Q/\Z)^*,G) \rar & \RT_k^{\rm tame}(G)
        \end{tikzcd}
        \]
        commutes.
    \end{enumerate}
\end{proposition}

Note that while $\mathcal{R}_{\pk,k}$ does not depend on the choice of $\overline{\zeta}$, the generator ${\rm Tine}_k(\pk)\in I_{\pk}^{\rm tame}$ does depend on this choice.

\begin{proof}
    The isomorphism is immediate from the canonical $G_k$-equivariant isomorphism $(\Q/\Z)^*\cong \hat{\Z}(1)$, as this implies
    \[
        G(-1) = \Hom(\hat{\Z}(1),G) \cong \Hom((\Q/\Z)^*,G).
    \]
    Part (i) is the bijection of orbits induced by an equivariant map.

    This follows from \Cref{lem:canonical}. Smith defined ${\rm Tine}_k(\pk)$ to be the element such that for each $\alpha\in \pk - \pk^2$,
    \[
        \frac{{\rm Tine}_k(\pk)(\alpha^{1/m})}{\alpha^{1/m}}
    \]
    equals the projection of $\overline{\zeta}$ to $\mu_m$. In particular, this implies the map $\widehat{\Z}(1) \mapsto I_{\pk}^{\rm tame}$ fits together with the composition in \Cref{lem:canonical} and the canonical isomorphism in \Cref{lem:Q/ZtoZ(1)} to create the commutative diagram
    \[
    \begin{tikzcd}
        (\Q/\Z)^* \rar{(q_e)}\drar[equals] &\displaystyle\lim_{\substack{\leftarrow\\(e,\pk)=1}}\mu_e \rar{\sim} &\displaystyle\lim_{\substack{\leftarrow\\(e,\pk)=1}}I_\pk/I_{\pk}^e \rar[equals] &I_{\pk}^{\rm tame}\\
        & \widehat{\Z}(1) \uar\arrow{urr}
    \end{tikzcd}
    \]
    where the far right diagonal map is the one sending $\overline{\zeta} \mapsto {\rm Tine}_k(\pk)$. Compressing this diagram using the composition of the top row $\iota_\pk$ implies
    \[
    \begin{tikzcd}
        (\Q/\Z)^* \rar{\iota_{\pk}} &I_{\pk}^{\rm tame}\\
        \widehat{\Z}(1)\uar[equals]\urar\\
    \end{tikzcd}
    \]
    commutes, where the = line represents the canonical isomorphism in \Cref{lem:Q/ZtoZ(1)} and the diagonal line is the map $\overline{\zeta}\mapsto {\rm Tine}_k(\pk)$. In particular, taking $\Hom(-,G)$ of the above diagram gives the commutative diagram
    \[
        \begin{tikzcd}
            & G(-1)\\
            \Hom(I_\pk,G) \rar{(\iota_{\pk}^*)^{-1}}\urar & \Hom((\Q/\Z)^*,G)\uar[equals]
        \end{tikzcd}
    \]
    where the unlabeled diagonal map is induced by $\overline{\zeta}\mapsto {\rm Tine}_k(\pk)$ and the = line on the right represents the canonical isomorphism. We get a larger commutative diagram by including the maps restriction to $I_\pk$ and $\mathcal{R}_{\pk,k}$:
    \[
        \begin{tikzcd}
            \Hom(G_k,G) \rar[equals]\drar[swap]{\res_{I_\pk}} & H^1(G_k,G) \rar{\mathcal{R}_{\pk,k}}\dar{\rm res_{I_\pk}} & G(-1)\\
            &\Hom(I_\pk,G) \rar{(\iota_{\pk}^*)^{-1}}\urar & \Hom((\Q/\Z)^*,G).\uar[equals] &
        \end{tikzcd}
    \]
    The bottom of the diagram is independent of the choice of $\overline{\zeta}$, so that the top row must be as well. This proves (ii).

    Part (iii) follows from this same diagram, after taking the $G_k$-orbits of each term in the righthand column as in part (i).
\end{proof}

\bibliographystyle{alpha}
\bibliography{BAreferencesV2020}

\end{document}